\documentclass[11pt]{amsart}
\usepackage[colorlinks=true]{hyperref}
\usepackage{amsthm,amsmath,amsxtra,amscd,amssymb,xypic, cleveref,xcolor,mathtools}
\usepackage[all]{xy}
\usepackage{float}

\usepackage{mathalpha,relsize}
\DeclareMathAlphabet{\mathdutchcal}{U}{dutchcal}{m}{n}
\SetMathAlphabet{\mathdutchcal}{bold}{U}{dutchcal}{b}{n}
\DeclareMathAlphabet{\mathdutchbcal}{U}{dutchcal}{b}{n}

\usepackage{mathrsfs,multirow}
\usepackage{marginnote}

\newcommand{\wt}{\widetilde}

\newcommand\ad{\mathrm{ad}}

\newcommand{\ra}{\rightarrow}
\newcommand{\hra}{\hookrightarrow}
\newcommand{\ol}{\overline}
\def\coch{\nu}

\def\dcal#1{\mathlarger{\mathdutchcal{#1}}}

\def\Sieg{\dcal{H}}
\newcommand\PROP[1]{($\star$)$_{#1}$ }

\def\FF{\mathbb{F}}

\def\dX{\dcal{X}}
\def\dY{\dcal{Y}}
\def\MT{\mathrm{MT}}
\newcommand{\CC}{{\mathbb{C}}}

\newcommand{\GG}{{\mathbb{G}}}
\newcommand{\HH}{{\mathbb{H}}}
\renewcommand{\AA}{{\mathbb{A}}}

\newcommand{\EE}{{\mathbb{E}}}
\newcommand{\PP}{{\mathbb{P}}}
\newcommand{\QQ}{{\mathbb{Q}}}
\newcommand{\RR}{{\mathbb{R}}}
\newcommand{\bS}{{\mathbb{S}}}
\newcommand{\ZZ}{{\mathbb{Z}}}
\newcommand{\LL}{{\mathbb{L}}}

\newcommand{\NN}{{\mathbb{N}}}
\newcommand{\DD}{{\mathbb{D}}}

\newcommand\dmax{d_{\mathrm{max}}}
\newcommand\mdsp{\mathrm{mdc_{sp}}}

\newcommand\XXX{\mathcal{V}}

\newcommand{\cB}{{\mathcal B}}

\newcommand{\cA}{{\mathcal A}}
\def\ASh{\dcal{A}}

\newcommand{\cM}{{\mathcal M}}
\newcommand{\cJ}{{\mathcal J}}

\def\rG{\mathrm{G}}
\def\rH{\mathrm{H}}
\def\rK{\mathrm{K}}
\def\rZ{\mathrm{Z}}
\def\rT{\mathrm{T}}

\newcommand{\op}{\operatorname}
\newcommand{\Sym}{\op{Sym}}
\newcommand{\Gal}{\op{Gal}}
\newcommand{\Sp}{\op{Sp}}

\newcommand{\GSp}{\op{GSp}}
\newcommand{\SL}{\op{SL}}
\newcommand{\SO}{\op{SO}}
\newcommand{\SU}{\op{SU}}
\newcommand{\GL}{\op{GL}}
\newcommand{\der}{\op{der}}

\newcommand{\Hom}{\op{Hom}}

\newcommand{\Aut}{\op{Aut}}
\newcommand{\Prod}{\op{Prod}}
\newcommand{\Sum}{\op{Sum}}
\newcommand{\Res}{\op{Res}}
\newcommand{\Ad}{\op{Ad}}

\newcommand\codim{\op{codim}}

\newcommand\KK{\mathbb{K}}

\newcommand\LGI{(LGI)}
\newcommand\naI{(naI)}
\newcommand\AnC{(An-C)}

\def\vs{\varsigma} \def\si{\sigma}

\def\GA{r}

\newcommand\Shi{S}

\def\Sh{\dcal{Sh}}
\newcommand\SD{\mathscr{S}} 

\newcommand\Mgct[1][g]{\cM_{#1}^{\rm ct}}
\newcommand\Agind[1][g]{\cA_{#1}^{\rm ind}}
\newcommand\dmc{\op{mdc}}
\newcommand\dmcg{\op{mdc_{vg}}}

\theoremstyle{plain}
\newtheorem{thm}{Theorem}[section]
\newtheorem*{mthm*}{Main Theorem}

\newtheorem{lm}[thm]{Lemma}
\newtheorem{prop}[thm]{Proposition}
\newtheorem{cor}[thm]{Corollary}
\newtheorem{fact}[thm]{Fact}
\newtheorem{qu}[thm]{Question}

\newtheorem{maintheorem}{Theorem}

\newenvironment{mthm}[1]
{\renewcommand{\themaintheorem}{#1}
\begin{maintheorem}}
{\end{maintheorem}}

\newtheorem{maincorollary}{Corollary}

\newenvironment{mcor}[1]
{\renewcommand{\themaincorollary}{#1}
\begin{maincorollary}}
{\end{maincorollary}}

\theoremstyle{definition}
\newtheorem{ex}[thm]{Example}
\newtheorem{dfx}[thm]{Definition}

\newtheorem*{conv*}{Convention}
\newtheorem*{notation}{Notation}

\newtheorem{rem}[thm]{Remark}

\theoremstyle{remark}
\newtheorem{claim}[thm]{Claim}

\begin{document}

\title[Compact subvarieties of $\cA_{\lowercase{g}}$]{Compact subvarieties of the moduli space of complex {A}belian varieties}

\author[S. Grushevsky]{Samuel Grushevsky}
\address{Department of Mathematics and Simons Center for Geometry and Physics, Stony Brook University, Stony Brook, NY 11794-3651}
\email{sam@math.stonybrook.edu}

\author[G. Mondello]{Gabriele Mondello}
\address{Dipartimento di Matematica ``Guido Castelnuovo'', Sapienza Universit\`a di Roma, Piazzale Aldo Moro 5 -- 00185 Roma, Italy}
\email{mondello@mat.uniroma1.it}

\author[R. Salvati Manni]{Riccardo Salvati Manni}
\address{Dipartimento di Matematica ``Guido Castelnuovo'', Sapienza Universit\`a di Roma, Piazzale Aldo Moro 5 -- 00185 Roma, Italy}
\email{salvati@mat.uniroma1.it}

\author[J. Tsimerman]{Jacob Tsimerman}
\address{Department of Mathematics, University of Toronto, Toronto, Canada}
\email{jacobt@math.toronto.edu}

\begin{abstract}
We determine the maximal dimension of compact subvarieties of 
the moduli space 
$\cA_{g}$
of complex principally polarized abelian varieties of dimension~$g$, and
we characterize all compact subvarieties of maximal dimension.
Moreover we determine the maximal dimension of a compact subvariety through a very general point of
$\cA_{g}$.
These results also allow us to draw some conclusions for compact subvarieties of the moduli space of complex curves of compact type.
\end{abstract}
\thanks{Research of the first author is supported in part by NSF grant DMS-21-01631. The second author is partially supported by INdAM GNSAGA research group,
the second and third authors are partially supported by 
PRIN 2022 research project ``Moduli spaces and special varieties''
and Sapienza research projects ``Algebraic and differential aspects of varieties and moduli spaces'' (2022),
``Aspects of modular forms, moduli problems, applications to L-values'' (2023),
``Global, local and infinitesimal aspects of moduli spaces'' (2024).}

\maketitle

\section{Introduction}\label{sec:intro}

\subsection{Compact subvarieties of maximal dimension}
Given a
complex
irreducible quasi-projective variety~$\XXX$, it is natural to ask how far is~$\XXX$ from being affine or projective. Perhaps one of the most intuitive ways to measure this is the {\em maximal dimension of a compact subvariety of~$\XXX$}, which we denote~$\dmc(\XXX)$.
On the other hand, if one takes $\XXX$ to be the blow-up
of an affine space at the origin,
such $\XXX$ has a compact exceptional divisor, though this large compact subvariety is somehow ``accidental''.
Thus, a perhaps more natural quantity, which in particular is invariant under 
proper birational morphisms,
is the {\em maximal dimension of an 
irreducible
compact subvariety of~$\XXX$ passing through a very general point of $\XXX$}, which we denote by $\dmcg(\XXX)$.

We observe that $0\le \dmcg(\XXX)\le\dmc(\XXX)\le \dim \XXX$.
Clearly $\dmcg(\XXX)=\dmc(\XXX)=0$ if $\XXX$ is affine, while the converse is in general not true. On the other hand, $\dmcg(\XXX)=\dmc(\XXX)=\dim \XXX $ if and only if~$\XXX$ is projective.

Note that $\dmcg(\XXX)$ can also be understood
as the maximum value of $d$ for which there exists
a family of $d$-dimensional 
irreducible
compact subvarieties of $\XXX$ whose total space dominates $\XXX$.

\begin{ex}
If we denote by $\dmc(\XXX,p)$ the maximal dimension
of an irreducible 
compact subvariety of $\XXX$ that contains the point $p\in\XXX$,
then the function $\XXX\ra\NN$,
defined as $p\mapsto \dmc(\XXX,p)$,  
is in general neither upper nor lower semi-continuous.
As an example,
    consider $\pi:\mathrm{Bl}_{\ell}(\CC^3)\ra \CC^3$ the blow-up
    of the complex affine space $\CC^3$ along a line $\ell\subset\CC^3$. Pick $q\in\ell$ and let $\XXX\coloneqq\mathrm{Bl}_{\ell}(\CC^3)\setminus\{\tilde{q}\}$, for some point $\tilde{q}\in\pi^{-1}(q)$.
    We claim that $\dmc(\XXX,p)$ takes value $1$
    on $\pi^{-1}(\ell\setminus\{q\})$ and $0$ elsewhere.
    Note that $\pi^{-1}(\ell\setminus\{q\})$ is neither open nor closed (even in the classical topology).
    
Indeed, since the only irreducible compact subvarieties of $\CC^3$
are points, any irreducible compact subvariety of $\XXX$
through the point $p$ must be contained in $\pi^{-1}(\pi(p))\cap\XXX$.

Then $\dmc(\XXX,p)=1$ if
$p$ belongs to $\pi^{-1}(\ell\setminus\{q\})$, as one can take the whole fiber $\pi^{-1}(\pi(p))\cong\PP^1$ as a compact subvariety. If $p\notin\pi^{-1}(\ell)$, then certainly $\dmc(\XXX,p)=0$, as $\pi^{-1}(\pi(p))=\{p\}$.
Finally if $p\in\pi^{-1}(q)\setminus\{\tilde q\}$, then 
$\pi^{-1}(q)\cap\XXX$  is isomorphic to $\CC$, and so an irreducible compact subvariety
through $p$ must consist only of $\{p\}$,
which again implies $\dmc(\XXX,p)=0$.
\end{ex}

\begin{rem}\label{rem:cover}
    If $f:\XXX'\to\XXX$ is a generically finite,
proper
map of 
irreducible
varieties, then $\dmcg(\XXX')=\dmcg(\XXX)$; and furthermore, if~$f$ is finite surjective, then also $\dmc(\XXX')=\dmc(\XXX)$.
\end{rem}

\begin{rem}\label{rem:cpt}
If $\overline \XXX\supsetneq \XXX$ is a projective compactification of~$\XXX$, such that the boundary $\partial \XXX\coloneqq\overline{\XXX}\setminus\XXX$ has codimension $d$, then by embedding $\overline \XXX\hra \PP^N$ and recursively choosing general hypersurfaces $H_1,\dots,H_{\dim\XXX-d+1}\subsetneq\PP^N$ of sufficiently high degree such that $\dim \left( H_1\cap\dots\cap H_i\cap  \partial \XXX\right)=(\dim\XXX-d)-i$ for any $i$, it follows that $H_1\cap\dots \cap H_{\dim\XXX-d+1}\cap\overline \XXX$ is a compact subvariety of $\XXX$, which moreover can be chosen to pass through any prescribed
finite collection of points of $\XXX$.
Thus the existence of such a compactification~$\overline \XXX$ implies $\dmcg(\XXX)\ge d-1$. We note, however, that in general there is no implication going the other way, as is easily seen by considering $\PP^m\times \CC^n\subsetneq \PP^{m+n}$ for various~$m,n$.
\end{rem}

In this paper we determine 
$\dmcg(\cA_{g})$ and $\dmc(\cA_{g})$ for the moduli space $\cA_{g}$
of complex principally polarized abelian varieties (ppav) of dimension $g$, and we draw some consequences for the moduli of complex curves of compact type, and for the locus of indecomposable abelian varieties. 

\begin{conv*}\label{conv:variety}
Throughout this paper, varieties 
are always assumed to be
reduced, but not necessarily irreducible. 
We will abuse notation and we will use the same symbol
for a reduced algebraic scheme over $\mathrm{Spec}(\CC)$ and for the associated complex variety.
\end{conv*}

\subsection{Compact subvarieties through a very general point of
$\cA_{g}$
}\label{sec:intro-general}

We recall that the moduli space of principally polarized complex
abelian varieties of dimension $g$
is a smooth Deligne-Mumford stack. Its underlying coarse moduli space 
$\cA_{g}$
is an irreducible quasi-projective variety of dimension $\frac{g(g+1)}{2}$.
Its Satake compactification 
$\cA_{g}^*$ of $\cA_{g}$
is an irreducible projective variety
and its boundary 
$\partial\cA^*_{g}=\cA^*_{g}\setminus\cA_{g}$ can be identified with $\cA^*_{g-1}$,
as in fact
it can be naturally stratified as
$\partial\cA^*_{g}\cong \cA_{g-1}\sqcup\cA_{g-2}\sqcup\dots\sqcup \cA_{0}$
(see \cite{satake-compact}).
Thus 
$\partial\cA^*_{g}$ has pure codimension $g$ in $\cA^*_{g}$, 
which by \Cref{rem:cpt} implies $\dmcg(\cA_{g})\ge g-1$.

Our first result is that this bound is sharp (as always, over $\CC$).

\begin{mthm}{A}\label{thm:dmcgAg}
The maximal dimension of a Hodge-generic compact subvariety of 
$\cA_{g}$
is $g-1$. In particular
$\dmcg(
\cA_{g}
)=g-1$.
\end{mthm}

The first claim in \Cref{thm:dmcgAg} is a particular, and slightly easier, case of \Cref{thm:dmcNCShimura}.
The second claim directly follows from the first one.
Indeed, as recalled in \Cref{sec:weakly} below, a subvariety of 
$\cA_{g}$
is {\it{Hodge-generic}} if it 
is not contained in one of the countably many {\it{special}} subvarieties of 
$\cA_{g}$,
different
from $\cA_{g}$ itself.

\subsection{Compact subvarieties of 
$\cA_{g}$
}\label{sec:intro-compact}
Turning to the question of maximal dimension of all compact subvarieties, the best known upper bound for 
$\cA_{g}$
is the 20-year-old result of Keel and Sadun \cite{kesa}, who proved that for any $g\geq 3$
a compact subvariety of 
$\cA_{g}$
has codimension
strictly larger than $g$, namely
$\dmc(
\cA_{g}
)\le \tfrac{g(g-1)}{2}-1$. This was conjectured by Oort, and underscored the difference between the moduli space $\cA_{g}$ of complex ppav (that we work with) 
and $\cA_{g,\mathbb{F}_p}$, as in finite characteristic there does exist a complete subvariety of $\cA_{g,\mathbb{F}_p}$ of codimension exactly~$g$,
namely the locus of ppav
whose subscheme of~$p$-torsion points is supported at zero.

We determine precisely the maximal dimension of compact subvarieties of
$\cA_{g}$, for all genera.

\begin{mthm}{B}\label{thm:dmcAg}
The maximal dimension of a compact subvariety of 
$\cA_{g}$
is
$$
 \dmc(
 \cA_{g}
 )=\max\left(g-1,\left\lfloor \frac{\lfloor g/2\rfloor ^2}{4}\right\rfloor\right)=\begin{cases}
          \ \ g-1\  & \mbox{if $g<16$}\, \\[4pt]
          \ \ \left\lfloor\tfrac{g^2}{16}\right\rfloor & \mbox{for even } g\ge 16\,\\[10pt]
          \left\lfloor\tfrac{(g-1)^2}{16}\right\rfloor & \mbox{for odd } g\ge 17\,.
        \end{cases}
$$
\end{mthm}
In \Cref{table:dimAg} we give our results and the bound of Keel-Sadun, for comparison, in some small genera, and also for $g=100$ to underscore the difference of the growth rates. 
\begin{table}[H]
$
\begin{array}{|lr||rrrr||rrrr||r|}
\hline & g = &3&4&5&6&15&16&17&18&100 \\
  \hline
  \text{Theorem A:} & \dmcg(
 \cA_{g}
 )= &2&3&4&5&14&15&16& 17& 99\\
  \text{Theorem B:} & \dmc(
 \cA_{g}
  )= &2&3&4&5&14&16& 16& 20 & 625\\
  \text{\em Keel-Sadun:} & \dmc(
 \cA_{g}
  )\le &2&5&9&14 &104   &119& 135& 152  &4949\\
 \hline
\end{array}
$

\bigskip

\caption{Maximal dimensions of compact subvarieties of
$\cA_{g}$.}
\label{table:dimAg}
\end{table}

What we actually prove is much stronger: we show that any compact subvariety of
 $\cA_{g}$
 is contained -- up to Hecke translations -- in the image of the product 
$\Shi'\times Z''\subset 
\cA_{g'}\times \cA_{g''}
$,
where 
$\Shi'$ is a compact special subvariety of 
$\cA_{g'}$
and $Z''$ is compact Hodge-generic 
in a non-compact special subvariety of 
$\cA_{g''}$,
with dimension $\dim Z''\le g''-1$, for some non-negative integers $g'+g''=g$
(see \Cref{prop:productShimura} and \Cref{thm:dmcNCShimura}).

We then determine the maximal dimension of compact special subvarieties of
 $\cA_{g}$.
In fact 
our analysis allows us to describe all maximal-dimensional 
irreducible compact subvarieties of 
 $\cA_{g}$;
 they are as follows:
\begin{itemize}
    \item[(a)]
Hodge-generic, for $g\leq 15$ and for $g=17$, or 
\item[(b)]
special subvarieties $\Shi_g\subsetneq
 \cA_{g}
 $ of a specific type (see \Cref{sec:best})
for $g=2$ or $g\geq 16$ even,
or 
\item[(c)]
weakly special subvarieties
that are irreducible components of Hecke translates of the
image of $[E]\times \Shi_{g-1}\subset 
\cA_{1}\times\cA_{g-1}
$ under the standard morphism
$\cA_{1}\times\cA_{g-1}\ra\cA_{g}$,
with
$\Shi_{g-1}\subsetneq
\cA_{g-1}
$ compact special as in subcase (b),
and with any fixed $[E]\in
\cA_{1}
$, for $g\geq 17$ odd. 
\end{itemize}

The existence of two essentially different types of maximal compact subvarieties is reflected in 
the two different growth regimes for $\dmc(
\cA_{g}
)$:
linear in the Hodge-generic case (a),
versus quadratic in the (weakly) special
cases (b) and (c). We give a more precise statement in \Cref{thm:maxvar}.

\begin{rem}
For any finite cover (e.g.~any finite abelian level cover)
$\cA'_{g}\ra\cA_{g}$,
we have
$\dmc(
 \cA_{g}
)=\dmc(
 \cA'_{g}
)$ and $\dmcg(
 \cA_{g}
)=\dmcg(
 \cA'_{g}
)$ by \Cref{rem:cover}. 
It 
also follows from that remark
that $\dmc$ and $\dmcg$ are preserved by finite (surjective) correspondences; as a consequence, our \Cref{thm:dmcgAg} and \Cref{thm:dmcAg}
apply to moduli of complex abelian varieties with any fixed polarization type and with any additional finite level structure. It also follows that, up to taking irreducible components, 
the collection of compact subvarieties of 
 $\cA_{g}$
that achieve
$\dmc(
 \cA_{g}
)$ or $\dmcg(
 \cA_{g}
)$ is stable under Hecke translations.
\end{rem}

\begin{rem}
    It is a natural interesting question to determine all maximal (as opposed to just maximal-dimensional) compact subvarieties of
 $\cA_{g}$,
 that is to determine 
all irreducible
compact subvarieties of 
 $\cA_{g}$
 that are not strictly contained in another 
irreducible
compact subvariety.
This investigation is related to understanding compact subvarieties of maximal dimension inside other non-compact Shimura varieties 
\end{rem}

\subsection{Compact subvarieties of $\Mgct$}\label{intro-Mg}
Recall that the Torelli map sending a smooth complex curve to its Jacobian is an injection 
$\cJ:\cM_{g}\hra \cA_{g}$
of the coarse moduli spaces (and is 2:1 onto its image as a map of stacks, but this does not matter for discussing compact subvarieties). 
Its image is disjoint from the {\it{decomposable locus}}
$\cA_{g}^{\mathrm{dec}}$,
namely the locus of ppav which are
isomorphic to the product of two smaller dimensional ppav,
which is naturally
the union
of the images of the standard morphisms
$\cA_{k}\times\cA_{g-k}\ra\cA_{g}$
for $0<k<g$.
A very interesting problem is to determine 
$\dmc(\Agind)$ and $\dmcg(\Agind)$
of
the locus 
$\Agind\coloneqq\cA_{g}\setminus\cA_{g}^{\mathrm{dec}}$
of indecomposable abelian varieties, 
see \Cref{rem:Agind}.

\begin{rem}
The homology classes of the images  of the standard maps
$\cA_{k}\times\cA_{g-k}\ra\cA_{g}$
for $0<k<g$ were recently studied in \cite{cmop,coptalk,iribarlopez}. It would be interesting to investigate how our approach, described
in \Cref{sec:introA}, relates to this study
(see also \Cref{rem:Agind}).
\end{rem}

\smallskip
The Torelli morphism extends to a proper morphism 
$\cJ:\Mgct\to\cA_{g}$
from the moduli space of curves of compact type, that is from the moduli space of stable nodal curves such that each node is separating. As a corollary of their bound 
$\dmc(\cA_{3})\le 2$ and of a theorem of Diaz~\cite{diaz} stating that $\dmc(\cM_{g})\le g-2$, Keel and Sadun \cite{kesa} deduce the bound $\dmc(\Mgct)\le 2g-4$ for any $g\ge 3$. Our results also have implications for $\Mgct$.

\begin{mcor}{C}\label{cor:dmcMgct}
The following equality and upper bound hold:
$$
\begin{aligned}
\dmc(\Mgct)&=\lfloor 3g/2\rfloor-2 &\hbox{for any $2\le g\le 23$\,,}\\
\dmc(\cJ(\Mgct))&\le g-1 &\hbox{for any $2\le g\le 15$\,.}
 \end{aligned}
$$
Moreover $\dmc(\cJ(\Mgct))\geq \lfloor \frac{2g}{3}\rfloor$ for all $g\geq 2$.

\end{mcor}
The point of the difference between these two quantities is that $\cJ$ is injective at a generic point of $\Mgct$,
but not along 
$\partial\Mgct=\Mgct\setminus\cM_{g}$,
as it sends a stable curve to the product of the Jacobians of its irreducible components, forgetting the locations of the nodes. 

To obtain the first equality, in \Cref{lm:compact-Mgct} we construct for any~$g\ge 3$ a compact subvariety of 
$\Mgct\setminus\cM_{g}$
of dimension $\lfloor 3g/2\rfloor-2$, starting from a compact curve in 
$\Mgct[2]$
(which thus gives $\dmc(\Mgct)\ge \lfloor 3g/2\rfloor -2$ for any~$g\geq 2$).

The second inequality 
in the corollary
implies the bound on the maximal dimension of compact subvarieties of 
$\Mgct$
that intersect 
$\cM_{g}$,
and follows from \Cref{thm:dmcAg} (in fact \Cref{thm:dmcAg} also improves Keel and Sadun's bound $\dmc(\Mgct)\le 2g-4$ to $\dmc(\Mgct)\le \lfloor \lfloor g/2\rfloor^2/4\rfloor$ for $g\le 28$). 

It is tempting to conjecture that in fact $\dmc(\cJ(\Mgct))\le g-1$ for all $g\ge 2$, which would then imply $\dmc(\Mgct)=\lfloor 3g/2\rfloor-2 $ for all $g\ge 2$. Notice that Krichever~\cite{krcomplete} claimed exactly such a bound, but unfortunately his proof had a gap that was never fixed.

\subsection{Idea of the proof of \Cref{thm:dmcgAg}}\label{sec:introA}
We observed at the beginning of \Cref{sec:intro-general} that 
$\dmcg(
\cA_{g}
)\ge g-1$.
Thus it is enough to show that $\dmcg(
\cA_{g}
)<g$.

The inspiration for the argument is the following. 

Let 
$\iota:\cA_{1}\times\cA_{g-1}\ra\cA_{g}$
be the natural product map 
$\iota([E],[A])=[E\times A]$, and let 
$\pi:\cA_{1}\times\cA_{g-1}\ra \cA_{1}$
be the projection onto the first factor.
Observe that $\iota$ is proper and that
the fibers 
$\{[E]\}\times\cA_{g-1}$
of $\pi$
have codimension $g$ in 
$\cA_{g}$.

The essential but very simple remark is that,
if $Z\subsetneq
\cA_{g}
$ is a compact subvariety,
then the properness of $\iota$ implies
that $\iota^{-1}(Z)$ is compact
and so $\pi(\iota^{-1}(Z))$ is a compact subset
of 
$\cA_{1}$,
which thus consists of finitely many points.

Assume now, for contradiction, that $Z\subsetneq
\cA_{g}
$ is a compact subvariety of dimension at least $g$. 
If we could ensure that, for at least one $[E]\in
\cA_{1}
$, the following property holds 
\begin{itemize}
    \item[\PROP{E}]
the intersection of $\pi^{-1}([E])=\{[E]\}\times
\cA_{g-1}
$ with~$Z$ is nonempty and has expected dimension $\dim Z-g\ge 0$, 
\end{itemize}
then property \PROP{E} would hold for all $[E]$ lying in a Zariski open subset of
$\cA_{1}$.
This would mean that $\pi(\iota^{-1}(Z))$ is a nonempty Zariski open
subset of 
$\cA_{1}$,
contradicting the above remark.

While \PROP{E} may not hold for any $[E]\in
\cA_{1}
$ as stated and without further assumptions on $Z$, 
this idea essentially works for Hodge-generic subvarieties $Z$, once the loci of products $\{[E]\}\times
\cA_{g-1}
$ are replaced by their Hecke translates, i.e.~by loci of ppav that are isogenous to such (i.e.~by loci of abelian varieties that surject onto~$E$).

An essential tool we will use for this is 
\Cref{cor:KU}, in which we borrow an idea from
\cite[Theorem 1.9(i)]{khur}, and whose proof relies on the Ax-Schanuel conjecture, proven for
$\cA_{g}$
by 
Pila and the fourth author \cite{pila-tsimerman},
and proven for an arbitrary Shimura variety by Mok, Pila and the fourth author \cite{MPT}.
The details of the proof of \Cref{thm:dmcgAg} can be found in \Cref{sec:reduction}.

\subsection{Structure of the proof of \Cref{thm:dmcAg}}\label{sec:introB}

Here we describe the structure of the dimension
estimates needed to prove \Cref{thm:dmcAg}, which also exploits induction on $g$.

Let $Z\subsetneq
\cA_{g}
$ be an irreducible compact subvariety of largest dimension and let $\Shi\subseteq
\cA_{g}
$ be the smallest
special subvariety that contains $Z$ (see \Cref{thm:specialclosure}).

We distinguish a few cases, according to the compactness
of $\Shi$ or its indecomposability (in the sense of \Cref{dfx:indecomposable}).

\begin{itemize}
\item[(o)]
Suppose that {\it{the tautological family of ppav over $\Shi$ is a product}}
of lower-dimensional ones.

Then $\Shi$ is the image 
under the standard map 
$\cA_{g'}\times\cA_{g-g'}\ra\cA_{g}$
of a special subvariety of some 
$\cA_{g'}\times\cA_{g-g'}$,
and so $\dim Z\le \dmc(
\cA_{g'}
)+\dmc(
\cA_{g-g'}
)$ can be bounded inductively on $g$.
\item[(o')]
Suppose that {\it{the tautological family of ppav over $\Shi$ is isogenous to a product}}
of lower-dimensional families.

Then, for some large $N$, the subvariety $\Shi$ is the image via a morphism 
$\Phi:\cA_{g'}(N)\times\cA_{g-g'}(N)\ra\cA_{g}$
of some special subvariety and the same estimates as in (o) apply.
\item[(i)]
Suppose that {\it{$\Shi$ is non-compact and 
the universal family of ppav over $\Shi$ is not isogenous to a product.
}}

Then we prove that $\dim Z\leq g-1$ 
in \Cref{thm:dmcNCShimura} (which uses Ax-Schanuel for 
Shimura varieties associated to
non-compact special subvarieties of 
$\cA_{g}$
).
\end{itemize}

Note that, if $S$ is not compact,
then the above (o') and (i)
imply that for some large $N$ we have
$Z=\Phi(\Shi'\times Z'')$, where $\Shi'$ is a compact special subvariety of 
$\cA_{g'}$
and $Z''$ is compact Hodge-generic
in a non-compact special subvariety $\Shi''\subseteq
\cA_{g-g'}
$.
Since $\dim Z''=g-g'-1$ by
\Cref{thm:dmcNCShimura} and by maximality, we can write an expression
for $\dmc(
\cA_{g}
)$ in terms
 of the maximal dimension of {\it{compact}} special
 subvarieties of 
 $\cA_{g'}$
 for all $g'\leq g$
 (see \Cref{sec:mdsp}).

Thus, in order to complete the proof,
we are only
left to determine the maximal dimension of a
compact special subvariety of 
$\cA_{g}$,
assuming that this is at least $g-1$.
Note that cases (ii) and (iii) below are independent
of (o), (o') and (i) above.
\begin{itemize}
    \item[(ii)]
    Suppose that {\it{$\Shi$ is compact, and its associated
    symplectic representation is decoupled}}
    (in the sense of \Cref{dfx:decoupled}). In this case $Z$ must be equal to $\Shi$.
    In \Cref{sec:decoupled} we determine those
    $\Shi$ of largest dimension (assuming that
    this is at least $g-1$), and in particular we determine those $\Shi$ 
    that achieve the optimal bound $\dim \Shi=\dmc(
    \cA_{g}
    )$ for $g=2$ and $g\geq 16$.
    Our analysis heavily relies on the fundamental works of Satake
    \cite{satake-compact,satake:analytic}
    recalled in \Cref{sec:Satake}.
Another important ingredient of our estimates
are the Hasse principles
recalled in \Cref{sec:compact}.    
    \item[(iii)]
    Suppose that {\it{$\Shi$ is compact and its associated
    symplectic representation is not decoupled}}.
    In \Cref{sec:non-decoupled} we show through
    numerical estimates that the compact subvarieties $\Shi$,
    of dimension at least $g-1$, appearing in this case have strictly smaller dimension than
    those that appear in the decoupled case (ii) above.
\end{itemize}

Finally, in \Cref{thm:maxvar} we collect
the list of all maximal dimensional compact
subvarieties of 
$\cA_{g}$
for all $g$.

\subsection{Structure of the paper}\label{sec:intro-structure}

In \Cref{sec:Shimura} we recall basic facts on algebraic groups, Shimura varieties, special and weakly special
subvarieties.

In \Cref{sec:special-Ag} we recall
the definition of Siegel Shimura datum 
and define indecomposable special subvarieties of
$\cA_{g}$.
Then we
discuss some properties of
symplectic representations
associated to special subvarieties of
$\cA_{g}$
and finally in \Cref{prop:productShimura} we prove a decomposability result
for embeddings of Shimura data isogenous to products
into Siegel Shimura datum.

In \Cref{sec:ax-schanuel} we recall the statement of Ax-Schanuel for Shimura varieties,
and we use it to prove \Cref{thm:dmcNCShimura}, which is a more general version of \Cref{thm:dmcgAg}, 
as it
deals with Hodge-generic compact subvarieties of any non-compact special subvariety of
$\cA_{g}$,
and not only with Hodge-generic compact subvarieties of
$\cA_{g}$
itself.
This permits us to give a first estimate for the maximal dimension of a compact subvariety of
$\cA_{g}$
(see \Cref{thm:finalestimate}).

In \Cref{sec:compact} we recall some Hasse principles
that will be necessary in the estimates
of \Cref{sec:decoupled}.

In \Cref{sec:Satake} we recall Satake's
classification of special subvarieties of
$\cA_{g}$
associated to decoupled symplectic representations,
that will be used in \Cref{sec:decoupled}.

In \Cref{sec:decoupled},  we investigate compact special subvarieties of
$\cA_{g}$
that are induced by decoupled symplectic representations, and determine which of them have dimension at least $g-1$
(see \Cref{thm:dmcAg2}). 

In \Cref{sec:non-decoupled} we investigate compact special subvarieties of
$\cA_{g}$
that are induced by symplectic representations that are not decoupled. 
We show that such subvarieties have dimension either smaller than~$g-1$,
or smaller than some Shimura variety originating from a decoupled representation as in \Cref{sec:decoupled}.
We thus conclude the proof of \Cref{thm:dmcAg}.

In \Cref{sec:jacobians} we discuss related problems for the moduli space of indecomposable
abelian varieties, and we prove \Cref{cor:dmcMgct} for compact subvarieties of the moduli
space 
$\Mgct$
of curves of compact type.

\subsection{Acknowledgments}
The first author is grateful to Universit\`a Roma ``La Sapienza'' for hospitality in June and October 2023, when part of this work was done. 
We are grateful to Sorin Popescu who endowed the Stony Brook lectures in Algebraic Geometry, which allowed the first and fourth author to meet and think about these topics. We are very grateful to Salim Tayou and Nicolas Tholozan for enlightening discussions and pointers to the literature on these topics.

\section{Algebraic groups and Shimura varieties}\label{sec:Shimura}

In this section we recall some useful notions on algebraic groups, Shimura data, 
and Shimura varieties. We refer to 
\cite{deligne-travaux}, \cite{deligne}, \cite{platonovrapinchuk},
\cite{moonen1998linearityI},
\cite{milne2005}, 
\cite{milne2011shimura}, \cite{milne:ag}, 
\cite{moonen-oort}, \cite{lanexample} for all information relevant to the current paper.

\begin{conv*}
All fields $\KK$ will be of characteristic $0$ and contained in the field $\CC$ of complex numbers.
By {\it{an algebraic group over $\KK$}} we mean an algebraic group scheme $\rG$
over a number field $\KK$ and we denote by $\rG(\KK)$ the set of its $\KK$-points. 
By a {\it{linear representation of $\rG$}} we mean
a morphism of group schemes $\rG\rightarrow \GL(V)$, where $V$ is a $\KK$-vector space.

If $\si:\KK\hra\LL$ is a field embedding, we denote by 
$\rG_{\si,\LL}=\rG\otimes_{\KK,\si}\LL$ the group scheme over $\mathrm{Spec}(\LL)$
obtained from $\rG$ by pull-back via $\si^*:\mathrm{Spec}(\LL)\rightarrow\mathrm{Spec}(\KK)$.
When the embedding $\si$ is understood, we simply write $\rG_{\LL}$.
Similarly, if $V$ is a $\KK$-linear representation of $\rG$,
we denote by $V_{\si,\LL}\coloneqq V\otimes_{\KK,\sigma}\LL$ the induced
representation of $\rG_{\si,\LL}$, or just $V_\LL\coloneqq V\otimes_\KK \LL$ when
$\si$ is understood.
\end{conv*}

\subsection{Algebraic groups}\label{sec:groups}
Let $\rG$ be an algebraic group over a number field $\KK\subset\CC$.
We will assume that $\rG$ is connected (in the Zariski topology),
which is equivalent to requiring that $\rG(\CC)$ is connected in the classical topology.

{\bf{Semisimple and simple groups.}}
The group $\rG$ is {\em semisimple} if every 
connected solvable normal subgroup of $\rG$ is trivial; it is {\em simple} if it is semisimple and every connected normal subgroup of $\rG$, different from $\rG$ itself, is trivial. Moreover, $\rG$ is {\em geometrically simple}  if $\rG_{\CC}$ is simple. 

{\bf{Reductive groups.}}
A group $\rG$ is {\em reductive} if 
every connected unipotent normal subgroup of $G$ is trivial.
The connected component of its center $\rZ(\rG)$ that contains the identity is a torus,
and $\rG$ is semisimple if and only if $\rZ(\rG)$ is finite. 
The {\it{adjoint group}} is defined as $\rG^{\ad}\coloneqq \rG/\rZ(\rG)$
and the {\it{derived subgroup}} $\rG^{\der}$ is the closed subgroup of $\rG$ generated by the commutators.
Then $\rG^{\ad}$ (resp.~$\rG^{\der}$) is semisimple
with trivial (resp.~finite) center, and we say that a connected reductive group $\rG$ is {\em adjoint} if it is semisimple and $\rG=\rG^{\ad}$. 

{\bf{Anisotropic groups.}}
A {\em split torus} $\rT$ is a product of a number of copies of the multiplicative group $\mathbb{G}_m$. A reductive group $\rG$ is called {\em anisotropic} if 
$\rG^{\ad}$ contains no positive-dimensional split tori.
It was shown in \cite[Corollary 8.5]{borel-tits} that a semisimple group over a number field $\KK$ is anisotropic if and only if it contains no nontrivial unipotent elements.
It follows that $\rG$ is anisotropic (as an algebraic $\KK$-group) if and only if $\Res_{\KK/\QQ}\rG$ is (as an algebraic $\QQ$-group).

{\bf{Isogenies.}}
An {\em isogeny} is a surjective homomorphism of connected algebraic groups $\rG'\rightarrow \rG$ with finite kernel; the kernel is then necessarily contained in~$\rZ(\rG')$.
An algebraic group $\rG$ is {\em simply connected} if every isogeny $\rG'\rightarrow \rG$ is an isomorphism.
Note that a reductive and simply connected algebraic group is necessarily semisimple.
For a reductive group $\rG$, let $\wt{\rG}$ denote the universal cover of $\rG^{\ad}$, so that $\wt{\rG}\ra \rG^\ad$ is an isogeny.

{\bf{Simple factors.}}
Let  $\rG$ be semisimple and let~$\rG_{(1)},\dots,\rG_{(\ell)}$ be all the simple closed normal $\KK$-subgroups of $\rG$. Then the product map $\rG_{(1)}\times\dots\times \rG_{(\ell)}\rightarrow \rG$ is an isogeny (see \cite[Theorem 22.121]{milne:ag}),
and we will call $\rG_{(j)}$ a {\it{simple $\KK$-factor}} of $\rG$.
If $\KK=\RR$, then we will say that a real factor $\rG_{(j)}$ is {\it{compact}}
if $\rG_{(j)}(\RR)$ is.

{\bf{Field of definition and splitting field.}}
Let $\rG$ be a semisimple group over $\KK$.
The Galois group $\Gal(\CC/\KK)$ acts on $\rG_{\CC}$ and permutes its simple $\CC$-factors
$\rG'_i$.  
The 
subgroup that sends $\rG'_i$ to itself
can be identified with a finite extension $\LL$ of $\KK$, called the {\it{field of definition
of $\rG'_i$}}; indeed, there exists a $\LL$-group $\rH'$ such that $\rG'_i=\rH'_\CC$.
The {\it{splitting field of $\rG$}} is the subfield of $\CC$ generated by
the fields of definition of all the simple factors of $\rG_\CC$.

{\bf{Restriction of scalars.}}
Let $\rH$ is a geometrically simple group over a totally real number field $\KK$
and $\rG=\rH_{\KK/\QQ}$ is the $\QQ$-group obtained from $\rH$ by {\it{restriction of scalars}}.
Then $\prod_{\si} \rH_{\si,\RR}\rightarrow \rG_\RR$ is an 
isomorphism, 
where $\si$ ranges over all embeddings $\KK\hra \RR$ and $\rH_{\si,\RR}\coloneqq \rH\otimes_{\KK,\si}\RR$.

\begin{dfx}\label{dfx:d(G)}
For $\rG$ a real reductive group with $\rG^{\ad}$ having maximal compact subgroup $\rK$, we define $d(\rG)\coloneqq\frac{1}{2}\left( \dim \rG^{\ad} - \dim \rK \right)$.
\end{dfx}

{\bf{Identity component of real groups.}}
If $\rG$ is a reductive group, we denote by 
$\rG^{\ad}(\RR)^+$ 
the connected component (in the classical topology) of $\rG^{\ad}(\RR)$ that contains the identity, and 
by $\rG(\RR)^+$ the preimage of $\rG^{\ad}(\RR)^+$ under $\rG(\RR)\ra \rG(\RR)^{\ad}$. Moreover,
we denote by $\rG(\QQ)^+$ the intersection
$\rG(\QQ)\cap \rG(\RR)^+$ within $\rG(\RR)$.

{\bf{Galois conjugate representations.}}
Let $\rG$ be an algebraic group over $\KK$ and let $U$ be a representation of  $\rG_{\CC}$.
For every $\vs\in \Gal(\CC/\KK)$, we define another representation $U^{\vs}$ of $\rG_{\CC}$, which will be called {\it{Galois conjugate to $U$}},
by setting
$U^{\vs}\coloneqq U\otimes_{\CC,\vs}\CC$ (which means that $\vs^{-1}(z)u\otimes 1 =u\otimes z$)  and then defining the $\rG_{\CC}$-action on $U^\vs$ as $g\cdot (u\otimes z)\coloneqq (g\cdot u)\otimes z$.

Note that, if $U$ is induced by a $\KK$-representation $W$ of $\rG$, then 
\[
\xymatrix@R=0in{
U=W\otimes_\KK \CC \ar[rr] && U^\vs=(W\otimes_\KK \CC)\otimes_{\CC,\vs}\CC\\
 w\otimes z\ar@{|->}[rr] && (w\otimes 1)\otimes z 
}
\]
is an isomorphism of $\rG_{\CC}$-representations.

{\bf{Galois action and representations defined over $\KK$.}}
Let $\rG$ be a reductive algebraic group over $\KK$ and let
$V$ be a $\KK$-representation of $\rG$.
Let also $U$
be an irreducible 
$\wt{\rG}_{\CC}$-invariant subspace
of $V_{\CC}=V\otimes_{\KK}\CC$.

Consider an element $\vs\in\Gal(\CC/\KK)$.
It acts both
on $\wt{\rG}_{\CC}$ and on $V_{\CC}$
in such a way that $\vs(g)\cdot (v\otimes\vs(z))=
\vs(g\cdot (v\otimes z))$
for $g\in\wt{\rG}_{\CC}$ and $v\otimes z\in V_{\CC}$.
In particular, 
$\vs(U)$ is a $\wt{\rG}_{\CC}$-subrepresentation
of $V_{\CC}$ and
the isomorphism
$s_\vs:U^\si\ra \vs(U)$ of $\CC$-vector spaces,
defined as $s_\vs(u\otimes z)\coloneqq z\vs(u)$,
induces the following commutative diagram
\[
\xymatrix{
\wt{\rG}_{\CC}\times U^\vs\ar[r] \ar[d]_{(\vs,s_\vs)} & U^\vs\ar[d]^{s_\vs}\\
\wt{\rG}_{\CC}\times \vs(U)\ar[r] & \vs(U)
}
\]
Now, $\wt{\rG}_{\CC}$
splits as a product $\wt{\rG}_{1,\CC}\times\dots\times\wt{\rG}_{k,\CC}$ of simple factors
and $U$ is isomorphic to $U_1\otimes\dots\otimes U_k$,
with $U_i$ an irreducible representation of $\wt{\rG}_{i,\CC}$.
Similarly, $U'\coloneqq\vs(U)\cong U'_1\otimes\dots\otimes U'_k$.

Observe that $U^\vs\cong U_1^\vs\otimes\dots\otimes U_k^\vs$ as $\wt{\rG}_{\CC}$-representations,
and that this is isomorphic to $U'_1\otimes\dots\otimes U'_k$
when the action of $\wt{\rG}_{\CC}$ on $U'$ is pre-composed with $\vs$.

Since $\vs$ acts on $\wt{\rG}_{\CC}$ by permuting its simple factors, we denote by $\wt{\rG}_{\vs(i),\CC}$ the factor $\vs(\wt{\rG}_{i,\CC})$. Thus 
$U_i^\vs\cong U'_{\vs(i)}$ as $\wt{\rG}_{i,\CC}$-representations,
where the action of $\wt{\rG}_{i,\CC}$ on $U'_{\vs(i)}$ is
pre-composed with $\vs$.
In particular, $U_i$ and $U'_{\vs(i)}$ have the same dimension,
and $U_i$ symplectic (resp.~orthogonal, or not self-dual)
if and only if $U'_{\vs(i)}$ is.

\subsection{Shimura data}\label{ssc:ShimuraData}
We follow \cite{milne2005} in our definitions of Shimura data and Shimura varieties.

We recall that the {\em{Deligne torus}} is the $\RR$-algebraic group defined as 
$\bS\coloneqq \Res_{\CC/\RR}(\GG_{m,\CC})$.
Note that $\bS(\RR)\cong\CC^\times$
and that
$\bS_\CC\cong \GG_{m,\CC}\times\GG_{m,\CC}$.

\begin{dfx}[{\cite[4.4]{milne2005}}]\label{def:shimura}
A {\em Shimura datum} is a pair $(\rG,\dX)$
where $\rG$ is a reductive algebraic group over $\QQ$,
and $\dX$ is a 
$\rG(\RR)$-conjugacy class
of homomorphisms of real algebraic groups $x:\bS\ra \rG_{\RR}$ 
such that
\begin{itemize}
\item[(i)]
the induced Hodge structure on the Lie algebra 
has
weights
$$\{(-1,1),\ (0,0),\ (1,-1)\},$$ 
i.e.~there is a direct sum decomposition $\mathfrak{g}_{\CC}=\mathfrak{s}\oplus\mathfrak{t}^+\oplus\mathfrak{t}^-$ of the Lie algebra of $\rG_\CC$, such that
$\Ad_{x(z)}$ acts trivially on $\mathfrak{s}$, as $\frac{z}{\bar z}$ on $\mathfrak{t}^+$
and as $\frac{\bar z}{z}$ on $\mathfrak{t}^-$;
\item[(ii)]
the conjugation by $\theta\coloneqq u(i)$ is a Cartan involution of $G^{\ad}_{\RR}$, 
namely $\rG^\theta\coloneqq \{g\in 
\rG^{\ad}(\CC)
\,|\, \theta\bar{g}\theta^{-1}=g\}$ is compact;
\item[(iii)]
$\rG^{\ad}$ has no nontrivial simple $\QQ$-factor $\rG'$
with 
$\rG'(\RR)$
compact.
\end{itemize}
The Shimura datum $(\rG,\dX)$ is said to be {\em{adjoint}} (resp.~{\em{semisimple}}) if $\rG$ is.

A {\em morphism of  Shimura data} $(\rG,\dX)\rightarrow (\rG',\dX')$
is a homomorphism $f:\rG\rightarrow \rG'$ of algebraic groups inducing a map $\dX\ra \dX'$.
Such morphism is said to be an {\em{embedding}} if $f$ is.
\end{dfx}

The {\em{product}} of two Shimura data $(\rG,\dX),(\rG',\dX')$
is the Shimura datum defined as $(\rG\times \rG', \dX\times \dX')$, where $\dX\times \dX'$ is naturally
identified with the 
$\rG(\RR)\times \rG'(\RR)$-conjugacy class
of
homomorphisms $(x,x'):\bS\ra \rG_\RR\times \rG'_\RR$ with $x\in \dX$ and $x'\in \dX'$.

\begin{rem}
Given a Shimura datum $(\rG,\dX)$ as in \Cref{def:shimura},
one can define another Shimura datum $(\rG^{\ad},\dX^{\ad})$ by composing the map $u:\bS\ra \rG_\RR$ with the projection 
$\rG_\RR\ra \rG_\RR^{\ad}$.
Then the natural map $\dX\ra \dX^{\ad}$ 
is finite surjective, and 
maps each component of $\dX$ isomorphically onto a component of $\dX^{\ad}$.
\end{rem}

If $(\rG,\dX)$ is a Shimura datum, then
$\dX$ acquires the structure of 
a disjoint union of finitely many Hermitian symmetric spaces,
that are pairwise conjugate to each other.

\begin{dfx}
The {\em dimension} of a Shimura datum $(\rG,\dX)$ is $\dim(\rG,\dX)\coloneqq\dim_\CC \dX$. 
\end{dfx}

Note that $\dim(\rG,\dX)=d(\rG)$, as introduced in \Cref{dfx:d(G)}.\\

{\bf{Shimura data and simple factors.}}
Let  $(\rG,\dX)$ be a Shimura datum and let $\dX^+\subseteq \dX$ be a connected component. Then $\dX^+$ is a connected Hermitian symmetric space,
and 
$\rG^{\ad}(\RR)^+$ 
acts via conjugation on $\dX^+$ transitively, isometrically, and holomorphically. 
Pick a point $x\in \dX^+$, so that
the group 
$\bS(\RR)=\CC^\times$ 
acts on $\dX^+$ via $\Ad_{\rG}\circ x$, thus fixing $x$. 
If 
$\theta=x(i)\in \rG(\RR)$, 
then the compact subgroup 
$\rG^\theta\cap \rG^{\ad}(\RR)^+$ 
is the identity component of the stabilizer of $x$ inside 
$\rG^{\ad}(\RR)^+$.

Recall that $\rG^{\ad}_{\RR}$ is the product of its simple closed normal subgroups. 
Let 
$\rG^{\ad}_c$ be the product of all compact simple factors of $\rG^{\ad}_\RR$, and $\rG^{\ad}_{nc}=\rG^{\ad}_1\times\dots \times \rG^{\ad}_k$ the product of all its non-compact
simple
factors. Note that we have natural factorizations  $$\dX=\dX_1\times\dots\times \dX_k\qquad\text{and}\qquad \dX^+=\dX^+_1\times\dots\times \dX^+_k.$$

We will often use the following result, whose proof can also be found
in {\cite[Theorem 3.13]{milne2011shimura}}.

\begin{lm}\label{lem:Shimgeomsimplefactors}
   Let $(\rG,\dX)$ be a Shimura datum, and $\rG^{\ad}_i$ an $\RR$-simple
factor of $\rG^{\ad}_{\RR}$. Then $\rG^{\ad}_i$ is geometrically simple.
Hence the splitting field of $\rG$ is 
a totally real number field.
\end{lm} 

\begin{proof}
By contradiction,
suppose $\rG^{\ad}_i$ were not geometrically simple. Since $\rG^{\ad}_{i,\CC}$ is adjoint and not simple, it can be written as a product of simple adjoint groups. Since $\Gal(\CC/\RR)$ acts transitively on the factors,  $\rG^{\ad}_{i,\CC}\cong \rH\times\ol{\rH}=(\Res_{\CC/\RR}\rH)_\CC$, where $\rH$ is a complex simple group
and $\rG^{\ad}_i\cong \Res_{\CC/\RR}\rH$.
Thus $\rG^{\ad}_i$ is non-compact. Moreover
$\rG_{i,\CC}$ has no Cartan involution by \Cref{H_C} below.
This contradicts \Cref{def:shimura}(iii),
thus proving the first claim.
The second claim immediately follows.
\end{proof}
 
\begin{lm}[{\cite[4.4.10]{simpson1992higgs}}]\label{H_C}
Let $\rH$ be a complex group of positive dimension, and let $ \rH_{\CC/\RR}$  be the real group obtained from $\rH$ by the restriction of scalars.
Then there is no element $\gamma$ 
in $\rH_{\CC/\RR}(\CC)$ 
such that
$\Ad_\gamma$ is a Cartan involution.
\end{lm}

The following well-known corollary will be essential for us:

\begin{cor}\label{cor:totallyrealsplit}
Let $(\rG,\dX)$ be a Shimura datum 
with $\rG$ adjoint (or simply-connected) and $\QQ$-simple. 
Then $\rG\cong \rH_{\KK/\QQ}$ for some totally real field~$\KK$ and some geometrically simple group~$\rH$ over~$\KK$.
\end{cor}

\begin{proof}
    We can write $\rG_\RR=\prod_{i=1}^k \rG_i$, where each $\rG_i$ is geometrically simple
    by \Cref{lem:Shimgeomsimplefactors}.
    Let $\KK$ be the field of definition of $\rG_1$, so that $\rG_1=\rH_\RR$ for some $\KK$-group $\rH$.
    It follows 
from \Cref{lem:Shimgeomsimplefactors}
that $\KK$ is totally real.
    Since $\rG$ is $\QQ$-simple, we obtain $\rG\cong \rH_{\KK/\QQ}$ as desired.
\end{proof}

\begin{rem}\label{rem:galois}
By \Cref{lem:Shimgeomsimplefactors}
the splitting field of the $\QQ$-group $\rG^{\ad}$
is a totally real field $\KK$.
Hence $\rG^{\ad}_\KK=\prod_{i=1}^k \rG^{\ad}_i$ with $\rG^{\ad}_i$ a geometrically simple $\KK$-group,
and so $\rG^{\ad}_\FF=\prod_{i=1}^k \rG^{\ad}_{i,\FF}$ for any field 
extension $\FF/\KK$.
Moreover, the Galois group $\Gal(\FF/\QQ)$ acts on $\rG^{\ad}_\FF$ by permuting its simple factors,
 meaning that $\vs(\rG^{\ad}_{i,\FF})=\rG^{\ad}_{i^\vs,\FF}$.
Note that a permutation $i\mapsto i^\vs$
is realized by an element $\vs\in\Gal(\FF/\QQ)$ 
if and only if the same permutation
is realized by some element of $\Gal(\KK/\QQ)$.
In particular, observe that 
\begin{itemize}
    \item[(a)]
condition (iii) in \Cref{def:shimura} implies that,
for each factor $\rG^{\ad}_{i,\FF}$,
there exists $\vs\in\Gal(\FF/\QQ)$
such that $\rG^{\ad}_{i^\vs,\RR}$ is non-compact;
\item[(b)]
if $(\rG,\dX)$ is non-compact
and $\rG^{\ad}$ is $\QQ$-simple, then all factors of $\rG^{\ad}_\RR$ are conjugate to one another,
and there are no compact factors.
\end{itemize}
Analogous considerations hold replacing $\rG^{\ad}$ by its universal cover $\wt{\rG}$.
\end{rem}

\subsection{Shimura varieties}\label{sec:Shimura-varieties}

We recall that the ring of {\it{finite ad{\`e}les}} $\AA_f$ is the subset of elements $a=(a_\ell)$ of $\prod_{\text{$\ell$ prime}} \QQ_\ell$ such that $a_\ell\in\ZZ_\ell$ for all but finitely many $\ell$'s. This can be equivalently defined as
$\AA_f=\hat{\ZZ}\otimes_{\ZZ}\QQ$, where $\displaystyle \hat{\ZZ}\coloneqq\lim_{\substack{\longleftarrow\\ N}} \ZZ/N=\prod_{\text{$\ell$ prime}}\ZZ_\ell$.

\begin{dfx}
    Let $(\rG,\dX)$ be a Shimura datum, and let $\rK\subset \rG_{\AA_f}$ be a compact open subgroup. There is a natural algebraic variety,  $\Sh_K(\dX,\rG)$, 
    defined over a number field $\EE(\rG,\dX)\subset \CC$,
    whose complex points are
$$
    \Sh_{\rK}(\rG,\dX)(\CC)
    \cong \rG(\QQ)\backslash \dX\times \rG(\AA_f)/\rK.$$
We call $\Sh_{\rK}(\rG,\dX)$ a \textit{Shimura variety}.
A {\em morphism of Shimura varieties} $\Sh_{\rK}(\rG,\dX)\ra \Sh_{\rK'}(\rG',\dX')$ is a morphism of Shimura data $(\rG,\dX)\ra (\rG',\dX')$ that sends $\rK\ra \rK'$ under the natural morphism $\rG_{\AA_f}\ra \rG'_{\AA_f}$.
\end{dfx}

We will use the same symbol for a morphism of Shimura varieties
and for the induced morphism of complex varieties.

\begin{rem}
As in \cite[Lemma 5.13]{milne2005},
if $(\GA_i)_{i\in I}$ is a system of representatives in 
$\rG(\AA_f)$ 
of the finite quotient 
$\rG(\QQ)\backslash \rG(\AA_f)/\rK$,
and denoting $\Gamma_i \coloneqq 
\rG(\QQ) 
\cap (\GA_i \rK \GA_i^{-1})$ for each $i\in I$, then
\[
\Sh_{\rK}(\rG,\dX)(\CC)
\cong \coprod_{i\in I} (\Gamma_i\backslash \dX)\, .
\]
Note that the Hermitian metric on $\dX$ descends to 
$\Sh_{\rK}(\rG,\dX)(\CC)$. 
Moreover 
$\Gamma_i\subset \rG(\QQ)$
is an {\it{arithmetic subgroup}},
namely $\Gamma_i$ is commensurable to 
$\rG(\QQ)\cap \GL_N(\ZZ)$
for any
embedding $\rG\hra \GL_N$.
\end{rem}

As we are interested in compact Shimura varieties, we shall make use of the following group-theoretic characterization.
A proof can be found in \cite[page 260, Theorem 5.5]{platonovrapinchuk}.

\begin{lm}\label{lm:aniso}
The Shimura variety $\Sh_{\rK}(\rG,\dX)(\CC)$ is compact if and only if $\rG^{\ad}$ is anisotropic.
\end{lm}

Motivated by this, we shall call a Shimura datum $(\rG,\dX)$ {\em compact} if $\rG^{\ad}$ is anisotropic.

\subsection{Special and weakly special subvarieties}\label{sec:weakly}

Given a Shimura datum $(\rG,\dX)$, there is a natural $\QQ$-variation of Hodge structures ($\QQ$-VHS) on~$\dX$  associated to any $\QQ$-representation $V$ of $\rG^{\ad}$
(see \cite[Section 2.3]{moonen1998linearityI}). 
This descends to a $\QQ$-VHS on 
$\Sh_{\rK}(\rG,\dX)(\CC)$ 
for any sufficiently small $\rK$.

Let $V_{\ad}$ denote the $\QQ$-VHS corresponding to the adjoint representation
of $\rG^{\ad}$ on its Lie algebra.
The  {\em{adjoint Mumford-Tate group}} of a point $[x,\GA]\in 
\Sh_{\rK}(\rG,\dX)(\CC)
$ is 
the smallest closed $\QQ$-subgroup 
$\MT^{\ad}_x(V)$
of $\rG^{\ad}$ 
whose complexification
contains the image of
$\bS_\CC\stackrel{x}{\longrightarrow} \rG_\CC\ra \rG^{\ad}_\CC$.
Note that
the adjoint Mumford-Tate group of a very general point of any irreducible component of 
$\Sh_{\rK}(\rG,\dX)(\CC)$ 
is $\rG^{\ad}$.

We say that a point $[x,\GA]\in 
\Sh_{\rK}(\rG,\dX)(\CC)
$ is {\em Hodge-generic} if its 
adjoint Mumford-Tate group is equal to~$\rG^{\ad}$. Likewise, an irreducible subvariety 
$Z\subseteq 
\Sh_{\rK}(\rG,\dX)(\CC)
$ 
is called Hodge-generic (within
$
\Sh_{\rK}(\rG,\dX)(\CC)
$) if its very general point is, or equivalently if it contains a Hodge-generic point.

\begin{dfx}[\cite{moonen1998linearityI}]
    A point $[x,\GA]\in 
\Sh_{\rK}(\rG,\dX)(\CC) 
$ is a {\em special point} if there exists a torus $\rT\subseteq \rG$
such that $x:\bS\ra \rG_\RR$ factors through $\rT_\RR$.
    An irreducible subvariety 
    $\Shi\subseteq 
\Sh_{\rK}(\rG,\dX)(\CC) 
$ is {\em weakly special} if it is totally geodesic;
moreover $\Shi$ is {\em special} if it is weakly special and contains a special point.
\end{dfx}

The following result of Moonen explains how all special and weakly special subvarieties arise from other Shimura varieties.

\begin{thm}[{\cite[Theorem 4.3]{moonen1998linearityI}}]\label{thm:weaklyspecials}
  Let $\Shi\subseteq 
\Sh_{\rK}(\rG,\dX)(\CC) 
$ be a weakly special subvariety. Then there exist
  \begin{enumerate}
      \item an embedding of Shimura data $(\rH,\dY)\hra (\rG,\dX)$
      \item a decomposition $(\rH^{\ad},\dY^{\ad})\cong (\rH^{\ad}_1,\dY^{\ad}_1)\times (\rH^{\ad}_2,\dY^{\ad}_2)$, and compatible components $\dY^+\cong \dY^{\ad,+}\cong \dY_1^{\ad,+}\times \dY_2^{\ad,+}$
      \item elements $\GA\in 
\rG(\AA_f)
$ and $y_2\in \dY_2^{\ad,+}$
  \end{enumerate}
such that $\Shi$ is the image of $\dY_1^{\ad,+}\times \{y_2\}\times  \{\GA\}$.

Moreover, if $\Shi$ is special, then $\rH_2$ may be taken to be trivial.
\end{thm}

Furthermore, we have the following result characterizing the smallest special subvariety containing a given irreducible variety.

\begin{thm}[{\cite[2.9]{moonen1998linearityI}}]\label{thm:specialclosure}
Let $Z\subseteq 
\Sh_{\rK}(\rG,\dX)(\CC)
$ be an irreducible subvariety.
Then there exists a smallest special subvariety
$\Shi\subseteq 
\Sh_{\rK}(\rG,\dX)(\CC) 
$ containing $Z$.
In particular, $\Shi$ is an irreducible component of the image of the map 
$\Sh_{\rK'}(\rG',\dX')(\CC)\ra \Sh_{\rK}(\rG,\dX)(\CC)$ 
induced by an embedding $(\rG',\dX')\hra (\rG,\dX)$ of Shimura data and a sufficiently small
compact open subgroup $\rK'$ of $\rG'_{\AA_f}$,
and
$(\rG')^{\ad}$ is the adjoint Mumford-Tate group of a very general point of $Z$.
\end{thm}

\section{Special subvarieties of 
$\cA_{g}$
}\label{sec:special-Ag}

In this section we review the construction of Siegel modular varieties, discuss symplectic representations
associated to their special subvarieties, and prove a splitting result for
non-compact special subvarieties of 
$\cA_{g}$
isogenous to products (\Cref{prop:productShimura}).

\subsection{Siegel modular variety}
Here we review the quintessential Siegel modular variety from the Shimura perspective. See \cite[\S6]{milne2005} for more details on this section.

Let~$V$ be a $\QQ$-vector space of dimension~$2g$, and let~$\psi$ be a symplectic form on~$V$.
We denote~ $\GSp(\psi)$ the general symplectic group,
whose set of $R$-points is
$$
\GSp(\psi)(R)
\coloneqq\left\{A\in \GL_R(V\otimes_\QQ R) \ \Big| 
\begin{array}{l}
\exists \chi(A)\in R^\times: \forall v,v'\in V\otimes_\QQ R \\
\psi(Av,Av')= \chi(A)\cdot\psi(v,v')
\end{array}\right\} .$$

The map $A\mapsto \chi(A)$ gives a {\it{similitude character}} $\GSp(\psi)\ra \GG_m$, whose kernel is the symplectic group $\Sp(\psi)$, which is also the derived subgroup of $\GSp(\psi)$. 

There is a natural Shimura datum $(\GSp(\psi),\dX(\psi))$ defined as follows. For any 
symplectic basis $\cB_\psi=(e_1,\dots,e_g,f_1,\dots,f_g)$ of $V$ define $\mu_{\cB}:\bS\ra \GSp(\psi)_\RR$ so that $$\mu_{\cB}(a+bi)(e_j)=ae_j+bf_j,\quad \mu_{\cB}(a+bi)(f_j)=af_j-be_j.$$ 
The set of all such $\mu_{\cB}$, as $\cB$ varies, constitutes $\dX(\psi)$. 

Then $\SD_g\coloneqq (\GSp(\psi),\dX(\psi))$ is the {\em Siegel Shimura datum} of genus~$g$. The name is justified by observing that~$\dX(\psi)$ can be
identified with $\Sieg^+_g\sqcup \Sieg^-_g$, where $\Sieg^+_g$ is the usual Siegel upper half-space.

If $L\subset V$ denotes an integral lattice that is self-dual with respect to~$\psi$, and $\rK\subset \rG_{\AA_f}$ denotes the stabilizer of $L\otimes\hat{\ZZ}$, then $\Sh_{\rK}(\GSp(\psi),\dX(\psi)))$
can be identified with the 
{\it{moduli space $\ASh_g$ of ppav}}.
For any $N\in\ZZ_{>0}$, if we denote $\rK_N\subset \rK$ the subgroup of those elements that act trivially on $L/(N\cdot L)$, 
then $\Sh_{\rK_{_N}}(\GSp(\psi),\dX(\psi))$ can be identified with the 
{\it{moduli space $\ASh_g(N)$ of ppav with full level-$N$ structure}}.

\begin{notation}
    In order to avoid too heavy notation, we use the symbol
    $\cA_g\coloneqq\ASh_{g,\CC}$ to denote the complex variety associated to $\ASh_g$, namely
    the {\it{coarse moduli space of complex ppav of dimension $g$}}.
    Similarly, we write $\cA_g(N)\coloneqq\ASh_g(N)_\CC$.
\end{notation}

Writing $(V,\psi)$ as a direct sum $(V_1,\psi_1)\oplus (V_2,\psi_2)$ of $\QQ$-symplectic spaces of dimensions $2g_1$ and $2g_2$, we obtain a group $$\GSp(\psi_1,\psi_2)\subset \GSp(\psi_1)\times\GSp(\psi_2)$$ defined as the subgroup of pairs with the same similitude character. We endow $\GSp(\psi_1,\psi_2)$ with  a Shimura datum $\dX(\psi_1,\psi_2)$ as follows.
Consider symplectic bases $\cB_k=(e^k_1,\dots,e^k_{g_k},f^k_1,\dots,f^k_{g_k})$ for $(V_k,\psi_k)$
and let $\dX(\psi_1,\psi_2)$ be the conjugacy class of the homomorphism
$\mu_{\cB_{1,2}}:\bS\ra \GSp(\psi_1,\psi_2)_\RR$ defined by
$\mu_{\cB_{1,2}}(a+bi)(e^k_j)=ae^k_j+bf^k_j$ and $\mu_{\cB_{1,2}}(a+bi)(f^k_j)=af^k_j-be^k_j$
for $k=1,2$.

Now there are two natural inclusions $\GSp(\psi_1,\psi_2)\subset \GSp(\psi_1)\times\GSp(\psi_2)$
and $\GSp(\psi_1,\psi_2)\subset\GSp(\psi)$.
The former determines an embedding of the Shimura datum $\SD_{g_1,g_2}\coloneqq(\GSp(\psi_1,\psi_2),\dX(\psi_1,\psi_2))$ into $\SD_{g_1}\times\SD_{g_2}$,
that identifies $\dX(\psi_1,\psi_2)$ to two of the four components of $\dX(\psi_1)\times \dX(\psi_2)$,
and which descends to \'etale covers of Shimura varieties.
The latter determines an embedding $\GSp(\psi_1,\psi_2)\hra \GSp(\psi)$.
If $L_k$ is an integral $\psi_k$-self-dual lattice in $V_k$,
then $L=L_1\oplus L_2$ is an integral $\psi$-self-dual lattice in $V$
and the embedding $\GSp(\psi_1,\psi_2)\hra \GSp(\psi)$ induces
the standard map 
$\iota:\ASh_{g_1}\times\ASh_{g_2}\ra\ASh_{g_1+g_2}$.

\begin{dfx}\label{dfx:indecomposable}
    An embedding of Shimura data $(\rG,\dX)\hra \SD_g$ is called {\em indecomposable} if it does not factor through the standard morphism $\SD_{g_1,g_2}\hra \SD_g$ for any positive $g_1,g_2$ with $g_1+g_2=g$. 
    A special subvariety $\Shi\subseteq 
    \cA_{g}$ is called {\em indecomposable} if the corresponding embedding of Shimura data is indecomposable.
\end{dfx}

From the representation-theoretic point of view,
an embedding $(\rG,\dX)\hra\SD_g$ is indecomposable if and only if the corresponding
$\GSp$-representation $V$ of $\rG$ is $\GSp$-irreducible (as defined in \Cref{sec:symplectic} below).

From a geometric point of view,
it is easy to see that an embedding $(\rG,\dX)\hra\SD_g$ is indecomposable if and only if the corresponding family of $g$-dimensional abelian varieties over $\dX$ has no nontrivial abelian subvarieties, namely is not isogenous to a product family.
 
The above considerations show that, for a decomposable special subvariety $\Shi\subset 
\cA_{g}
$,
there exist positive integers $g_1,g_2$ with $g_1+g_2=g$ and a commutative diagram
\[
\xymatrix{
\Sh_{\rK}(\rG,\dX)(\CC)^+ \ar[d]_\pi \ar@{^(->}[r] & \Sh_{\rK}(\rG,\dX)(\CC)  \ar@{^(->}[r]^{\Psi\quad} & \cA_{g_1}(N)\times\cA_{g_2}(N) \ar[d]^{\Phi} \\
S \ar@{^(->}[rr] && \cA_{g}
}
\]
where 
$\Sh_{\rK}(\rG,\dX)(\CC)^+$
is a connected component of 
$\Sh_{\rK}(\rG,\dX)(\CC)$,
the vertical map on the left is a finite cover, $\Psi$ is a morphism of Shimura varieties
and $\Phi$ is a Hecke translate of the standard map.{
Hence $\Shi$ is a Hecke translate of a special subvariety of 
$\cA_{g}$
contained
in the locus of products.

\subsection{Symplectic representations}\label{sec:symplectic}

Now we want to briefly discuss representations that are associated to morphisms of some Shimura datum to the Siegel Shimura datum.

\begin{dfx}\label{dfx:Sp-irr}
Let $\rG$ be a connected reductive algebraic group over $\KK$. An {\em $\Sp$-representation} $(V,\psi)$ of $G$
is a $\KK$-representation $V$ of $\rG$ equipped with a 
non-degenerate symplectic pairing $\psi$ preserved by the action of~$\rG$.  An {\em $\Sp$-subrepresentation} is a $\rG$-invariant symplectic subspace. 
An $\Sp$-subrepresentation is called {\em $\Sp$-irreducible} if it has no nontrivial symplectic subrepresentations. 

Likewise, a {\em{$\GSp$-representation}} of $\rG$ is a representation $V$ together with a non-degenerate symplectic pairing $\psi$ that is preserved by $\rG$ up to scaling; $\GSp$-subrepresentations and $\GSp$-irreducibility are defined analogously.
\end{dfx}

Given a $\GSp$-representation $(V,\psi)$ of $\rG$, we obtain a
character $\chi$ of $\rG$ such that $\psi(gv,gv')=\chi(g)\psi(v,v')$
for all $v,w\in V$. The natural vector space isomorphism $V\ra V^\vee$, induced by $\psi$, gives an isomorphism of $\rG$-representations $V\ra V^\vee \otimes \chi$.  We call 
$V^\vee(\chi)\coloneqq V^{\vee}\otimes\chi$ the \textit{$\GSp$-dual} representation.

The following lemma for $\Sp$-representations
was stated in \cite[Theorem 3]{malcev} without proof,
and then
proven in \cite[Section 2.3]{satake:imb} over $\RR$. 
For completeness, here we recall the argument which works over any field.

\begin{lm}\label{lm:direct-sum}
Any $\GSp$-representation of $\rG$ is a direct sum of $\GSp$-irreducible subrepresentations.
\end{lm}

\begin{proof}
We proceed by induction on the dimension of the $\GSp$-representation $(V,\psi)$. The case $\dim V=0$ is trivial.

Suppose $\dim V>0$.
It is enough to show that there exists a nontrivial
$\GSp$-irreducible subrepresentation $U$ of $V$.
Indeed, $V$ will then decompose as a direct sum $V=U\oplus U^\perp$
of $\rG$-invariant subspaces, and the conclusion follows
applying the inductive hypothesis to $U^\perp$.

Now, consider
an irreducible $\rG$-invariant linear subspace $W$ of $V$.
The pairing $\psi|_W$ induces a $\rG$-map $\beta:W\ra W^\vee\otimes \chi$ that must be either zero or an isomorphism, since $\mathrm{ker}(\beta)$ is an invariant linear summand of $W$.

If $\beta$ is an isomorphism, then $W$ is a $\GSp$-subrepresentation,
hence $\GSp$-irreducible, and so we are done.

Suppose now that
the restriction of $\psi$
to every irreducible $\rG$-invariant linear subspace of $V$ is zero.
The pairing against $W$ with $\psi$ gives a map $\phi:V\ra W^\vee(\chi)$.
Since $\psi$ is non-degenerate, the map $\phi$ is nonzero and thus surjective (by the irreducibility of $W^\vee$).
Since $\rG$ is reductive and $\ker(\phi)$ is $\rG$-invariant,
there exists an irreducible $\rG$-invariant linear subspace
$W'$ of $V$ such that $V=\ker(\phi)\oplus W'$.
Then the induced map $\psi': W'\ra W^\vee(\chi)$ is 
an isomorphism. Since $W$ is isotropic for $\psi$, it follows that
$W\cap W'=\{0\}$ and so $W\oplus W'$ is a $\GSp$-subrepresentation.
We conclude by observing that $W\oplus W'$
must be $\GSp$-irreducible, as its only nontrivial $\rG$-invariant subspaces are themselves irreducible, and by our assumption the restriction of $\psi$ to them is zero.
\end{proof}

If $\rho$ is an irreducible linear representation of $\rG$,
the {\it{isotypic component}} $V^\rho$ 
is the subspace of~$V$ generated by all subrepresentations of $V$
isomorphic to $\rho$. The decomposition $V=\bigoplus_{\rho} V^\rho$
is canonical.

If $\rho$ is symplectic, then $V^{\rho}$ is a direct sum
of representations isomorphic to $\rho$ by \Cref{lm:direct-sum}.
If $\rho$ is not symplectic and $\rho^\vee\otimes\chi$ is its $\GSp$-dual,
then $V^\rho\oplus V^{(\rho^\vee\otimes\chi)}$ is symplectic
and so it is a direct sum of $\GSp$-irreducible representations
isomorphic to $\rho\oplus(\rho^\vee\otimes\chi)$.
It follows that each $V^\rho$ decomposes as a direct sum
of copies of $\rho$, though non-canonically.
\\

The following simple observation
will be useful in
\Cref{sec:decoupled} and in \Cref{sec:non-decoupled}.

\begin{lm}\label{lm:m}
    Let  $(V,\psi)$ be 
a $\GSp$-irreducible 
    representation of a reductive group $G$.
Then all the irreducible subrepresentations of 
    $V_{\CC}$ 
    are either Galois conjugate to each other or to each other's $\GSp$-duals. In particular, they all have the same dimension, and are all of the same {\em type}: orthogonal, or symplectic, or not self-dual (NSD).
\end{lm}

\begin{proof}
    Let 
$U\subseteq V_\CC$ be an irreducible linear summand, and let $U'\subseteq V_\CC$
denote the
     subrepresentation generated by all
    subspaces of $V_\CC$ that are isomorphic to a Galois conjugate of $U$ or of its $\GSp$-dual $U^\vee(\chi)$. Then $U'$ is Galois stable, and thus 
    $U'=V'_\CC$ for some rational subrepresentation $V'$ of $V$.
    Moreover, 
    $V$ decomposes as $V'\oplus V''$, with $V''$ not having any of the irreducible linear summands that appear in~$V'$
    nor their $\GSp$-duals.    
    Thus $V'$ is symplectic. Since $V$ is assumed to be $\GSp$-irreducible, we must have $V=V'$ and the conclusion follows.
\end{proof}

We have seen that special subvarieties of 
$\cA_{g}$
are associated
to embeddings of Shimura data $(\rG,\dX)$ into $\SD_g$.
These Shimura data, already studied in \cite{deligne}, are called {\it{of Hodge type}} in
\cite[end of Section II.3]{milne1990}.
This motivates the following definition.

\begin{dfx}
A {\em{symplectic representation}} of a Shimura datum $(\rG,\dX)$ is a $\GSp$-representation $(V,\psi)$ of $\rG$
that induces a morphism $(\rG,\dX)\ra (\GSp(\psi),\dX(\psi))$ of Shimura data. This morphism is called \textit{irreducible} if $(V,\psi)$ is $\GSp$-irreducible as a representation of $\rG$.
\end{dfx}

We will often invoke the following useful result
from \cite[Section 1.3.8]{deligne} (see also \cite[Corollary 10.9]{milne2011shimura}), 
which is essentially
a byproduct of the analysis carried out in \cite[Section 2]{satake-compact}.

\begin{lm}\label{lm:10.7}
Let $(V,\psi)$ be a symplectic representation of a Shimura datum~$(\rG,\dX)$, and let $U\subseteq V_\CC$ denote an irreducible complex-linear subrepresentation of $\rG_\CC$. Let $\wt{\rG}_1,\dots,\wt{\rG}_k$ denote all the simple non-compact
factors of $\wt{\rG}_\RR$.
Then at most one of the $\wt{\rG}_{i,\CC}$ acts nontrivially on $U$.
\end{lm}

\subsection{Products of Shimura data in $\SD_g$}\label{sec:products}
 Using the results recalled above, we now study embeddings into $\SD_g$
 of Shimura data that geometrically factor as a product 
of a non-compact Shimura datum and another Shimura datum. We show that these embeddings 
factor in a suitable way.

We suspect that the following result might be already known, but we could not find 
a reference in the literature. So we include a complete proof.

\begin{thm}\label{prop:productShimura}
    Let $(\rG_{(1)},\dX_{(1)})$ be an adjoint Shimura datum with $\rG_{(1)}$ simple and such that $\rG_{(1),\RR}$ has no compact factors, and let $(\rG_{(2)},\dX_{(2)})$ be any adjoint Shimura datum of positive dimension. 
    If $f:(\rH,\dX_{\rH})\hra \SD_g$ is an embedding of Shimura data, with $(\rH^{\ad},\dX_{\rH}^{\ad})\cong (\rG_{(1)},\dX_{(1)})\times (\rG_{(2)},\dX_{(2)})$, then $f$ is decomposable.
\end{thm}
\begin{proof}
Let $(V,\psi)$ be the $2g$-dimensional $\GSp$-representation of $\rH$
corresponding to the morphism $f$. 
Consider the symplectic representation of 
$\wt{\rG}_{(1)}\times \wt{\rG}_{(2)}$ on $V$ via
$\wt{\rG}_{(1)}\times \wt{\rG}_{(2)}\cong\wt{\rH}\twoheadrightarrow \rH^{\der}\ra \Sp(V,\psi)$, and
canonically decompose $V$ into a direct sum of isotypic components
for the action of $\wt{\rH}$.
We can then write
$V=V_\emptyset\oplus V_1\oplus V_2\oplus V_{1,2}$, where 
$V_\emptyset$ is the trivial isotypic component,
$V_1$ (resp.~$V_2$) is the sum of isotypic components
on which $\wt{\rG}_{(1)}$ 
(resp.~$\wt{\rG}_{(2)}$)
acts nontrivially and $\wt{\rG}_{(2)}$ 
(resp.~$\wt{\rG}_{(1)}$)
acts trivially,
and $V_{1,2}$ is the sum of isotypic components
on which both $\wt{\rG}_{(1)}$ and $\wt{\rG}_{(2)}$ act nontrivially.

Note that
the decomposition $V=V_\emptyset\oplus V_1\oplus V_2\oplus V_{1,2}$ is orthogonal for $\psi$
and $\rH^{\der}$-invariant.
Since $\rH^{\der}$ and $\rZ(\rH)$ generate $\rH$,
the images of $\wt{\rG}_{(1)}$ and $\wt{\rG}_{(2)}$ in $\rH$ are both normal.
Hence both splittings $V=(V_\emptyset\oplus V_2)\oplus (V_1\oplus V_{1,2})$
and $V=(V_\emptyset\oplus V_1)\oplus (V_2\oplus V_{1,2})$ are preserved by $\rH$.
It follows that each of $V_\emptyset$, $V_1$, $V_2$, $V_{1,2}$ is a $\rH$-subrepresentation.

We claim that $V_{1,2}=\{0\}$.

Assuming the claim,
we can write $V=V_1'\oplus V_2$, where $V'_1\coloneqq V_\emptyset\oplus V_1$. Setting $\psi_1\coloneqq \psi
\mid_{V'_1}$ and $\psi_2\coloneqq \psi\mid_{V_2}$, we then obtain that $f$ factors through $f':\rH\ra \GSp(\psi_1,\psi_2)$, which gives the desired result.

In order to prove the claim, we proceed by contradiction and we consider
a nontrivial irreducible 
$(\wt{\rG}_{(1),\CC}\times\wt{\rG}_{(2),\CC})$-invariant subspace
$U$ of 
$(V_{1,2})_{\CC}$.
Note that both $\wt{\rG}_{(1),\CC}$ and $\wt{\rG}_{(2),\CC}$
act nontrivially on $U$.
Hence
we can write $U\cong U_1\otimes U_2$, 
where $U_j$ is a nontrivial irreducible 
representation of $\wt{\rG}_{(j),\CC}$ 
for $j=1,2$.

Now let $\wt{\rG}'_2$ be a simple factor of $\wt{\rG}_{(2),\RR}$
such that $\wt{\rG}'_{2,\CC}$ acts nontrivially on $U_2$.
By \Cref{rem:galois}(a), there exists $\vs\in\Gal(\RR/\QQ)$ such that 
$\wt{\rG}''_2\coloneqq \vs(\wt{\rG}'_2)$
is non-compact. 
Let $\vs'\in\Gal(\CC/\QQ)$ be an element that sends
$\wt{\rG}'_{2,\CC}$ to $\wt{\rG}''_{2,\CC}$.
Up to replacing 
$\wt{\rG}'_2$ by $\wt{\rG}''_2$ and
$U$ by $\vs'(U)$, we may and do assume that 
$\wt{\rG}'_2$ is non-compact.
Likewise, let $\wt{\rG}'_1$ be a simple
factor of 
$\wt{\rG}_{(1),\RR}$ such that
$\wt{\rG}'_{1,\CC}$ acts nontrivially on $U_1$. 
By assumption 
$\wt{\rG}'_1$ 
is non-compact.

Thus, $\wt{\rG}'_1$ and $\wt{\rG}'_2$ are distinct, non-compact,
simple $\RR$-factors of $\wt{\rH}_\RR$,
and both $\wt{\rG}'_{1,\CC}$ and $\wt{\rG}'_{2,\CC}$ act
nontrivially on $U$.
This contradicts \Cref{lm:10.7}, and so the claim is proven.
\end{proof}

\section{Ax-Schanuel for Shimura varieties,\\ and Hodge-generic compact subvarieties}\label{sec:ax-schanuel}

In this section we recall the statement of the (weak) Ax-Schanuel conjecture for Shimura varieties, proven in \cite{MPT}.
As a consequence, we provide a compactness criterion for subvarieties of 
$\cA_{g}$,
borrowing an idea from
\cite[Theorem 1.9(i)]{khur}, and
use this criterion to prove a more general version of \Cref{thm:dmcgAg}.
Finally, we will link the maximal dimension of a
compact subvariety of 
$\cA_{g}$ 
to the 
maximal dimensions
of compact special subvarieties
of 
$\cA_{g'}$
for $g'\leq g$.

\subsection{Ax-Schanuel for Shimura varieties, and a compactness criterion}\label{sec:ax}

Consider a Shimura variety~$\Sh_{\rK}(\rG,\dX)$ and let $\dX^+$ be a connected component of $\dX$.
Denote by $\pi:\dX^+\rightarrow 
\Sh_{\rK}(\rG,\dX)(\CC) 
$ the natural projection,
and let $\dX^+\subset \check{\dX}$ be the embedding in the Hermitian symmetric space of compact type dual to $\dX^+$, which is a projective variety (see, for instance, \cite[Chapter VIII, Proposition 7.14]{helgason}).

\begin{thm}[Weak Ax-Schanuel \cite{MPT}]\label{thm:AS}
Let $Y$ be an algebraic subvariety of $\dX^+$ (namely, obtained by intersecting $\dX^+$ with an algebraic subvariety of $\check{\dX}$),
and let $Z\subseteq 
\Sh_{\rK}(\rG,\dX)(\CC)
$ be an irreducible algebraic subvariety.
If $\pi^{-1}(Z)\cap Y$ has an analytic irreducible component~$C$ 
of dimension larger than expected, then~$\pi(C)$ is contained in a weakly special subvariety of 
$\Sh_{\rK}(\rG,\dX)(\CC)$,
of dimension smaller than 
$\Sh_{\rK}(\rG,\dX)(\CC)$
itself.
\end{thm}

As a consequence of \Cref{thm:AS}, one obtains \Cref{thm:KU} below, which is essentially a simplified version of \cite[Theorem 1.9(i)]{khur}, whose proof we include for completeness.

For a subvariety~$Z'\subseteq 
\Sh_{\rK}(\rG,\dX)(\CC)
$, for any $\gamma\in 
\rG(\QQ)^+
$ the image $\pi(\gamma\cdot\pi^{-1}(Z'))$
is a subvariety of 
$\Sh_{\rK}(\rG,\dX)(\CC)$
called the {\em $\gamma$-translate of $Y$}. 
We remark that, even if $Z'$ is irreducible, its $\gamma$-translate need not be.
However, an irreducible component of the $\gamma$-translate of a (weakly) special subvariety is again (weakly) special.

\begin{thm}\label{thm:KU}
Let $Z\subseteq 
\cA_{g} 
$ be an irreducible subvariety, such that the adjoint Mumford-Tate group 
of its very general point $\MT^{\ad}_Z$ is $\QQ$-simple,
and let $\Shi\subseteq
\cA_{g} 
$ be the smallest special subvariety containing~$Z$.

If $\Shi'\subseteq 
\cA_{g} 
$ is 
a special subvariety such that 
\begin{itemize}
\item[(a)]
$\dim Z+\dim \Shi'\ge \dim \cA_g$, and
\item[(b)]
$\Shi'\cap \Shi\neq\emptyset$,
\end{itemize}
then there exists 
$\gamma\in
\GSp(\psi)(\QQ)^+
$ such that
the intersection of $Z$ with the $\gamma$-translate
of $\Shi'$ inside $
\cA_{g} 
$ has a component
of the expected dimension.
\end{thm}

\begin{rem}\label{rem:intersection}
A simple observation that is regularly used in the below arguments is the following. Let $M$ and $N$ be closed complex-analytic subvarieties of an ambient complex variety $X$. Then if $M\cap N$ has a (nonempty) component $C$ of the expected dimension, then $M\cap N_t$ has a (nonempty) component of the expected dimension for any small perturbation $(N_t)$ of $N=N_0$ inside $X$.
Indeed, we can first localize to a small ball $B\subset X$ centered at a point $p$ of $C$.
Then we can reduce to the case of $M,N$ of complementary dimensions in $X$, by
intersecting both $M$ and $N$ by $\dim(C)$ general hypersurfaces of $B$ through $p$.
Namely, we can assume that $\dim(M)+\dim(N)=\dim(X)$, and that $B\cap (M\cap N)=\{p\}$.
Up to shrinking $B$ a little, we can also assume that $\partial B\cap(M\cap N_t)=\emptyset$
for small $t$.
The above claim then follows from the upper semi-continuity of the dimension $t\mapsto\dim (B\cap M)\cap N_t$, and from the positivity and constancy of the local intersection number $t\mapsto \mathrm{deg}\,(B\cap M)\cdot N_t$.
\end{rem}

\begin{proof}[Proof of \Cref{thm:KU}]
Let $\Shi$ correspond to the 
embedding $(\rG,\dX)\hra (\GSp(\psi),\dX(\psi))$ of Shimura data.
By \Cref{thm:specialclosure} it follows that $\rG^{\ad}\cong \MT_Z^{\ad}$ and is therefore $\QQ$-simple. 

Now we obtain a natural map $\pi:\dX^+\ra \Shi\subseteq 
\cA_{g}
$. Fix a point $z_0\in Z$ that is
Hodge-generic in $S$, choose $\tilde{z}_0\in\pi^{-1}(z_0)$, and let $\gamma_0\in 
\rG(\RR)^+
$ be such that $\tilde{z}_0\in \gamma_0\cdot \pi^{-1}(\Shi\cap\Shi')$.

By (a), we have that
$$\dim Z+ \dim (\Shi'\cap\Shi)\geq \dim Z + \dim\Shi'+\dim\Shi-\dim\cA_g\geq \dim \Shi.$$ Moreover, since $z_0$ is Hodge-generic in $S$, by \Cref{thm:specialclosure} and \Cref{thm:weaklyspecials} there are no weakly special subvarieties of $\Shi$ containing $z_0$, other than $\{z_0\}$ or $\Shi$ itself. Thus, it follows from \Cref{thm:AS} that 
$\pi^{-1}(Z)\cap (\gamma_0\cdot\pi^{-1}(\Shi\cap\Shi'))$ 
has an analytic irreducible component of the expected dimension, containing~$\tilde{z}_0$.
 
By \Cref{rem:intersection},
for any $\gamma\in 
\rG(\RR)^+
$ sufficiently close to~$\gamma_0$, there will also exist an analytic irreducible component 
$C_\gamma$
of $\pi^{-1}(Z)\cap (\gamma\cdot\pi^{-1}(\Shi\cap\Shi'))$ that still has the expected (non-negative) dimension.
In particular, since 
$\rG(\QQ)^+$
is dense in 
$\rG(\RR)^+$
for the classical topology, 
this property holds for any element of 
$\rG(\QQ)^+$
sufficiently close to~$\gamma_0$. 
We fix one such $\gamma$ in 
$\rG(\QQ)^+$.
Regarding $\gamma$ as an element of 
$\GSp(\psi)(\QQ)$,
it is then enough to consider the $\gamma$-translate of $S'$ inside 
$\cA_{g}$,
since a component of its intersection
with $Z$ is exactly $\pi(C_\gamma)$.
\end{proof}

We would like to extend \Cref{thm:KU} to the intersection of $Z$ with a weakly special subvariety.
The following \Cref{cor:KU}, which can also be understood as a non-compactness criterion for $Z$,
will play a key role in the proof of \Cref{thm:dmcNCShimura}.

In order to explain the idea behind \Cref{cor:KU} below with an example, consider the special subvariety $S'\subsetneq
\cA_{g} 
$ obtained as the image of the natural map 
$\iota:\cA_{1}\times\cA_{g-1}\ra \cA_{g}$.
The image $\iota([E]\times 
\cA_{g-1}
)$ is closed of codimension $g$ for all $[E]\in
\cA_{1}
$;
moreover, it limits to~$\partial
\cA_{g}^*
$ as $[E]$ approaches the boundary point of  
$\cA_{1}$.
Now, for the sake of contradiction, consider a compact subvariety $Z\subsetneq
\cA_{g} 
$ of dimension at least $g$.
If the intersection of $Z$ and $\iota([E]\times
\cA_{g-1}
)$ is nonempty and 
has a component of the expected dimension, such property is preserved under small perturbations of $[E]$ (see \Cref{rem:intersection}). Thus, by algebraicity, $Z$ has nonempty intersection with $\iota([E]\times
\cA_{g-1}
)$ for all but possibly finitely many $[E]$, which shows that $Z$ is not compact. In order to achieve non-emptiness
and expected dimension of the intersection,
we use $\gamma$-translates and \Cref{thm:AS},
as in the proof of \Cref{thm:KU}, and deal with a general finite map $\tau':S_1'\times S_2'\to S'$ instead of the standard~$\iota$.

\begin{cor}\label{cor:KU}
Let $Z$ and $S$ be as in \Cref{thm:KU}.
Assume that $\Shi'$ is a special subvariety
corresponding to $(\rH,\dY)\hra(\GSp(\psi),\dX(\psi))$ and that $(\rH^{\ad},\dY^{\ad})\cong (\rH_1,\dY_1)\times (\rH_2,\dY_2)$, so that 
the natural morphism $\dY_1^+\times \dY_2^+\ra \Shi'$ then factors through a finite map $\tau':\Shi'_1\times \Shi'_2\ra \Shi'$.

Assume moreover that
\begin{itemize}
\item[(a)]
$\dim Z+\dim\Shi'_2\ge \dim \cA_g$, and
\item[(b)]
$\tau'(\{s'_1\}\times \Shi'_2)\cap \Shi\neq\emptyset$ for a general  $s'_1\in \Shi'_1$.
\end{itemize}
Then
\begin{itemize}
\item[(i)]
there exists $\gamma\in 
\rG(\QQ)^+
$ such that 
the $\gamma$-translate of $\tau'(\{s'_1\}\times\Shi'_2)$ inside 
$\cA_{g}$
intersects $Z$
in a component of the expected dimension, for a generic $s'_1\in \Shi'_1$;
\item[(ii)]
if $\Shi'_1$ is non-compact, then $Z$ is non-compact.
\end{itemize}
\end{cor}
\begin{proof}
(i) Fix any $s'_1\in\Shi'_1$. The same proof as in \Cref{thm:KU} works identically to show that there exists $\gamma\in 
\rG(\QQ)^+
$ such that 
the $\gamma$-translate of $\tau'(\{s'_1\}\times \Shi'_2)$ 
intersects $Z$ in a component of expected dimension. But now this property holds for $s'_1$ in a Zariski-open subset of $S'_1$.

(ii) Let $\gamma$ be as in (i), and replace $\Shi'$ 
by 
a component of its $\gamma$-translate in 
$\cA_{g}$,
so that the conclusion of (i) holds.
Then $(\tau')^{-1}(Z)$ surjects onto a Zariski open subset of $\Shi'_1$. Since this is non-compact, it follows that $(\tau')^{-1}(Z)$ is not compact. Since $\tau'$ is finite (and therefore proper), it follows that $Z$ is not compact either.
\end{proof}

\subsection{Hodge-generic compact subvarieties of 
$\cA_{g}$
}\label{sec:reduction}

Given an irreducible subvariety $Z\subseteq
\cA_{g}
$, by definition $Z$ is Hodge-generic within the smallest special subvariety~$\Shi$ 
of 
$\cA_{g}$
containing~$Z$. 
We recall that indecomposability for $S$ was introduced in
\Cref{dfx:indecomposable}.

In this section we prove the following.

\begin{thm}\label{thm:dmcNCShimura}
Let $\Shi\subseteq 
\cA_{g}
$ be an indecomposable, non-compact, special subvariety, and let $Z\subsetneq \Shi$ be an irreducible compact subvariety that is Hodge-generic within~$\Shi$. Then $\dim Z \le g-1$.
\end{thm}

The first main result of the present paper is an immediate consequence.

\begin{proof}[Proof of \Cref{thm:dmcgAg}]
This is simply the $\Shi=
\cA_{g}
$ case of \Cref{thm:dmcNCShimura}.
\end{proof}

\begin{rem}
In fact, the proof of \Cref{thm:dmcNCShimura} in the special case $\Shi=
\cA_{g}
$ is very quick, and does not require any of the below technical \Cref{lm:modcurveinnoncompact}, \Cref{cor:modcurveinnoncompact} or \Cref{lm:Eisog}. Indeed, it implements the idea discussed in \Cref{sec:introA}, making use of \Cref{cor:KU}(ii)
with $\Shi'_1=
\cA_{1}
$ and $\Shi'_2=
\cA_{g-1}
$
and $\tau'=\iota:
\cA_{1}\times\cA_{g-1}\ra \iota(\cA_{1}\times\cA_{g-1})\subset\cA_{g}
$.
\end{rem}

The idea of the proof of \Cref{thm:dmcNCShimura}, for arbitrary non-compact $\Shi$, is to first
produce a special subvariety of $\cA_g$ as the image of
a morphism $\Phi:\Shi'_1\times\Shi'_2\rightarrow
\cA_{g}$ with 
$\Shi'_1=\cA_1(N)$ and $\Shi'_2=\cA_{g-1}(N)$ 
for some $N$
(where we recall that 
$\cA_g(N)$
denotes the coarse moduli of complex ppav with a full level $N$ structure) 
in such a way that
$\Phi(\{s'_1\}\times\Shi'_2)$ intersects $\Shi$ 
for $s'_1$ in a Zariski-dense open subset of $\Shi'_1$.
By \Cref{cor:KU}(ii), the compactness of $Z$ then forces $\dim Z\le g-1$.

The construction of $\Phi$ uses the following preliminary lemma, in which
we first exhibit a holomorphic map from a modular curve to a totally geodesic subvariety of $\Shi$.

\begin{lm}\label{lm:modcurveinnoncompact}
Let $\Sh_{\rK}(\rG,\dX)^+$ be a connected component of the Shimura variety
    $\Sh_{\rK}(\rG,\dX)$ and suppose that $\Sh_{\rK}(\rG,\dX)(\CC)$ is non-compact.
    Then there is a homomorphism of $\QQ$-groups $\phi:\SL_2\rightarrow \rG$ and a point $x\in \Sh_{\rK}(\rG,\dX)(\CC)^+$ such that $x\mapsto \phi(g)\cdot x$ induces a holomorphic map of the modular curve (with level) $\cA_1(N)\rightarrow \Sh_{\rK}(\rG,\dX)(\CC)^+$, for some large enough $N$, whose image is totally geodesic.
\end{lm}
\begin{proof}
The adjoint action of $\rG$ on its Lie algebra
determines a canonical variation of Hodge structure on 
$\Sh_{\rK}(\rG,\dX)(\CC)$
(see \cite[\S5]{milne2011shimura}). 
Since 
$\Sh_{\rK}(\rG,\dX)(\CC)$
is non-compact, the group $\rG$ contains nontrivial unipotent elements, and the variation associated to 
$\Sh_{\rK}(\rG,\dX)(\CC)$
degenerates. The conclusion now follows from \cite[Cor 5.19]{schmid}. (Technically, Schmid works only for period domains corresponding to the full orthogonal/symplectic group, but the argument goes through unchanged in our context). 
\end{proof}

We can now construct a morphism of Shimura varieties from a refined version of the lemma above:

\begin{cor}\label{cor:modcurveinnoncompact}
     Let $\Sh_{\rK}(\rG,\dX)^+$ be a 
     connected component of the Shimura variety $\Sh_{\rK}(\rG,\dX)$,
     and suppose that 
$\Sh_{\rK}(\rG,\dX)(\CC)$
is non-compact.
     Then there is an embedding of Shimura data $(\rH',\dY')\hra (\rG,\dX)$ 
     and a connected component $(\dY')^+$ of $\dY'$
     such that $(\rH')^{\ad}\cong \SL_2^{\ad}$
     and $(\dY')^+$ maps to 
$\Sh_{\rK}(\rG,\dX)(\CC)^+$.
\end{cor}

\begin{proof}
    By \Cref{lm:modcurveinnoncompact} there is a weakly special subvariety of 
$\Sh_{\rK}(\rG,\dX)(\CC)^+$ 
isomorphic to a modular curve.  By \Cref{thm:weaklyspecials} this means that there is an embedding $(\rH,\dY)\hra (\rG,\dX)$ and a splitting $$(\rH^\ad,\dY^\ad)=(\rH^{\ad}_1,\dY^{\ad}_1)\times (\rH^{\ad}_2,\dY^{\ad}_2)$$ such that the modular curve is the image of $\dY^{\ad,+}_1\times\{y_2\}\times\{\GA\}$ for some $y_2\in \dY^{\ad,+}_2$ and $\GA\in 
\rG(\AA_f)
$. It follows that $\dY^{\ad,+}_1$ is the standard upper half-plane, and $\rH_1^{\ad}\cong \SL_2^{\ad}$.

    Now we simply take a special point $y_{2,s}\in \dY^{\ad,+}_2$ corresponding to a homomorphism whose image is contained in 
    $\rT_\RR$ for
    a maximal $\QQ$-torus $\rT\subset \rH^{\ad}_2$, and define $\rH'\subset \rH$ as 
    the connected component of the preimage of $\rH^{\ad}_1\times \rT$ that contains the identity.

Pick any lift $y'$ of $(y_1,y_{2,s})$, which
thus factors through $\rH'$.
Letting $\dY'$ be the 
$\rH'(\RR)$-conjugacy class
of $y'$
yields an embedding of Shimura data $(\rH',\dY')\hra (\rH,\dY)\hra (\rG,\dX)$ such that $(\rH')^{\ad}\cong \SL_2^{\ad}$, as desired.
    Finally, the connected component of $\dY'$ that
    contains $y'$ maps to 
$\Sh_{\rK}(\rG,\dX)(\CC)^+$,
and so this
    is the desired $(\dY')^+$.
\end{proof}

We can now construct the desired morphism $\Phi$.

\begin{lm}\label{lm:Eisog}
    Let $\Shi\subseteq 
    \cA_{g} 
    $ be a non-compact special subvariety. Then there is a morphism of Shimura varieties $\Phi:\ASh_1(N)\times \ASh_{g-1}(N)\rightarrow \ASh_g$, and an irreducible component $\cA_1(N)^+$ such that for each $[E]\in
    \cA_1(N)^+
    $ the intersection $\Shi\cap \Phi\left([E]\times 
    \cA_{g-1}(N)
    \right)$ is nonempty.
\end{lm}

\begin{proof}
Let $\iota:(\rG,\dX)\hra \SD_g=(\GSp(\psi),\dX(\psi))$ be the embedding of Shimura data, associated to a representation of $\rG$ on the $\QQ$-symplectic space $(V,\psi)$, and let $\dX^+$ be a connected component
of $\dX$,
that induce $\dX^+\twoheadrightarrow\Shi\subseteq
\cA_{g}
$.
Denote by $(\dX^{\ad})^+$ the component of $\dX^{\ad}$ which is the image of $\dX^+$.

By \Cref{cor:modcurveinnoncompact} there is an embedding of Shimura data $(\rH',\dY')\hra (\rG,\dX)$ 
and a component $(\dY')^+$ of $\dY'$
so that
$(\rH')^{\ad}\cong \SL^\ad_2$
and $(\dY')^+$ maps to $\dX^+$. This gives a smaller non-compact special subvariety in $\Shi$, and hence we may and do reduce to the case 
$(\rG,\dX)=(\rH',\dY')$ and so $\rG^{\ad}\cong \SL^{\ad}_2$. 

Without loss of generality, we may assume that 
the smallest
closed normal $\QQ$-subgroup of $\rG$
whose set of real points contains the image of each $x\in \dX$ is $\rG$ itself (otherwise, we could replace $\rG$
by such closed normal subgroup, which is also reductive).

Let $\rT=\rZ(\rG)^\circ$ be the connected
component of $\rZ(\rG)$ that contains the identity.
We therefore have an isogeny $f:\SL_2\times \rT\ra \rG$.

Let $j:\mathrm{U}_1\hra \SL_{2,\RR}$ be the embedding given by the standard identification
$\mathrm{U}_1\cong \SO_{2,\RR}$, and note that the composition $$\mathrm{U}_1\xrightarrow{j}\SL_{2,\RR}\hra \SL_{2,\RR}\times \rT_\RR\xrightarrow{f} \rG_\RR\ra \rG^{\ad}_\RR\cong \SL_{2,\RR}^{\ad}$$ agrees with the restriction of an element 
$x^{\ad}:\bS\ra \rG^{\ad}_\RR$ of 
$(\dX^{\ad})^+$.

From here on in the proof, we always consider $\rG$ as acting on $V_\CC$ via $\iota$ and $\SL_{2,\RR}\times \rT_\RR$ as
acting on $V_\CC$ via $\iota\circ f$.
In what follows we study the weights of certain actions of $\mathrm{U}_1$ on $V_\CC$.\\ 

{\it{Step 1: $\mathrm{U}_1$ acts on $V_\CC$ via $\iota\circ f\circ j$ with weights in $\{ 0,\pm 1\}$.}}

To prove this, consider any element $x\in \dX^+$ 
that projects to $x^{\ad}\in (\dX^{\ad})^+$.
Let $d>0$ be an integer such that $x\circ m_d$ lifts to a map $\widehat{x\circ m_d}:\bS\ra \SL_{2,\RR}\times \rT_{\RR}$ (where $m_d:\mathrm{U}_1\ra \mathrm{U}_1$ is the $d$'th power map). Denote by 
$\coch:\bS\ra \rT_\RR$ the projection of $\widehat{x\circ m_d}$ to the
second factor, and by
$\xi_d=(j\circ m_d,\coch_1):\mathrm{U}_1\ra \SL_{2,\RR}\times \rT_\RR$ the restriction of $\widehat{x\circ m_d}$ to $\mathrm{U}_1$. Note that $\coch$ and $\coch_1$ are independent of $x$.

We know, by inspecting a single element of $\dX(\psi)$,
that the restriction to $\mathrm{U}_1$ of any element 
$\bS\ra\GSp(\psi)$
of $\dX(\psi)$ acts on $V_\CC$ with weights $\{\pm 1\}$.
Since the composition $\iota\circ f\circ\xi_d$
is the restriction to $\mathrm{U}_1$ of
the $d$'th power of an element in $\dX(\psi)$, 
the action via $\xi_d$ has weights $\{\pm d\}$.

Now, the action $\xi_{-d}=(j\circ m_{-d},\coch_1)$
on $V_\CC$ is conjugate to $\xi_d$, and so it
also has weights $\{\pm d\}$.
(This follows from the fact that
$\Ad_J\circ j=j\circ m_{-1}:\mathrm{U}_1\ra\SL_{2,\CC}$,
where 
$J\in \SL_2(\CC)$
is the matrix whose entries are $0$ on the diagonal and $i$
away from the diagonal, and so
$\Ad_{\iota(f(J))(x|_{\mathrm{U}_1}}\circ\, m_d)=x|_{\mathrm{U}_1}\circ m_{-d}$.)

The homomorphisms $\xi_d$ and $\xi_{-d}$ commute, meaning that $\xi_d(z)\cdot \xi_{-d}(z)=\xi_{-d}(z)\cdot\xi_d(z)\in \SL_{2,\RR}\times \rT_\RR$.
 Thus the weights of 
 the action $\xi_d\cdot (\xi_{-d})^{-1}$
 are in $\{0,\pm 2d\}$, because they are
  the sum of the weights of the action $\xi_d$ and $\xi_{-d}$.
But now $\xi_d\cdot (\xi_{-d})^{-1}=(j\circ m_{2d},0)$,
and so $\mathrm{U}_1$ acts via $\iota\circ f\circ j$
with weights in $\{0,\pm 1\}$, as desired.\\

{\it{Step 2: Decomposing $V$ as an $\SL_2$-module.}}

Denoting $V_1$ the standard two-dimensional representation of $\SL_2$, recall that all irreducible representations of~$\SL_2$ are of the form $V_\ell\coloneqq \Sym^\ell V_1$ for some integer $\ell\ge 0$. Since the set of weights of $\mathrm{U}_1$ acting on $V_\ell$ contains $\ell$, we conclude that, as a representation of $\SL_2$ acting via $\iota\circ f$,  
an irreducible subrepresentation of $V$
must be isomorphic either to $V_0$ or to $V_1$.

We may therefore write $V=V'\oplus V''$ as a direct sum of $\QQ$-representations, where $V'$
(non-canonically) decomposes as a $\psi$-orthogonal sum $V'=\bigoplus_{i=1}^m V'_i$
with $V'_i\cong V_1$ as an $\SL_2$ representation, 
and $V''$ is a trivial representation. Note that $m\geq 1$, because the homomorphism 
$\SL_2\xrightarrow{f} \rG\xrightarrow{\iota} \GSp(\psi)$ is not constant. It is clear that $V',V''$ are both symplectic under $\psi$ and 
$\psi$-orthogonal to each other.\\

{\it{Step 3: $\rT$ acts on $V'$ via scalars.}}

First, since $\rT$ commutes with $\SL_2$, it must preserve $V'$ as an $\SL_2$-isotypic space. 
Next, in Step 2 we showed that the weights of 
$\xi_d$ and of $\xi_{-d}$ are precisely $\{\pm d\}$,
and that the weights of $j\circ m_d$ are contained in $\{0,\pm d\}$. 

Next, we want to show that $\coch_1$ acts on $V'_\CC$
with weight $0$. Let $k$ be a weight for the action of $\coch_1$ on $V'_\CC$.
Since $\coch_1$ and $j\circ m_d$ commute, we can
find $v\in V'_\CC$ that is an eigenvector
for $j\circ m_d$ of weight $d$,
and for $\coch_1$ of weight $k$.
Then $\xi_d$ and $\xi_{-d}$ act on $\CC v$ with weights $d+k$ and $-d+k$, respectively, which implies
that $k=0$, so that $\coch(\mathrm{U}_1)$ acts trivially on $V'_\CC$.

Moreover, by examining $\dX(\psi)$ we see that $\widehat{x\circ m_d}(\GG_{m,\RR})$ acts via scalars on $V_\CC$ and hence also on $V'_\CC$, and thus so does $\coch(\GG_{m,\RR})=\widehat{x\circ m_d}(\GG_{m,\RR})$. 
Since we assumed that $\dX$ does not factor through a 
closed normal
$\QQ$-subgroup of $\rG$ other than $\rG$ itself, we conclude that 
$\coch(\bS)$ is 
Zariski dense
in $\rT_\RR$. This implies that
$\rT_\RR$ acts on $V'_\CC$
via scalars, and so does $\rT$ on $V'$, 
which completes the proof of Step 3.\\

{\it{Step 4: construction of the morphism $\Phi$.}}

Finally, let $U=V'_1$ and $W=V''\oplus\left(\bigoplus_{i=2}^m V_i'\right)$. 
By Step 2 and Step 3 above, $\SL_2\times \rT$ acts on $V$
preserving the mutually orthogonal subspaces $U$ and $W$, and thus so does~$\rG$.
It follows that
$\rG\hra \GSp(\psi)$ factors as
$$\rG\hra \GSp(\psi|_U,\psi|_W)\hra \GSp(\psi),$$ and hence 
$$(\rG,\dX)\hra (\GSp(\psi|_U,\psi|_W),\dX(\psi|_U,\psi|_W))\hra (\GSp(\psi),\dX(\psi)).$$ This map of Shimura data yields a map of Shimura varieties 
$\Phi:\ASh_1(N)\times\ASh_{g-1}(N)\ra \ASh_g$ for sufficiently large $N$, and the proof is complete.
\end{proof}

We are now ready to prove the theorem, implementing the original idea. 
\begin{proof}[Proof of \Cref{thm:dmcNCShimura}]
Assume, for contradiction, that $Z\subsetneq \Shi$ is a compact Hodge-generic subvariety of $\Shi$, of dimension at least $g$. 

Consider the morphism $\Phi:\ASh_1(N)\times\ASh_{g-1}(N)\rightarrow\ASh_g$ 
and the component $\cA_1(N)^+$
constructed in \Cref{lm:Eisog}, so that for any $[E]\in\cA_1(N)^+$ the intersection $\Shi\cap \Phi\left([E]\times\cA_{g-1}(N)\right)$ is nonempty.

Since $\dim Z+\dim \cA_{g-1}(N)\ge \dim \cA_g$ and since $\cA_1(N)^+$ is non-compact, by \Cref{cor:KU}(ii) it follows that $Z$ must also be non-compact, contradicting our hypotheses.
\end{proof}

\subsection{Compact special subvarieties of 
$\cA_{g}$ determine $\dmc(\cA_{g})$
}\label{sec:mdsp}

Denote by $\mdsp(g)$ the {\it{maximal dimension of a compact special subvariety of 
$\cA_{g}$}}.
Since 
$\cA_{0}$
is a point, we have $\mdsp(0)=0$.
Moreover, if 
$\Shi_1\subseteq\cA_{g_1}$ and $\Shi_2\subseteq\cA_{g_2}$ 
are compact
special subvarieties of largest dimension, 
then $\Shi_1\times \Shi_2\subseteq 
\cA_{g_1}\times\cA_{g_2}\ra\cA_{g_1+g_2}
$
maps to a compact special subvariety of dimension $\dim \Shi_1+\dim \Shi_2$,
and so 
$$
\mdsp(g_1+g_2)\geq \dim \Shi_1+\dim \Shi_2=\mdsp(g_1)+\mdsp(g_2).
$$

Combining \Cref{prop:productShimura} and \Cref{thm:dmcNCShimura}
leads to the following exact expression for the maximal dimension of a compact subvariety of
$\cA_{g}$.

\begin{prop}\label{thm:finalestimate}
    The maximal dimension of a compact subvariety of 
    $\cA_{g}$
    is
    $$\dmc(
    \cA_{g}
    )= \max\left(\mdsp(g),\max_{0\le g'<g} \left(g-g'-1+\mdsp(g')\right)\right). $$
\end{prop}
\begin{proof}
For $g=0$ we have 
$\dmc(\cA_{0})=0$.
Assume now $g\geq 1$.

For every $g'\in[0,g-1]$
let $\Shi_{g'}\subseteq
\cA_{g'}
$ be a compact special subvariety of maximal dimension, that is $\dim \Shi_{g'}=\mdsp(g')$,
and let $Z_{g-g'}\subseteq
\cA_{g-g'}
$ a compact subvariety of dimension $g-g'-1$ (for example, a complete intersection). Then $\Shi_g$ and
$\Shi_{g'}\times Z_{g-g'}$ for $g'=0,\dots,g-1$
are compact subvarieties of 
$\cA_{g}$.
Thus $\dmc(
\cA_{g}
)\geq \max\left(\mdsp(g),\max_{0\le g'<g} \left(g-g'-1+\mdsp(g')\right)\right)$.
We have now to show that equality holds.

Let $Z\subsetneq
\cA_{g}
$ be an irreducible compact subvariety and let $\Shi\subseteq
\cA_{g}
$ be the smallest special subvariety containing~$Z$. Since 
$\dmc(\cA_{g})\ge\dmcg(\cA_{g})=g-1$,
it is enough to deal with the case $\dim Z\ge g$, in which case from \Cref{thm:dmcNCShimura} it follows that $\Shi\subsetneq
\cA_{g}$. 

If $\Shi$ is compact, then $\dim Z\leq \dim \Shi\leq \mdsp(g)$ and the theorem holds. 

So assume from now on that $\Shi$ is not compact and let $(\rG,\dX)\hra\SD_g$ be the embedding
of Shimura data associated to $\Shi$.
It follows that $(\rG^{\ad},\dX^{\ad})$ has a non-compact factor $(\rG^{\ad}_1,\dX^{\ad}_1)$.  Then we can apply \Cref{prop:productShimura} to conclude that $(\rG,\dX)\hra \SD_g$ must factor through $\SD_{g_1,g_2}$ for some positive integers $g_1+g_2=g$. 
As a consequence, $Z$ is the image of a compact subvariety of 
$\cA_{g_1}(N)\times\cA_{g_2}(N)$
via a morphism 
$\ASh_{g_1}(N)\times\ASh_{g_2}(N)\ra\ASh_{g}$
for some $N$, which is a Hecke translate of the standard map.

By induction we see that
\begin{align*}
    \dim Z&\leq 
    \dmc(\cA_{g_1}) + \dmc(\cA_{g_2})
    \\
    &\leq \max\! \Big(\mdsp(g_1)+\mdsp(g_2), \max_{g_1'<g_1}(g_1-g'_1-1+\mdsp(g'_1)+\mdsp(g_2)),\\&
    \phantom{\leq \max\Big(}
     \max_{g_2'<g_2}(g_2-g'_2-1+\mdsp(g_1)+\mdsp(g'_2)), \\ 
      & \phantom{\leq \max\Big(}
     \max_{\substack{g'_1<g_1, \ g_2'<g_2}}(g-g'_2-g'_1-2+\mdsp(g'_1)+\mdsp(g'_2) \ \Big) \\
    &\leq \max\left(\mdsp(g), \max_{0\le g'<g}(\mdsp(g')+g-g'-1)\right)\, ,\\ 
\end{align*}
as desired.
\end{proof}

Given \Cref{thm:finalestimate}, the only missing ingredient to
prove \Cref{thm:dmcAg} is to determine the 
maximal dimension of a compact
special subvariety of 
$\cA_{g}$,
whenever
this is at least $g-1$. \Cref{sec:decoupled} and \Cref{sec:non-decoupled} below will be dedicated to this task, and we first recall some more machinery for dealing with compact Shimura varieties.

\section{Compact factors in compact Shimura data}\label{sec:compact}

Let $\rG=\Res_{\KK/\QQ} \rH$ be a simple adjoint $\QQ$-group, 
where $\rH$ is a geometrically simple group over a totally real field $\KK$.

Note that the $\QQ$-group $\rG$ is anisotropic if and only if the $\KK$-group $\rH$ is anisotropic,
and recall that $\rG_\RR\cong \prod_{\si} \rH_{\si,\RR}$, where $\si$
ranges over all embeddings $\si:\KK\hra\RR$.

\begin{fact}
If $\rG_\RR$ has a compact factor $\rH_{\si,\RR}$ for some $\si$,
then $\rG$ is anisotropic.
\end{fact}
\begin{proof}
Since $\rH_{\si,\RR}$ is compact, it contains no nontrivial unipotent elements.
The same then holds for $\rH$, which is thus anisotropic, and so is $\rG$.
\end{proof}

In this section we study the following:

\begin{qu}\label{qu:compact}
Suppose that the $\QQ$-group $\rG$ is anisotropic. Must $\rH_{\si,\RR}$
be compact for some embedding $\si:\KK\hra\RR$?
\end{qu}

We shall restrict our analysis to 
the $\KK$-groups
$\rH=\Aut(V,\eta)^{\ad}$, where $V$ is a (finitely generated) left module for a division algebra $\DD$ with center $\LL$
endowed with an involution $*$
(satisfying $(\alpha\beta)^*=\beta^* \alpha^*$),
such that $\KK$ is the fixed field in $\LL$,
and $\eta$ is a non-degenerate form on $V$
that is $\epsilon$-Hermitian with respect to $*$ (where $\epsilon=\pm1$).

We say that {\it{$\eta$ represents $0$}} (or, equivalently,
that {\it{$\eta$ is isotropic}}) if
there exists $v\in V\setminus\{0\}$ such that $\eta(v,v)=0$.

\begin{rem}
As a motivation for \Cref{qu:compact}, we anticipate that
in \Cref{sec:decoupled} we will begin
to investigate the dimensions of compact special subvarieties $\Shi$ of 
$\cA_{g}$.
Each of these $\Shi$ will be induced by a symplectic representation
of a reductive $\QQ$-group. The basic, but most important case,
is when their adjoint groups are exactly those $\rG$
mentioned above.
To have a sharp estimate for the dimension of the special subvariety $\Shi$
we must know whether a factor $\rH_{\si,\RR}$ must necessarily be compact.
\end{rem}

We start with the following lemma which is certainly well-known, but we gather it for convenience of the reader:

\begin{lm}\label{lm:anisotropicif0}
 Let $\rH=\Aut(V,\eta)^{\ad}$ be a $\KK$-group as above, 
 and assume that $\rH$ is geometrically simple
 (and in particular nontrivial). Then $\rH$ is isotropic if and only if $\eta$ represents $0$.
\end{lm} 

\begin{proof}
Suppose first that $\rH$ is isotropic, so that there is a nontrivial unipotent element $h\in \rH(\KK)$.   
Since $h$ is unipotent, there exist two linearly independent vectors $v,w\in V$ such that $h(v)=v$ and  $h(w)=w+v$. But then 
    $0=\eta(h(w),h(v))-\eta(w,v)=\eta(v,v)$, and so $\eta$ represents $0$ as desired.

    Now suppose that $\eta$ represents $0$. Then there exists a nonzero vector $v\in V$ such that $\eta(v,v)=0$. Let $U=\DD v$, pick an element $w'\not\in U^{\perp}$, and by scaling assume that $\eta(w',v)=1$. Let $w=w'+dv$ with $d=-\eta(w',w')/2 \in \DD$, so that still $\eta(w,v)=1$.
    Noting that $d=\epsilon d^*$, we have $\eta(w,w)=\eta(w',w')+d+\epsilon d^*=0$.
Then since $\eta$ is non-degenerate on $W\coloneqq\mathrm{Span}(v,w)$, 
    it follows that $V\cong W\oplus W^{\perp}$ is an orthogonal decomposition. 
    
    Now we obtain a split 
$\KK$-torus
$\rT\subset \Aut(V,\eta)$ by considering
    the automorphisms of $V$ that
    rescale $w,v$ by $t,t^{-1}$ and leave $W^{\perp}$ fixed. 
    If $\rT$ is not contained in the center of $\Aut(V,\eta)$, then
    $\rH=\Aut(V,\eta)^{\ad}$ contains a nontrivial split torus and so is isotropic.

    Suppose now that $\rT$ is contained in the center of $\Aut(V,\eta)$.
Then every element of $\Aut(V,\eta)$ must preserve $\DD v$ and $\DD w$, and likewise any isotropic subspace. 
We claim that $W^{\perp}=\{0\}$. 
Indeed, for any $r\in W^{\perp}$ 
and for $d=-\eta(r,r)/2\in\DD$
the vector $w+r+dv$ is isotropic, and so each element of $\Aut(V,\eta)$ must preserve $\DD(w+r+dv)$. But our torus $\rT$ preserves $\DD(w+r+dv)$ only if $r=0$.

Hence $V\cong \DD^2$, and our form is  equivalent to $\eta((x_1,y_1),(x_2,y_2))=x_1\alpha y_2^*+\epsilon y_1\alpha^* x_2^*$ for some $0\neq \alpha\in \DD$. Now $(x,1)$ is isotropic whenever $(x\alpha)=-\epsilon (x\alpha)^*$. 
Since $\DD(x,1)$ is not preserved by our torus $\rT$ for any $x\neq 0$, the vector $(x,1)$ cannot be isotropic, and so the equation $(x\alpha)=-\epsilon(x\alpha)^*$ has no nonzero solutions.
This
means that $\epsilon =1$ and $*$ must be trivial. But in this case we have  $\DD=\KK$ and $\rH\cong \mathrm{O}^{\ad}_{2,\KK}$ is the trivial group. The proof is finished. 
\end{proof}

\subsection{Hasse principles}\label{subsec:Hasse}

Let $\KK$ be a totally real field and let $\rH=\Aut(V,\eta)^{\ad}$ be 
a simple adjoint $\KK$-group as at the beginning of \Cref{sec:compact},
with $\eta$ non-degenerate.

Consider the following properties of $\eta$:
\begin{itemize}
    \item[\LGI]
{\it{$\eta$ satisfies the local-to-global principle for isotropy}},
namely $\eta$ is isotropic if and only if $\eta$ becomes isotropic over every completion of $\KK$;
\item[\naI{}]
{\it{there is no non-archimedean obstruction to the isotropy of $\eta$}},
namely $\eta$ is isotropic over all non-archimedean completions of $\KK$.
\end{itemize}

If $\eta$ satisfies \LGI{} 
and \naI{} above, then the group $\rH$ satisfies the following
\begin{itemize}
    \item[\AnC{}]
$\rH$ is anisotropic if and only if there exists an embedding $\si:\KK\hra \RR$
such that $\rH_{\si,\RR}$ is compact.
\end{itemize}

In light of \Cref{lm:anisotropicif0}, to check whether a reductive group 
with adjoint group
$\rH=\Aut(V,\eta)^{\ad}$ is anisotropic it is sufficient to check whether $\eta$ represents $0$. We thus record Hasse principles for all types of forms, summarizing what is in \cite{scharlau}, thus providing an almost complete answer to \Cref{qu:compact}.

    \begin{itemize}
        \item[(q)]
        Let $V$ be a vector space over $\KK$ and let $\eta$ be a symmetric bilinear form on $V$.
        Then $\eta$ satisfies \LGI{} above,
see \cite[6.6.5]{scharlau}. 
        Moreover, if $\dim V\geq 5$, then $\eta$ also satisfies \naI{},
see \cite[6.6.6(vii)]{scharlau}.
        
        This implies that {\it{property \AnC{} holds for $\SO(V,\eta)$,
        if $V$ is a vector space over $\KK$ of $\dim_{\KK}V\geq 5$
and        $\eta$ is a symmetric bilinear form on $V$}}.\\

\item[(uF)] Let $\LL/\KK$ be 
a quadratic extension with involution $*$.
Let $V$ be a vector space over $\LL$
and let $\eta$ be a Hermitian form on $V$ with respect to $*$.
Then $\eta$  is isotropic if and only if its trace bilinear form is (see also \cite[10.1.1]{scharlau}).
        Hence we can reduce to case (q) above. So $\eta$
        satisfies \LGI{}.       
Moreover, if $\dim_{\LL}V\geq 3$ (which means $\dim_{\KK}V\geq 6$)
        then $\eta$ also satisfies \naI{}.

        This implies that {\it{property \AnC{} holds for
        $\SU(V,\eta)$, if $V$ is a vector space over $\LL$ of $\dim_{\LL}V\geq 3$
        and $\eta$ is a Hermitian form on $V$}}.\\

          \item [(uH)] Let $\HH$ be a quaternion division algebra over $\KK$ with canonical involution $*$, and let $(V,\eta)$ be a Hermitian space over $(\HH,*)$.
        This case is treated similarly to (uF), see also \cite[10.1.8(iii)]{scharlau}: 
        the form $\eta$ is isotropic if and only if its
        trace form is, and so it satisfies \LGI{}.
        Moreover, if $\dim_{\HH}V\geq 2$ (which means $\dim_{\KK}V\geq 8$),
        then $\eta$ also satisfies \naI{}.

 This implies that {\it{property \AnC{} holds for
        $\SU(V,\eta)$, if $V$ is a vector space over $\HH$ of $\dim_{\HH}V\geq 2$
        and $\eta$ is a Hermitian form on $V$}}.\\

        \item [(uD)] Let $\LL/\KK$ be a quadratic extension, and let $\DD$ be a non-commutative division algebra over $\LL$ with an involution $*$ extending 
        the nontrivial automorphism of $\LL$. Let $(V,\eta)$ be a Hermitian space over $(\DD,*)$.
        Then $\eta$ satisfies \LGI{}, see \cite[10.6.2]{scharlau}.
Moreover, over a non-archimedean prime, for every dimension, 
there are only two classes of non-degenerate
Hermitian forms, and they are distinguished by their determinant \cite[end of page 275]{scharlau}.
It follows that, if $n=\dim_{\DD} V\geq 3$, then $(V,\eta)$ is isometric to
$(\DD^n,\eta')$, where $\eta'$ is the Hermitian form induced by the diagonal matrix
with 
diagonal
entries
$(1,-1,-a,1,1,\dots,1)$, and $a$ is the determinant
of the matrix that represents $\eta$ with respect to any basis of $V$.
Since $\eta'$ is manifestly isotropic, we conclude that $\eta$ satisfies \naI{} whenever $\dim_{\DD}V\geq 3$.

        This implies that {\it{property \AnC{} holds for
        $\SU(V,\eta)$, if $V$ is a vector space over $\DD$ of $\dim_{\DD}V\geq 3$ and
        $\eta$ is a Hermitian form on $V$}}.      \\

        \item [(skH)] 
        Let $\HH$ be a quaternion division algebra over $\KK$ with canonical involution $*$, and let $(V,\eta)$ be a skew-Hermitian space over $(\HH,*)$. 

        If $\dim_\HH V\geq 3$, then $\eta$ satisfies property \LGI{}, see \cite[10.4.1]{scharlau}.
        If $\dim_\HH V\geq 4$, then $\eta$ also satisfies property \naI{}, see \cite[10.3.6]{scharlau}.

This implies that {\it{property \AnC{} holds for
        $\SU(V,\eta)$, if $V$ is a vector space over $\HH$ of $\dim_{\HH}V\geq 4$ and
        $\eta$ is a skew-Hermitian form on $V$}}.  
\end{itemize}

\section{Decoupled embeddings into Siegel Shimura data}\label{sec:Satake}

The purpose of this section is to recall Satake's classification of embeddings 
of Shimura data
into Siegel Shimura data, whose associated complex
representations are particularly simple, namely such that they satisfy
the decoupling condition of \Cref{dfx:decoupled} below
(which is precisely
Condition (9) in \cite[Section 7]{satake:analytic}).

Such classification is of fundamental importance:
indeed, in \Cref{sec:decoupled} and \Cref{sec:non-decoupled}
we will see that compact special subvarieties
of 
$\cA_{g}$
of maximal dimension (assuming that this is at least $g-1$) are associated
to decoupled representation.

\subsection{Embedding of Hermitian symmetric spaces}

Given an embedding $(\rG,\dX)\hra (\GSp(\psi),\dX(\psi))$ of Shimura data with associated
symplectic representation $V$, we may write $\wt{G}_\RR=\prod_{i=1}^k \wt{\rG}_i$ as a product of simple real groups, and ask for the irreducible representations $U_i$ of 
$\wt{\rG}_{i,\CC}$ that occur as $\wt{\rG}_{i,\CC}$-invariant subspaces of $V_\CC$.
Note that each non-compact $\wt{\rG}_i$ gives rise to a Hermitian symmetric space $\dX_i$. Since
$\dX$ is the product of such $\dX_i$,
its dimension equals the sum of the dimensions
of the $\dX_i$'s. 
 The pairs $(\wt{\rG}_i,U_i)$ have been classified 
 in \cite{satake-compact}.
In \Cref{table:milneslist}, we present all the possible non-compact $\wt{\rG}_i$, only recording the groups up to isogeny.

\begin{table}[H]
\tiny
\begin{tabular}{ | c | c | c |c|}
  \hline  \textbf{Non-compact} & \textbf{Hermitian} & \textbf{Repr.} & \textbf{Non-self-dual/}\\
\textbf{real group $\wt{\rG}_i$} & \textbf{symmetric} &  &  \textbf{/Symplectic/}\\
\textbf{(up to isogeny)} & \textbf{space $\dX_i$} &  $\dim U_i$ & \textbf{/Orthogonal}\\
  \hline\hline 
$\SL_{2}$ & $\begin{array}{c}\SL_2(\RR)/\SO_2(\RR)
  \\
  \dim=1
  \end{array}$  & $2$ & Symp\\  
\hline\hline
$\begin{array}{c}\\
\SU_{p,n-p}\\
n\geq 3,\quad 1\le p\le \frac{n}{2}
\end{array}$ & 
$\begin{array}{c}
\SU_{p,n-p}/\mathrm{S}(\mathrm{U}_p\times\mathrm{U}_{n-p})
  \\
  \dim=p(n-p)
  \end{array}$ & $n$ & NSD \\  
\hline
$\SU_{1,n-1}$, $n\geq 3$ & 
$\begin{array}{c}
\SU_{1,n-1}/\mathrm{S}(\mathrm{U}_1\times\mathrm{U}_{n-1})\\
\dim=n-1
\end{array}$
& 
$\begin{array}{c}\\
\displaystyle\binom{n}{c}
\\
\\
c\in [2,n-2]
\end{array}$
& $\begin{cases} c\neq \frac{n}{2} & \textrm{NSD}\\ c=\frac{n}{2}\textrm{ even} & \textrm{Orth }\\  c=\frac{n}{2}\textrm{ odd} & \textrm{Symp }  \end{cases}$\\  
\hline
\hline
$\SO^*_{2r}$, $r\geq 5$ & 
$\begin{array}{c}
\SO^*_{2r}(\RR)/\mathrm{U}_r\\
\dim=\frac{r(r-1)}{2}
\end{array}$
& $2r$ & Orth\\
\hline 
\hline  
$\Sp_{r}$, $r\geq 2$ & $\begin{array}{c}
\Sp_{r}(\RR)/\mathrm{U}_r\\
\dim=\frac{r(r+1)}{2}
\end{array}$ & $2r$ & Symp\\
\hline
\hline  
$\SO_{2p-2,2}$, $p\geq 3$ 
& 
$\begin{array}{c}
\SO_{2p-2,2}/\mathrm{S}(\mathrm{O}_{2p-2}\times\mathrm{O}_2)\\
\dim=2p-2
\end{array}$ & $2^{p-1}$ & $\begin{cases}p\equiv 2(4) & \textrm{Symp}\\ p\equiv 0(4) & \textrm{Orth}\\  p\equiv 1,3(4) &  \textrm{NSD}\end{cases}$\\
\hline 
$\SO_{2p-1,2}$, $p\geq 2$ & 
$\begin{array}{c}
\SO_{2p-1,2}/\mathrm{S}(\mathrm{O}_{2p-1}\times\mathrm{O}_2)\\
\dim=2p-1
\end{array}$ & $2^p$ & $\begin{cases}p\equiv 0,3(4) &\textrm{Orth}\\ p\equiv 1,2(4) & \textrm{Symp}\end{cases}$\\ 
\hline
\end{tabular}\\

\medskip

\caption{Irreducible Hermitian symmetric spaces of non-compact type that map nontrivially to some Siegel upper half-space}\label{table:milneslist}
\end{table}

\subsection{Decoupled embeddings of Shimura data}

In this subsection we focus on embeddings
of Shimura data into some Siegel Shimura datum
associated to symplectic representations 
that are decoupled
(in the sense of \Cref{dfx:decoupled} below).
We will describe an almost complete classification
of such embeddings obtained by Satake in \cite[Section 8]{satake:analytic}
and we condense it in  \Cref{table:Satakelist}, which
can be viewed a refinement of \Cref{table:milneslist}.
While \Cref{table:milneslist} will suffice for most of our purposes, it will turn out that for case (I) in \Cref{sec:best},
which is the most important for us,
we will need \Cref{table:Satakelist}

\begin{dfx}\label{dfx:decoupled}
A $\CC$-representation of a complex reductive group $\rH$
is called \textit{decoupled} if every irreducible subrepresentation of $\rH^{\der}$
factors through a simple quotient of $\rH^{\der}$.
A $\QQ$-representation $V$ of a reductive $\QQ$-group $\rG$
is called decoupled if $V_{\CC}$ is a decoupled representation of $\rG_{\CC}$.
\end{dfx}

In this section we will focus on the following problem.

\begin{qu}\label{qu:satake}
    Consider all embeddings of $(\rG,\dX)$ 
    into Siegel Shimura data corresponding to
    a $\GSp$-irreducible decoupled representation $V$,
    and assume that $\dim \dX>0$.
    Which $\QQ$-groups $\rG$ and
    which irreducible linear subrepresentations of $V_\CC$
    occur?
\end{qu}

A complete answer to \Cref{qu:satake}
was provided by Satake in 
\cite[Section 8]{satake:analytic}, building upon
his previous work \cite{satake:imb}.
See also 
\cite[Section 1.3]{deligne}, \cite[Section IV.6]{satake:book} and \cite[Chapter 10]{milne2005}.\\

For the rest of this section we consider an embedding of
$(\rG,\dX)$ into some Siegel Shimura datum
corresponding to an $\GSp$-irreducible decoupled representation $V$,
as in the hypotheses of \Cref{qu:satake}.

Note preliminarily that the restriction of $V$
to $\rG^{\der}$ must factor through a $\QQ$-simple
quotient of $\rG^{\der}$. Hence
the groups $\rG$ appearing
in \Cref{qu:satake} must have
$\rG^{\der}$ that is $\QQ$-simple.

Hence, by \Cref{cor:totallyrealsplit} we may write $\wt{\rG}\cong \rH_{\KK/\QQ}$ for a totally real field $\KK$ and a geometrically simple $\rH$ over $\KK$. Since $V$ is decoupled and irreducible we also have $V\cong W_{\KK/\QQ}$ for a $\GSp$-irreducible $\KK$-representation $W$ of $\rH$. 
We remark that,
by \Cref{cor:totallyrealsplit},
the simple factors of $\wt{\rG}_\RR$
are isomorphic to $\rH_{\si,\RR}\coloneqq \rH\otimes_{\KK,\si}\RR$,
for some embedding $\si:\KK\hra \RR$.
Thus \Cref{table:Satakelist} is a refinement of \Cref{table:milneslist}.

Let $U\subset W_\CC$ denote an irreducible linear subrepresentation 
of $\rH_\CC$.
Below we first list the various simple groups $\rH^{\der}$ that can occur (up to isogeny), and then summarize that information in a table, together with the possibilities for $U$. 
As for our conventions, we denote
\begin{itemize}
\item
by $\KK$ a totally real number field;
\item 
by $\LL$  a totally imaginary quadratic extension of $\KK$;
 \item
by $\DD$ a division algebra of degree $\delta$ over $\LL$ with involution $*$, 
with fixed field $\KK$ inside $\LL$;
\item 
by $\HH$ a quaternion division algebra over $\KK$ equipped with its canonical involution; 
\item 
by $\eta$ is a non-degenerate quadratic form on $\DD^m$ or $\HH^m$, or $\KK^m$, which may be symmetric, alternating,  Hermitian, or skew-Hermitian.
\end{itemize}

Note that, upon identifying $\DD\otimes_{\LL}\CC$ with
$\mathcal{M}_{\delta,\delta}(\CC)$, we obtain $\CC$-linear projections $\DD\otimes_{\LL}\CC\ra \CC^\delta$
(for example, taking the first column of the matrix).
Thus we also obtain projections
$\DD^r\otimes_\LL\CC\ra \CC^\delta\otimes\CC^r$ that are  $\SU(\DD^r,\eta)$-equivariant, and so $\SU(\DD^r,\eta)$
has a {\it{standard}} representation on $\CC^\delta\otimes\CC^r$.

\begin{table}[H]
\tiny
\resizebox{1 \textwidth}{!}{
\begin{tabular}{ |c|c | c  | c |c|c|}
  \hline 
  &\textbf{Group $\rH^{\der}$} & \textbf{Non-compact}  &  
  \textbf{Irreducible}
  & \textbf{Non-self-dual/}\\
 &\textbf{over $\KK$} & \textbf{Hermitian} &  
 \textbf{complex linear} &   \textbf{/Symplectic/}\\
  &\textbf{(up to isogeny)} & \textbf{symmetric space}  & 
  \textbf{representation $U$} & \textbf{/Orthogonal}\\
  \hline\hline 
\begin{tabular}{c}
  (A$_1$)
\end{tabular} &  
  \begin{tabular}{c}
  $\textrm{Nm}_1(\HH)$\end{tabular} & 
  \begin{tabular}{c}
  $\SU_{1,1}/\mathrm{S}(\mathrm{U}_1\times\mathrm{U}_{1})$ \\ $\dim=1$
  \end{tabular} & 
  \begin{tabular}{c} (dual) standard $\CC^2$\\ $\dim=2$\end{tabular} &
  Symp\\  
\hline 
\hline
  \begin{tabular}{c}
  (D$_4$)
\end{tabular} &  
  \begin{tabular}{c}
  ??? \\

  \end{tabular} & 
  \begin{tabular}{c}
  $\SO^*_8/\mathrm{U}_4=\SO_{6,2}/(\SO_6\times\SO_2)$
 \\ $\dim=6$
  \end{tabular} & 
  \begin{tabular}{c} standard \\ $\dim=8$\end{tabular} &
  Orth\\  
\hline \hline
\begin{tabular}{c}
  (I)$_{p,r\delta-p}$\\
  $r\delta\geq 3$
\end{tabular} &  
  \begin{tabular}{c}
  $\SU(\DD^r,\eta)$\\ $\eta$ Hermitian\end{tabular} & 
  \begin{tabular}{c}
  $\SU_{p,r\delta-p}/\mathrm{S}(\mathrm{U}_p\times\mathrm{U}_{r\delta-p})$ \\ $\dim=p(r\delta-p)$
  \end{tabular} & 
  \begin{tabular}{c} (dual) standard $\CC^{r\delta}$\\ $\dim=r\delta$\end{tabular} &
  NSD\\  
\hline 
\begin{tabular}{c}
  (I')$_{r\delta-1,1}^c$\\
  $r\delta\geq 4$\\
  $2\leq c\leq r\delta-2$\\
\end{tabular} & 
\begin{tabular}{c}
$\SU(\DD^r,\eta)$\\ $\eta$ Hermitian\end{tabular} & 
\begin{tabular}{c}
$\SU_{r\delta-1,1}/\mathrm{S}(\mathrm{U}_{r\delta-1}\times\mathrm{U}_1)$ \\
$\dim=r\delta-1$
  \end{tabular}  & \begin{tabular}{c}$\Lambda^c\CC^{r\delta}$ \\ 
  $\dim=\binom{r\delta}{c}$\end{tabular}
& $\begin{cases} c\neq \frac{r\delta}{2} & \textrm{NSD}\\ c=\frac{r\delta}{2}\textrm{ even} & \textrm{Orth }\\  c=\frac{r\delta}{2}\textrm{ odd} & \textrm{Symp }  \end{cases}$\\  
\hline
\hline \begin{tabular}{c}
  (II)$_r$ \\ 
  $r\geq 2$ with $r\neq 4$ \end{tabular} &
\begin{tabular}{c}
$\SU(\HH^r,\eta)$ \\
$\eta$ skew-Hermitian
\end{tabular}
& \begin{tabular}{c}
$\SO^*_{2r}/\mathrm{U}_r$\\
$\dim=\frac{r(r-1)}{2}$
\end{tabular}
&
\begin{tabular}{c}
standard\\
$\dim=2r$
\end{tabular} 
&
Orth
\\
\hline \hline
\begin{tabular}{c}
  (III.1)$_r$ 
  \\ $r\geq 2$
\end{tabular} &
\begin{tabular}{c}
$\Sp(\KK^{2r},\eta)$ \\
$\eta$ symplectic
\end{tabular}
& \begin{tabular}{c}
$\Sp_{r}/\mathrm{U}_r$\\
$\dim=\frac{r(r+1)}{2}$
\end{tabular}
&
\begin{tabular}{c}
standard\\
$\dim=2r$
\end{tabular} &
Symp
\\
\hline \begin{tabular}{c}
  (III.2)$_r$ 
    \\ $r\geq 2$
\end{tabular} &
\begin{tabular}{c}
$\SU(\HH^r,\eta)$ \\
$\eta$ Hermitian
\end{tabular}
& \begin{tabular}{c}
$\Sp_r/\mathrm{U}_r$\\
$\dim=\frac{r(r+1)}{2}$
\end{tabular}
&
\begin{tabular}{c}
standard\\
$\dim=2r$
\end{tabular} &
Symp
\\
\hline \hline
\begin{tabular}{c}
  (IV.1)$_{2p}$ 
  \\ $p\geq 3$, $p\neq 4$
\end{tabular} &
\begin{tabular}{c}
$\SO(\KK^{2p},\eta)$ \\
$\eta$ symmetric
\end{tabular}
& \begin{tabular}{c}
$\SO_{2p-2,2}/(\SO_{2p-2}\times\SO_2)$\\
$\dim=2p-2$
\end{tabular}
&
\begin{tabular}{c}
spin${^\pm(W)}$, $W$ isotropic\\
$\dim=2^{p-1}$
\end{tabular} 
&
$\begin{cases}p\equiv 2(4) & \textrm{Symp}\\ p\equiv 0(4) & \textrm{Orth}\\  p\equiv 1,3(4) &  \textrm{NSD}\end{cases}$
\\
\hline
\begin{tabular}{c}
  (IV.1)$_{2p+1}$ 
     \\ $p\geq 2$
\end{tabular} &
\begin{tabular}{c}
$\SO(\KK^{2p+1},\eta)$ \\
$\eta$ symmetric
\end{tabular}
& \begin{tabular}{c}
$\SO_{2p-1,2}/(\SO_{2p-1}\times\SO_2)$\\
$\dim=2p-1$
\end{tabular}
&
\begin{tabular}{c}
spin${(W)}$, $W$ isotropic\\
$\dim=2^p$
\end{tabular} 
&
$\begin{cases}p\equiv 0,3(4) &\textrm{Orth}\\ p\equiv 1,2(4) & \textrm{Symp}\end{cases}$
\\
\hline \begin{tabular}{c}
(IV.2)$_{2r}$ 
\\ $r\geq 3$ with $r\neq 4$
\end{tabular} &
\begin{tabular}{c}
$\SU(\HH^r,\eta)$ \\
$\eta$ skew-Hermitian
\end{tabular}
& \begin{tabular}{c}
$\SO_{2r-2,2}/(\SO_{2r-2}\times\SO_2)$\\
$\dim=2r-2$
\end{tabular}
&
\begin{tabular}{c}
spin$^{\pm}{(W)}$, $W$ isotropic\\
$\dim=2^{r-1}$
\end{tabular} 
&
$\begin{cases}r\equiv 2(4) & \textrm{Symp}\\ r\equiv 0(4) & \textrm{Orth}\\  r\equiv 1,3(4) &  \textrm{NSD}\end{cases}$
\\
\hline \end{tabular}
}\\

\medskip

\caption{Satake's classification} \label{table:Satakelist}
\end{table}

Satake's list (see \Cref{table:Satakelist}) is then the following:
\begin{itemize}
    \item[(A$_1$)]
    As we will treat them separately, we isolate out all groups of type A$_1$. They all arise as the norm 1 units in a quaternion algebra $\HH$ over $\KK$, and $U\cong\CC^2$ is the standard representation of $\SL_{2,\CC}$.
    \item[(I)]
    $\rH^{\ad}=\SU(\DD^r,\eta)^{\ad}$
    with $\eta$ Hermitian (so that $\rH_\RR^{\ad}=\SU_{p,r\delta-p}^{\ad}$),
    and $U\cong \CC^\delta\otimes\CC^r$ is the standard
    representation or its dual
\item[(I')]
$\rH^{\ad}=\SU(\DD^r,\eta)^{\ad}$
    with $\eta$ Hermitian with  
    $\rH_\RR^{\ad}=\SU_{1,r\delta-1}^{\ad}$, and $U\cong \Lambda^c(\CC^\delta\otimes\CC^r)$ is the
    $c$'th exterior power of the standard representation, or its dual
    (the case with reversed signature is analogous)
\item[(II)]
$\rH^{\ad}=\SU(\HH^r,\eta)^{\ad}$ 
with $\eta$ skew-Hermitian
(so that $\rH_\RR^{\ad}=(\SO^*_{2r})^{\ad}$,
where $\SO^*_{2r}=\SO_{2r,\CC}\cap \mathrm{U}_{r,r}$),
and $U\cong \CC^2\otimes\CC^r$ is the standard representation
or its dual
\item[(III.1)]
$\rH^{\ad}=\Sp(\KK^{2r},\eta)^{\ad}$ with $\eta$ symplectic
(so that $\rH_\RR^{\ad}=\Sp_r(\RR)^{\ad}$),
and $U\cong\CC^{2r}$
\item[(III.2)]
$\rH^{\ad}=\SU(\HH^r,\eta)^{\ad}$ with 
$\eta$ Hermitian (so that $\rH_\RR^{\ad}=\Sp_r^{\ad}$),
and $U\cong\CC^2\otimes\CC^r$ is standard
\item[(IV.1)]
$\rH^{\ad}=\SO(\KK^r,\eta)^{\ad}$ with $\eta$ symmetric
of signature $(2,r-2)$
(so that $\rH^{\ad}_\RR=\SO_{2,r-2}^{\ad}$);
the representation is respectively the 
spin representation $U\cong \Lambda L$ or its dual
for $r$ odd, or the odd/even spin representation
$U\cong \Lambda^{\pm} L$ or its dual for $m$ even,
where $L\subset (\RR^{2,r-2}\otimes_\RR\CC)$ is a certain maximal isotropic subspace
(the case with reversed signature is analogous)
\item[(IV.2)]
$\rH^{\ad}=\SU(\HH^r,\eta)^{\ad}$  with $\eta$ skew-Hermitian and
$\rH^{\ad}_\RR$ isomorphic to $\SO_{2,2r-2}^{\ad}$;
the representation $U$ is as in case (IV.1).
\item[(D$_4$)] 
$\rH^{\ad}$ is of $D_4$-type. Here the situation is more complicated and Satake does not give a classification, but this will not be necessary for us.
\end{itemize}

In what follows we will sometimes
consider the cases (I)$_{p,r\delta-p}$ as
subcases of the same case (I)$_n$ with $n=r\delta$
for all $p$,
and the cases (I')$^c_{r\delta-1,1}$
as subcases of the same case (I')$_n$ for all $c$.

\section{Decoupled Achievable pairs}\label{sec:decoupled}

In this section we study compact special subvarieties of 
$\cA_{g}$ 
associated to decoupled symplectic
representations and we determine those of largest dimension (provided that such dimension
is at least $g-1$). The key ingredients will be
Satake's classification recalled in \Cref{sec:Satake} and the Hasse principles 
illustrated in \Cref{sec:compact}.

\begin{dfx}
A pair $(d,g)$ is \textit{achievable} if
there exists a symplectic representation $(V,\psi)$ of
a compact Shimura datum $(\rG,\dX)$
such that $\dim V=2g$ and $\dim \dX=d$.
If moreover $V$ is decoupled as a $G$-representation, then the pair $(d,g)$ is said \textit{d-achievable}.
If $V$ is $\GSp$-irreducible, then the pair $(d,g)$ is
{\it{indecomposably (d-)achievable}}.
\end{dfx}

Equivalently, $(d,g)$ is achievable if 
$\cA_{g}$
contains a compact special subvariety $\Shi$ of dimension $d$,
and it is indecomposably achievable if 
the embedding $(\rG,\dX)\hra\SD_g$ of Shimura data
associated to $\Shi$ is indecomposable.

We first remark the following.

\begin{lm}\label{lm:g+1}
    If $(d,g)$ is (d-)achievable,
    then $(d,g+1)$ is (d-)achievable too.
    \end{lm}
    \begin{proof}
        Let $\iota:(\rG,\dX)\hra\SD_g$ be an
    embedding of a compact Shimura datum $(\rG,\dX)$
    with $d(\rG)=d$,
associated to an irreducible decoupled symplectic representation $V$.
Let $\tau\in\Sieg^{\pm}_1$ be (a lift of) a special point of 
$\cA_{1}$
and denote by $\rT_\tau\subsetneq\GSp_2$ the stabilizer of $\tau$. If $\rG'=\{(t,g)\in \rT_\tau\times \rG\,|\,\chi(t)=\chi(\iota(g))\}$ and 
$\dX'=\{\tau\}\times \dX$,
then $(\rG',\dX')$ is a compact Shimura datum, that
comes naturally with
an embedding into $\SD_{1,g}\hra \SD_{g+1}$.
Finally note that, if $V$ is decoupled, then the
symplectic representation $V'=\QQ^2\oplus V$ associated
to $(\rG',\dX')$ is decoupled too.
\end{proof}

We recall that there exist (Hodge-generic) $d$-dimensional compact subvarieties of 
$\cA_{g}$
for every $d<g$. Hence we call a pair $(d,g)$ {\em negligible} if $d<g-1$.
Moreover, we introduce a  partial ordering on $\NN\times\NN$ by declaring
   $$(d,g)\preceq (d', g')\iff  d\le d'\ {\rm  and}\ g\ge g',$$
If $(d,g)\preceq(d',g')$, we will say that $(d,g)$ is {\em dominated} by $(d',g')$; 
in this case, $(d,g)$ is {\it{strictly dominated}} by $(d',g')$ if furthermore $d<d'$.

\subsection{Dominating indecomposably d-achievable pairs}\label{ssc:key}

For this subsection we assume that 

\begin{itemize}
    \item 
$(V,\psi)$ is an irreducible, decoupled, symplectic representation of $(\rG,\dX)$, with $\rG^{\ad}$ anisotropic. 
\end{itemize}

We are interested in achievable pairs $(d,g)$ that can be obtained from such symplectic representations. 

Replacing $\rG$ by its image in $\GSp(\psi)$, we can assume that 
$\wt{\rG}$
is $\QQ$-simple. It follows that $\wt{\rG}\cong \rH_{\KK/\QQ}$ 
for some totally real field $\KK$
and some geometrically simple $\KK$-group $\rH$,
and that the representation $V|_{\wt{\rG}}$, obtained from $V$ by restricting to $\rG^{\der}$
and then composing with $\wt{\rG}\ra \rG^{\der}$,
is of type
$V|_{\wt{\rG}}\cong \Res_{\KK/\QQ} W$ for some $\GSp$-representation $W$ over $\KK$ of $\rH$.

Moreover,
$\wt{\rG}_\RR=\prod_{i=1}^k \wt{\rG}_i$,
where $\wt{\rG}_i\coloneqq \rH_{\si_i,\RR}$ for 
$\{\si_1,\dots,\si_k\} = \Hom(\KK,\RR)$,
and all $\wt{\rG}_i$ are $\RR$-forms of the same complex group $\rH_{\CC}$.
Consequently, being $V$ decoupled, its complexification decomposes as $V_\CC=\oplus_{i=1}^k V_i$, where the action of 
$\wt{\rG}_\CC$
on each~$V_i$ factors through~$\wt{\rG}_{i,\CC}$.

Let $U_i$ denote a nontrivial $\wt{\rG}_{i,\CC}$-invariant irreducible subspace of $V_i$.
By \Cref{lm:m} all linear summands of $V_\CC$ are Galois conjugate to $U_i$ or to its $\GSp$-dual $U_i^\vee(\chi)$. It follows that each $V_i$ is (non-canonically) isomorphic
to the direct sum of $m(V_i)$ such linear summands.
Since $G$ is $\QQ$-simple, the Galois group 
$\Gal(\CC/\QQ)$
permutes all the summands $V_i$ of $V_\CC$, and thus
$m=m(V_i)$ is independent of $i$.

\begin{rem}\label{rmk:m2}
The $\wt{\rG}_{i,\CC}$-representation $V_i$ is symplectic
but $U_i$ need not be.
If $U_i$ is symplectic, then each $V_j$ is (non-canonically) isomorphic to 
$U^{\vs_1}_i\oplus\dots\oplus U^{\vs_m}_i$, for suitable elements 
$\vs_1,\dots,\vs_m$ of the Galois group.
If $U_i$ is not symplectic, then $m\geq 2$ is even and each $V_j$
is (non-canonically) isomorphic to $(U_i^{\vee}(\chi)^{\vs_1}\oplus U_i^{\vs_1})\oplus\dots\oplus
(U_i^{\vee}(\chi)^{\vs_{m/2}}\oplus U_i^{\vs_{m/2}})$,
for suitable elements $\vs_1,\dots,\vs_{m/2}$ of the Galois group.
\end{rem}

The following estimate will play a key role in the determination of the largest
compact special subvariety of 
$\cA_{g}$ of dimension larger than $g-1$.

\begin{lm}\label{lm:key-estimate}
With notation as above, depending on the type of $\rH$, the resulting indecomposably d-achievable pair $(d,g)$ 
is either negligible or dominated by one of the  pairs given in \Cref{table:d-achiev} below. 
Moreover, the pair $(1,2)$ in the (A$_1$) case
and the pairs $\large((k-1)\lceil\frac{n}{2}\rceil\cdot\lfloor\frac{n}{2}\rfloor, nk\large)$
in the case (I$_n$) with $n\geq 3$ and $k\geq 2$ are indecomposably d-achievable.
\end{lm}

\begin{table}[H]
\begin{tabular}{ |c | c |}
 \hline  \textbf{Type of group $\rH$}  & \textbf{Non-negligible} \\
  \textbf{and representation}  & \textbf{dominating pairs $(d,g)$} \\
\hline\hline (A$_1$) & $(1,2)$ 
\\
\hline (D$_4$) & Negligible \\ 
\hline (I)$_n$, with $n\ge 3$ & $\prec\Big((k-1)\lceil\frac{n}{2}\rceil\cdot\lfloor\frac{n}{2}\rfloor, nk\Big)$ with $k\geq 2$  \\  
\hline (I') & Negligible \\  
\hline (II)$_r$ with $r\geq 4$ & $\preceq \Big((k-1)\frac{r(r-1)}{2},2rk\Big)$ with $k\geq 2$ \\  
\hline (III)$_r$ with $r\geq 2$ & $\preceq \Big((k-1)\frac{r(r+1)}{2}, 2rk\Big)$ with $k\geq 2$\\
\hline (IV) & Negligible  \\ 
\hline
\end{tabular}\\

\medskip

\caption{List of possible non-negligible dominating pairs}\label{table:d-achiev}
\end{table}

\begin{rem}\label{rem:A-dominates}
As \Cref{table:d-achiev} shows,
the case (III)$_3$
is strictly dominated by (I)$_6$;
moreover all the other (II)$_r$ and (III)$_r$
cases are strictly dominated by the case (I)$_{2r-1}$.
\end{rem}

\begin{rem}\label{rem:suboptimal}
For $3\leq g\leq 15$, the pair $(g-1,g)$ is not 
$d$-achievable, but it is for $g=2$ (see \Cref{ssc:SL2} below).
\end{rem}

We now give the proof of \Cref{lm:key-estimate}, by a case-by-case analysis.

\subsection{Proof of \Cref{lm:key-estimate}}\label{sec:estimates}

We separately analyze every case occurring
in \Cref{table:Satakelist}.

Note that Satake's classification of complex irreducible representations arising
from embeddings of Hermitian symmetric spaces inside the Siegel upper half-space
in \Cref{table:milneslist} is enough for all but one case.
To treat case (I) in \Cref{sec:best},
which is also the most important one for our purposes,
we have to use
the more accurate \Cref{table:Satakelist}.

\subsubsection{Case (A$_1$)}\label{ssc:SL2}

We know that $\dim U_i=2$ for all $i$. Hence $g=mk$, and $d$ is the number of $i$ such that 
$\wt{\rG}_i$ is non-compact, hence $d\leq k$. Now, if $m=1$, then it follows that $\dim W=2$ and hence $\rH\cong \SL_{2,\KK}$, which is not anisotropic. Hence $m\geq 2$ and so $d\leq \frac{g}{2}$.
The only potentially
non-negligible case is then $(d,g)=(1,2)$.

{\it{Existence for $(d,g)=(1,2)$.}} There are indeed compact quaternionic Shimura curves in 
$\cA_{2}$
(see \cite{hashimoto1995}, for instance).
Thus $(d,g)=(1,2)$ is indecomposably d-achievable.

\subsubsection{Case (D$_4$)}

By the sixth line of \Cref{table:milneslist}
the representation $U_i$, 
corresponding to $\rG_i$ isogenous to 
$(\SO_{6,2})_\RR$,
is not symplectic (it is in fact orthogonal) and $8$-dimensional.
Hence $m\ge 2$, and so  $\dim V_i\ge 16$. Thus the indecomposably d-achievable pairs are dominated by $(6k,8k)$, which are negligible.

\subsubsection{ Case (I)}\label{sec:best}

Consider the case (I)$_{p,n-p}$, with the standard representation, which is not self-dual.

Suppose first that 
$\wt{\rG}(\RR)$
has at least one compact factor and look at the second line of \Cref{table:milneslist}.
We have $d\leq (k-1)\cdot\lceil\frac{n}{2}\rceil\cdot\lfloor\frac{n}{2}\rfloor$. Moreover $m\geq 2$, and so $g\geq nk$.
Thus $(d,g)$ is dominated by $((k-1)\cdot\lceil\frac{n}{2}\rceil\cdot\lfloor\frac{n}{2}\rfloor, nk)$, which
is indecomposably d-achievable as we show below.

Assume now that 
$\wt{\rG}(\RR)$
has no compact factors. By \Cref{lm:anisotropicif0}, and according to 
cases (uF) and (uD) in \Cref{subsec:Hasse}, this can only occur if our group $\rH$ is isogenous to $\SU(\DD^r,\eta)$, where $\DD$ is a division algebra of degree $\delta\geq 2$ over a CM extension $\LL/\KK$, with involution $\vs$ extending  the nontrivial automorphism of $\LL/\KK$ and $\eta$ is a Hermitian form on $\DD^r$ with $r\delta=n$ and $r\leq 2$. 
In this case, 
as a complex 
$\SU(\DD^r,\eta)_\CC$-representation,
$(\DD^r)_\CC\cong \mathcal{M}_{\delta,\delta}(\CC)^{\oplus r}\cong (\CC^\delta\otimes\CC^r)^{\oplus\delta}$ has $\delta$ complex irreducible summands isomorphic to the standard
$\CC^\delta\otimes\CC^r$. Moreover its endomorphism
algebra as a complex 
$\SU(\DD^r,\eta)_\CC$-representation
contains (and therefore is equal to) $\DD$. Hence $\DD^r$ is irreducible as a $\KK$-representation
and not self-dual, which implies $m\geq 2\delta$. Thus the best possible achievable pair here is dominated by $\Big(k\cdot\lceil\frac{n}{2}\rceil\cdot\lfloor\frac{n}{2}\rfloor,\delta nk\Big)$. 

\begin{claim}\label{claim:F}
The case of 
$\wt{\rG}(\RR)$
without compact factors
does no better than the above case of 
$\wt{\rG}(\RR)$
with one compact factor, and the two estimates agree exactly when $\delta=2$ and $k=1$, and so $n=r\delta\in\{2,4\}$.
\end{claim}

\begin{proof}
Note first that the function $F(n)\coloneqq\lceil\frac{n}{2}\rceil\cdot\lfloor\frac{n}{2}\rfloor$
satisfies $\frac{n^2-1}{4}\leq F(n)\leq \frac{n^2}{4}$.
The case with non-compact factors potentially produces
pairs $P'_0(s,\delta)=(s F(\delta),s\delta^2)$ for $r=1$
and $P''_0(s,\delta)=(sF(2\delta),2s\delta^2)$ for $r=2$,
where $\delta\geq 2$ and $s\geq 1$.
We want to compare such pairs with the pairs
$P_1(k,n)=((k-1)F(n),kn)$ obtained in the case with $1$ compact factor for $k\geq 2$.

Consider first the case $r=1$.
For $s\delta=2$, namely $(s,\delta)=(1,2)$, we obtain
$P'_0(1,2)=(1,4)=P_1(2,2)$.
For $s\delta>2$, we have $s\delta^2\geq 8$
and the resulting pair $P'_0(S,\delta)$
is strictly dominated by
$P_1(k,n)$ with $k=2$ and 
$n=\lfloor \frac{s\delta^2}{2}\rfloor$.
Indeed, we have
$kn\leq s\delta^2$ and
$(k-1)F(n)\geq \frac{n^2-1}{4}
\geq
\frac{((s\delta-1)^2/4-1)^2}{4}>\frac{s^2\delta^2}{4}\geq
\frac{s\delta^2}{4}\geq sF(\delta)$.

Consider now the case $r=2$.
For $(s,\delta)=(1,2)$, we obtain
$P''_0(1,2)=(4,8)=P_1(2,2)$.
For $s\delta>2$
the resulting pair $P''_0(s,\delta)$ is strictly
dominated by $P_1(k,n)$ with $k=2$ and $n=s\delta^2$.
Indeed, we have $kn=2s\delta^2$ and
$(k-1)F(n)\geq \frac{n^2-1}{4}
\geq \frac{(s\delta^2)^2-1}{4}>s\delta^2=sF(2\delta)$.

This completes the proof of the claim.
\end{proof}

For 
$\wt{\rG}(\RR)$
without compact factors,
the desired conclusion follows from \Cref{claim:F}.
Indeed, if $\delta=2$ and $k=1$, we have $n=r\delta\in\{2,4\}$: in these cases we achieve the pairs $(1,4)$ and $(4,8)$, both of which are negligible.\\

{\it{Existence for $(d,g)=((k-1)\cdot\lceil\frac{n}{2}\rceil\cdot\lfloor\frac{n}{2}\rfloor, nk)$.}}
Let $\KK$ be a totally real field of degree $k$ over $\QQ$, let $\EE/\KK$ be a totally
imaginary quadratic extension, which comes with a natural involution,
and let $\rH\coloneqq \textrm{GU}(\EE^n,\eta)$, where $\eta$ is a Hermitian form on $\EE^{n}$.
By \Cref{claim:signature} below, the form $\eta$ can be chosen to have
signature $(\lceil\frac{n}{2}\rceil,\lfloor\frac{n}{2}\rfloor)$
at all but one 
infinite
place of $\KK$, and definite signature at the remaining 
infinite 
place $\si$. 

We give 
$\rG\coloneqq \mathrm{Res}_{\KK/\QQ}\rH$ 
the Shimura datum $\dX$ 
consisting of the conjugacy class of $x:\bS\ra \rG_\RR=\rH_{\si_1,\RR}\times\dots\times\rH_{\si_k,\RR}$
defined
as follows. 
Fix an infinite place $\si_i:\KK\hra\RR$
and note that $\si_i$ extends to an embedding $\EE\hra\CC$.
Let $\EE_i\coloneqq\si_i(\EE)$ and let $\eta_i$ be the hermitian form
on $\EE_i^n$ induced by $\eta$.
Pick an orthogonal decomposition  of
$\EE_i$-spaces
as $\EE_i^n\cong P_i\oplus N_i$ such that
$\eta_i$ is positive definite on $P_i$ and negative definite on $N_i$. 
Then define $x$
so that $x(z)$ acts as $z$ on $P_{i,\RR}$ and as~$\ol{z}$ on $N_{i,\RR}$ 
for all $z\in\bS$
and for $i=1,\dots,k$. 
This determines a Shimura datum, and makes $\EE^{n}$ of type $\{(-1,0),(0,-1)\}$. We then obtain the desired symplectic embedding via \cite[Lemma 10.15]{milne2011shimura}.\\

Here we recall that the following elementary fact was used in the above
construction.

\begin{claim}[Classical]\label{claim:signature}
Let $\EE/\KK$ be as above.
For every $p_1,\dots,p_k\in \{0,1,\dots,n\}$,
there exists a Hermitian form $\eta$ on $\EE^n$ that has signature
$(p_i,n-p_i)$ at the $i$-th infinite place $\si_i:\KK\hra\RR$.
\end{claim}

\begin{proof}
Since $\KK$ has dense image in $\KK\otimes_\QQ \RR\cong\RR^k$, and since $\KK^n$ thus has dense image in $(\RR^n)^k$,
there exists $(\lambda_1,\dots,\lambda_n)\in\KK^n$
whose image $(\si_i(\lambda_1),\dots,\si_i(\lambda_n))$ in
$\RR^n$ has $p_i$ positive entries and $n-p_i$ negative entries for each 
embedding $\si_i:\KK\hra\RR$.
It is then enough to take as $\eta$
the Hermitian form represented by a diagonal matrix
with entries $(\lambda_1,\dots,\lambda_n)$.
\end{proof}

\subsubsection{Case (I')}

Suppose we are in case (I')$^c_{n-1,1}$, 
and consider the third line of \Cref{table:milneslist}.
The representation is $\Lambda^c \CC^n$ and $d\leq (n-1)k$. 

Suppose first that the representation is not symplectic. Then $m\geq 2$,
and so $g\geq k\frac{m}{2}\binom{n}{c}\geq k\binom{n}{2}=k\frac{n(n-1)}{2}\geq d\cdot\frac{n}{2}\geq2d$. Since
$n\geq 4$, we have $g\geq k\frac{n(n-1)}{2}\geq 6$, and so the achieved pair is negligible.

Suppose now that
the representation is symplectic. Then $c=\frac{n}{2}$ is odd, so $n\geq 6$. Here we have $g\geq \frac{k}{2}\cdot \binom{n}{\frac{n}{2}}\geq 10$, and so $g/d\geq 
\frac{1}{2(n-1)}\binom{n}{\frac{n}{2}}\geq 2$.

Thus we only get negligible pairs.

\subsubsection{Case (II)}

Suppose we are in case (II)$_r$, so that the group $\rH$ is isogenous to $\SU(\HH^r,\eta)$ with $r\geq 2$ 
and $\eta$ a skew-Hermitian form. 
Since the representation is not symplectic, we have $m\geq 2$.

Suppose $r=2$. Then $d\leq k$ and $g\geq 4k$.
Hence $(d,g)$ is dominated by $(k,4k)$, which is negligible.

Suppose $r=3$. Then $d\leq 3k$ and $g\geq 6k$. Hence our pairs are dominated by $(3k,6k)$, which is negligible.

Now suppose $r\geq 4$. 
By \Cref{lm:anisotropicif0}, and according to 
case (skH) in \Cref{subsec:Hasse},
the group 
$\wt{\rG}(\RR)$
can have at most $(k-1)$ 
non-compact factors. Hence the resulting pairs are dominated by  $((k-1)\frac{r(r-1)}{2},2rk)$, as desired.

\subsubsection{Case (III)}

In case (III.1) the group is isotropic, so this case does not occur.

Consider now the case (III.2), in which the group  $\rH$ is isogenous to $\SU(\HH^r,\eta)$ with $\eta$ Hermitian. We assume $r\geq 2$, as $r=1$ gives a form of $\SL_2$ and is handled above. 
By \Cref{lm:anisotropicif0}, and according to 
case (uH) in \Cref{subsec:Hasse},
the fact that $\rH$ is anisotropic implies
that 
$\wt{\rG}(\RR)$
has at least one compact factor.
Now, if $W$ were irreducible, then 
$\rH_\RR$ and $\Sp(W)_\RR$ would have the same dimension
and so the map $\rH\ra\Sp(W)^{\ad}$ would be an isogeny,
which would therefore imply that $\rH$ is not anisotropic.
Hence $m\geq 2$ and
the resulting pair is dominated by
$\Big((k-1)\frac{r(r+1)}{2}, 2rk\Big)$.

\subsubsection{Case (IV)}

In all these cases $\rH^{\ad}$ is a form of $\SO^{\ad}_n$ for $n\geq 5$. 

From 
the last two lines of
\Cref{table:milneslist}
we have
$g\geq mk \cdot 2^{\lfloor\frac{n-3}{2}\rfloor}$.
Moreover,
by \Cref{lm:anisotropicif0} and by 
(q) and (skH) in \Cref{subsec:Hasse},
the group 
$\wt{\rG}(\RR)$
must have at least 1 compact factor, and so $d\leq (k-1)(n-2)$.

Consider the case $n=5$.
We have $d\leq 3(k-1)$ and $g\geq 2km$.
If $W$ is irreducible,
then both $\rH_\RR$ and $\Sp(W)_\RR$ would have dimension $10$; thus $\rH\ra\Sp(W)^{\ad}$ would be an isogeny
and so $\rH$ would be isotropic.
Hence $W$ is not irreducible and so $m\geq 2$,
which implies that $g\geq 4k$. Thus the resulting pair is negligible.

Assume now $n=6,7,8$. The representation is not symplectic and so $m\geq 2$.
For $n=6$ we have $g\geq 2mk\geq 4k$ and $d\leq 4(k-1)$.
For $n=7$ we have $g\geq 4mk\geq 8k$ and $d\leq 6(k-1)$.
For $n=8$ we have $g\geq 4mk\geq 8k$ and $d\leq 8(k-1)$.
So these cases produce negligible pairs.

Assume finally $n\geq 9$, so that
$2^{\lfloor\frac{n-3}{2}\rfloor}\geq n-2$.
It follows that $g\geq k(n-2)$ and so the resulting
pairs are negligible.\\

The proof of \Cref{lm:key-estimate} is complete.

\subsection{Hodge-generic subvarieties and
compact special subvarieties arising from decoupled representations} 

In view of the statement
of \Cref{thm:dmcAg},
for every $g\geq 1$ we define the quantity
$$
\dmax(g)\coloneqq \max\left(g-1,\left\lfloor \frac{\lfloor g/2\rfloor ^2}{4}\right\rfloor\right)= \begin{cases}
          \ \ g-1 & \mbox{if } g<16\,; \\
          \ \ \left\lfloor\tfrac{g^2}{16}\right\rfloor & \mbox{for even } g\ge 16\,;\\
          \left\lfloor\tfrac{(g-1)^2}{16}\right\rfloor & \mbox{for odd } g\ge 17\,.
        \end{cases}
$$

First, we have some elementary numerical
property of the function $\dmax$.

\begin{lm}\label{lm:dmax}
For all $1\leq g_1\leq g_2$ we have
\[
\dmax(g_1+g_2)\geq \dmax(g_1)+\dmax(g_2).
\]
Moreover, the above inequality is strict unless $g_1=1$ and $g_2\geq 16$ is even.
\end{lm}
\begin{proof}
Suppose first that $g_1=1$.
Then $\dmax(1+g_2)\geq \dmax(g_2)$ and the inequality is strict unless $g_2\geq 16$ and is even.

Now assume $g_1\geq 2$.

Suppose that $g_1+g_2\leq 15$.
Then $\dmax(g_1+g_2)=g_1+g_2-1> (g_1-1)+(g_2-1)=\dmax(g_1)+\dmax(g_2)$.

Suppose now $g_1+g_2\geq 16$.

By direct computation, $\dmax(14)=13$, $\dmax(15)=14$, $\dmax(16)=\dmax(17)=16$
and $\dmax(18)=\dmax(19)=20$.
It is then easy to check that, for every even $g'\geq 2$ smaller than $g_1$, we have $\dmax(g_1)-\dmax(g_1-g')\leq \dmax(g_2)-\dmax(g_2-g')$.
This implies that it is enough to verify the statement for $g_1=2$ or $g_1=3$.

By inspection,
for $g_1=2$ we have $\dmax(2+g_2)\geq 2+\dmax(g_2)$ and so $\dmax(2+g_2)>\dmax(2)+\dmax(g_2)$.
Similarly, for $g_1=3$ we have $\dmax(3+g_2)\geq 3+\dmax(g_2)$ and so $\dmax(3+g_2)>\dmax(3)+\dmax(g_2)$.
\end{proof}

As a consequence of the analysis done in \Cref{ssc:key} and \Cref{sec:estimates},
we obtain the following bound for the d-achievable pairs. 

\begin{prop}\label{prop:estimate} 
For a d-achievable pair $(d,g)$ we have $d\leq \dmax(g)$,
and equality holds for $g=2$, and for $g\geq 16$.
Moreover, $(\dmax(g),g)$ is indecomposably d-achievable if and only if $g=2$, or if $g$ is even and $g\geq 16$.
\end{prop}
\begin{proof}
Let us first check the indecomposably d-achieved pairs.
Here we invoke \Cref{lm:key-estimate}.
In particular,
we observe that by \Cref{rem:A-dominates}
we only have to consider only the $(I)_n$ case with $n\geq 3$ and $k\geq 2$, which gives the indecomposably d-achievable pairs  
  $$(d,g)=\left((k-1)\left\lceil\frac{n}{2}\right\rceil\cdot
  \left\lfloor\frac{n}{2}\right\rfloor,  kn\right)\, .$$ 

By \Cref{table:d-achiev} we see that, for $g=2$, the only indecomposably d-achievable pair is $(1,2)$, obtained in the (A$_1$) case.
For $3\leq g\leq 15$ in the (I$_n$) case
we obtain the pairs 
$(2,6)$, $(4,8)$, 
$(6,10)$, $(9,12)$, $(12,14)$ for $k=2$,
the pairs
$(4,9)$, $(8,12)$, $(12,15)$ for $k=3$,
and the pairs
$(k-1,2k)$, $(2k-2,3k)$ for $k\geq 4$. 
All such pairs are negligible.

Assume from now on that $g\geq 16$.\\

{\it{Case $g=kn\geq 16$ even.}}
We have $d=(k-1)\left\lceil\frac{g}{2k}\right\rceil\cdot\left\lfloor\frac{g}{2k}\right\rfloor$, and so for $k=2$ we obtain $d=\dmax(g)$.

Since $\tfrac{g}{2k}=\tfrac{n}{2}$, the product $\left\lceil\frac{g}{2k}\right\rceil\cdot\left\lfloor\frac{g}{2k}\right\rfloor$ is simply equal to~$\tfrac{n^2}{4}$ if $n$ is even, and to $\tfrac{n^2-1}{4}$ if~$n$ is odd, and thus
\[
d\leq (k-1)\frac{n^2}{4}=(k-1)\frac{g^2}{4k^2}\,.\]
The derivative of this, as a function of~$k$, is equal to $\tfrac{2-k}{k^3}g^2$, and thus it is a strictly decreasing function for $k\ge 2$. Thus for $k\ge 3$ and $g\ge 16$ we obtain the estimate
\[ (k-1)\frac{g^2}{4k^2}\leq 2\frac{g^2}{4\cdot 3^2}=\frac{g^2}{18} <\frac{g^2-4}{16}
\leq \left\lceil \frac{g}{4}\right\rceil\cdot\left\lfloor \frac{g}{4}\right\rfloor=\dmax(g).
\]
Hence the achieved pairs are strictly dominated by the case $k=2$.

{\it{Case $g=kn\geq 17$ odd.}}
Then
\[
d=\frac{k-1}{4}\left(\frac{g^2}{k^2}-1\right),
\]
which, as a function of $k$,
has derivative $-\frac{1}{4}-\frac{k-2}{4k^3}g^2$,
and so is again decreasing for $k\geq 2$.

Write $g=2\ell+1$ with $\ell\ge 8$ and 
note that $k\geq 3$ must be odd.
Then $d\le \frac{1}{2}\left(\frac{g^2}{9}-1\right)=\frac{(2\ell+1)^2}{18}-\frac{1}{2}$. 
Observe that $\frac{\ell^2}{4}-\frac{1}{4}\le \lfloor\frac{(g-1)^2}{16}\rfloor$
and $\frac{(2\ell+1)^2}{18}-\frac{1}{2}<\frac{\ell^2}{4}-\frac{1}{4}$ for $\ell\ge 8$. Hence, $d< \lfloor\frac{(g-1)^2}{16}\rfloor=\dmax(g)$ for $g\ge 17$ odd.\\

We thus conclude that, for $g\geq 16$, 
all indecomposably d-achievable pairs $(d,g)$ satisfy
$d\leq \dmax(g)$, and that
the pair
$(\dmax(g),g)$ is indecomposably
d-achievable only for $g=2n$ even,
in the (I)$_n$ case, with $k=2$.\\

Let now $(d,g)$ be an d-achievable pair,
which is not indecomposably d-achievable.
It corresponds to a compact special subvariety
of 
$\cA_{g}$
induced by an embedding
$(G,X)\hra\SD_g$ of Shimura data that factors
through $\SD_{g',g-g'}$ with $g'\leq g-g'$. This implies
that $d\leq \dmax(g')+\dmax(g-g')$.
By \Cref{lm:dmax}
we have $\dmax(g')+\dmax(g-g')\leq \dmax(g)$,
and equality $d=\dmax(g)$ can be attained
only for $g\geq 17$ odd and $g'=1$. This is
indeed the case by \Cref{lm:g+1}.
\end{proof}

The following corollary will be useful for the next section:

\begin{cor}\label{cor:decoupledboundsdg}
 Let $\rG$ be a simple $\QQ$-group with $\wt{\rG}_\RR=\prod_{i=1}^k \wt{\rG}_i$. Assume that
$\wt{\rG}(\RR)$
has at least one compact factor, and let $U$ be a 
$\CC$-irreducible decoupled representation of $\wt{\rG}_\CC$ occurring in \Cref{table:milneslist}. Then 
$d(\rG)\leq \dmax(k\dim U)$.

If equality holds, then $\rG$ is of type $(I_n)$ with $k=2$.
\end{cor}

\begin{proof}
This is a byproduct of the analysis carried over
in \Cref{sec:estimates}, when $\wt{\rG}_\RR$ has at most $k-1$ noncompact factors. 
Here we recall the relevant estimates,
keeping in mind \Cref{table:milneslist} and the exact same analysis as in \Cref{prop:estimate}.
\begin{itemize}
    \item[(A$_1$)]
We have $\dim U=2$, so $d(\rG)\leq k-1<\dmax(2k)$.
\item[(D$_4$)]
We have $\dim U=8$, so $d(\rG)\leq 6(k-1)<8k-1\leq \dmax(8k)$.
\item[(I$_n$)]
We have $\dim U=n$
and $d(\rG)\leq (k-1)\lceil\frac{n}{2}\rceil \lfloor\frac{n}{2}\rfloor\leq \dmax(kn)$.
By \Cref{prop:estimate}, equality can hold only if $k=2$.
\item[(I$'_n$)]
We have $\dim U=\binom{n}{c}\geq n$ and
$d(\rG)\leq (k-1)(n-1)<k\dim U-1\leq \dmax(k\dim U)$.
\item[(II$_r$)]
We have $\dim U=2r$ and
$d(\rG)\leq (k-1)\frac{r(r-1)}{2}< k\frac{k}{2} \frac{r^2-1}{2}\leq \dmax(2kr)$.
\item[(III$_r$)]
In case (III.1), the group $G$ is isotropic.
In case (III.2), we have
$\dim U=2r$ and $d(\rG)\leq (k-1)\frac{r(r+1)}{2}
< \frac{k^2 r^2-1}{4}\leq \dmax(2kr)$,
where the second inequality depends on
$k,r\geq 2$.
\item[(IV$_n$)]
We have $\dim U=2^{\lfloor \frac{n-1}{2}\rfloor}$,
and $d(\rG)\leq (k-1)(n-2)< k 2^{\lfloor \frac{n-1}{2}\rfloor}-1\leq \dmax(k\dim U)$,
since $n-2\leq 2^{\lfloor \frac{n-1}{2}\rfloor}$ for $n\geq 5$.
\end{itemize}
\end{proof}

In view of \Cref{prop:estimate} and \Cref{thm:finalestimate}, a compact subvariety of 
$\cA_{g}$
of maximal dimension
can be either a Hodge-generic compact subvariety (for example, a complete intersection),
or a compact special subvariety, or a product of the two types.
For all cases in which the special subvariety corresponds to a decoupled representation, we obtain the following result.
 
\begin{prop}\label{thm:dmcAg2}
Let $Z$ be a compact subvariety of 
$\cA_{g}$
of maximal possible dimension, which
is either Hodge-generic, or a special subvariety induced by a decoupled representation,
or a product of these two types.
Then $\dim Z=\dmax(g)$.
\end{prop}
\begin{proof} 
Recall that, by \Cref{thm:dmcNCShimura}, the maximal dimension of a compact Hodge-generic subvariety of 
$\cA_{g'}$
is $g'-1$. 
By \Cref{thm:finalestimate},
\Cref{lm:dmax} and \Cref{prop:estimate},
the bound $\dim Z\leq\dmax(g)$ is immediate
and in fact it is attained.
\end{proof}

\begin{rem}\label{rem:unique}
There are two interesting cases in which the optimal bound in \Cref{thm:dmcAg2}
is achieved in two different ways.
\begin{itemize}
    \item
For $g=2$ a compact curve in 
$\cA_{2}$
can be constructed \begin{itemize}
    \item 
as a Hodge-generic subvariety (e.g.~a component of a very general complete intersection), or
\item
as a special subvariety, as in \Cref{ssc:SL2}.
\end{itemize}
\item
For $g=17$ a compact subvariety of 
$\cA_{17}$
of largest dimension
can be constructed 
\begin{itemize}
    \item
again as a Hodge-generic subvariety (e.g.~a component of a very general complete intersection), or
\item
as the product of a (special or non-special) point 
$[E]\in
\cA_{1}
$ and a largest (16-dimensional) compact Shimura
subvariety $\Shi_{16}$ of 
$\cA_{16}$,
as in \Cref{sec:best}; or as (components of) Hecke translates of $[E]\times \Shi_{16}$
inside 
$\cA_{17}$.
\end{itemize}
\end{itemize}
In all the other cases, the construction described in the proof of \Cref{thm:dmcAg2} is essentially unique (see \Cref{rem:suboptimal}).
\end{rem}

\section{Compact special subvarieties of
$\cA_{g}$
\\ 
from non-decoupled representations}\label{sec:non-decoupled}

In this section we study the dimension of compact special subvarieties of 
$\cA_{g}$
arising from non-decoupled representations.
We show that such special subvarieties always have smaller dimension than the already constructed special subvarieties,
and so \Cref{thm:dmcAg2} indeed gives the maximal dimension of a compact subvariety of 
$\cA_{g}$.

\subsection{Setting}\label{sec:set8}

We want to investigate indecomposable embeddings
of compact Shimura data $(\rG,\dX)$ into a Siegel Shimura datum.
Hence we let $\rG$ be a reductive
algebraic group over $\QQ$ and 
$\wt{\rG}_i$ the geometrically simple factors of 
$\wt{\rG}_\RR$,
so that $\wt{\rG}_\RR=\prod_{i=1}^k \wt{\rG}_i$.

\begin{rem}\label{rem:one-compact}
    Since the embedding of $(\rG,\dX)$ in a Siegel Shimura
    datum is indecomposable and compact (as defined
    at the end of \Cref{sec:Shimura-varieties}), 
for every simple $\QQ$-factor 
    $\wt{\rG}_{(j)}$
    of $\wt{\rG}$,
    the real group $\wt{\rG}_{(j)}(\RR)$
must have at least one {\it{compact}}
    factor by \Cref{prop:productShimura}.    
\end{rem}

Let then $V$ be a $\GSp$-irreducible $\QQ$-representation of $\rG$ associated to an indecomposable embedding of $(\rG,\dX)$.
We recall that, as a $\wt{\rG}_\CC$ representation, $V_\CC$ 
has a canonical decomposition into a direct sum of isotypic components.
Each isotypic component can be written
as a direct sum of a certain number of irreducible subspaces $V_\alpha$:
note that, though such direct sum into irreducibles is non-canonical,
the number of summands is.
Hence, $V_\CC$
is (non-canonically)
isomorphic to a direct sum $\bigoplus_\alpha V_\alpha$ of
irreducible $\wt{\rG}_\CC$-representations $V_\alpha$.

\begin{rem}
Since $V$ is $\GSp$-irreducible,
 by \Cref{lm:m} all $V_\alpha$ are Galois conjugate to each other, or to each other's $\GSp$-duals. Moreover all $V_\alpha$ are self-dual or none of them is. Also, in the self-dual case, either all $V_\alpha$ are symplectic or all $V_\alpha$ are orthogonal.
\end{rem}

Furthermore, each $V_\alpha$ 
is isomorphic to
$V_{\alpha,1}\otimes\dots \otimes V_{\alpha,k}$, 
where $V_{\alpha,i}$ is an irreducible linear representation of $\wt{\rG}_{i,\CC}$: when this happens, we say that each $V_{\alpha,i}$ \textit{occurs in $V$} at the $i$-th position.

\subsubsection{An auxiliary decoupled representation}\label{ssc:aux}

Given $\rG$ and $V$ as above, we introduce
a $\wt{\rG}_\CC$-representation
$V'$, which will play an auxiliary role in
the proof of \Cref{lm:inefficient}.

As discussed above, we can choose a splitting 
$V_\CC=\bigoplus_\alpha V_\alpha$
and isomorphisms $V_\alpha\cong V_{\alpha,1}\otimes\dots\otimes V_{\alpha,k}$ for each $\alpha$.

Since each $\wt{\rG}_{i,\CC}$ is 
semisimple,
its $1$-dimensional
representations are trivial. Hence,
for each $\alpha$, we define the $\wt{\rG}_\CC$-representation $V'_\alpha\coloneqq \bigoplus_{i\in I_\alpha}V_{\alpha,i}$,
where $I_\alpha=\{i\,|\,\dim(V_{\alpha,i})>1\}$.

Our auxiliary decoupled representation is
$V'\coloneqq \bigoplus_\alpha V'_\alpha$.
Though the construction of $V'$ is not very natural,
its isomorphism class as a $\wt{\rG}_\CC$-representation is well-defined.\\

In order to compare the dimensions of $V$ and $V'$,
here we introduce some combinatorial quantities.

For a fixed $\alpha$ consider the multiset of numbers
$\{\dim V_{\alpha,1},\, \dots,\, \dim V_{\alpha,k}\}$, which is independent of $\alpha$
since $V$ is $\GSp$-irreducible (the same phenomenon holds for an irreducible linear representation of $\rG$), and let $N_V$ be the sub-multiset
consisting of numbers strictly greater than one.
Denoting by $\Prod(N)$ and $\Sum(N)$ respectively the product and the sum of the elements of a multiset $N$, for any $\alpha$ we have $\dim V_\alpha=\Prod(N_V)$ and $\dim(V'_\alpha)=\Sum(N_V)$.
Clearly $\Prod(N_V)\geq \Sum(N_V)$.

For computational purposes in the next two subsections
it will be useful to
distinguish two cases according to a certain property of $N_V$, which we define here below.

\begin{dfx}
A finite multiset $N$ of integers strictly greater than one 
is {\em inefficient}
if $\Prod(N)\geq 2\cdot \Sum(N)$, 
and {\it{efficient}} otherwise.
A representation $V$ of $\rG$, which is
irreducible or $\GSp$-irreducible,
is 
{\em (in)efficient} if $N_V$ is.
\end{dfx}

\subsection{Estimates for inefficient representations}

Dimension estimates for inefficient representations can be easily reduced to the decoupled case, and the following lemma shows that inefficient representations are never associated to a compact special subvariety of 
$\cA_{g}$

of maximal dimension.

\begin{lm} \label{lm:inefficient}
Let $(\rG,\dX)\hra \SD_g$ is an indecomposable embedding of Shimura data, with
associated $\GSp$-irreducible representation $V$. Assume that either
\begin{enumerate}
    \item $V$ is inefficient as a $\wt{\rG}$-representation, or
    \item The irreducible subrepresentations of $V_\CC$ are not symplectic. 
\end{enumerate}
Then $d(\rG)<\dmax(g)$, where $2g=\dim V$.
\end{lm}

\begin{proof}
Let $\wt{\rG}=\prod_{j=1}^\ell \wt{\rG}_{(j)}$ be
the decomposition of $\wt{\rG}$ into its
simple $\QQ$-factors.
Let $k_j$ denote the number of simple factors of $\wt{\rG}_{(j),\RR}$, so that $k=\sum_{j=1}^\ell k_j$ is the number
of simple factors of $\wt{\rG}_\RR$.

For every $j$, let $U_j$ be an irreducible nontrivial $\CC$-representation
of a simple factor of $\wt{\rG}_{(j),\CC}$, which occurs in $V$ as
defined in \Cref{sec:set8}.

\begin{claim}\label{claim:ineq}
We have $g\geq \sum_{j=1}^\ell k_j\cdot \dim U_j$.
\end{claim}

We will prove the above claim below.
Assuming \Cref{claim:ineq}, here we show how to conclude the proof.

By \Cref{rem:one-compact} each 
$\wt{\rG}_{(j)}(\RR)$
must have at least one compact factor, and so
$\dmax(k_j\dim U_j)\geq d(\wt{\rG}_{(j)})$
by \Cref{cor:decoupledboundsdg}.
Then, using \Cref{lm:dmax} and \Cref{claim:ineq} we obtain
$$\dmax(g) \geq \dmax\Big(\sum_{j=1}^{\ell} k_j\dim U_j\Big)\geq \sum_{j=1}^{\ell} \dmax(k_j\dim U_j)\geq \sum_{j=1}^{\ell} d(\wt{\rG}_{(j)})=d(\rG).$$

Now, since $\dim U_1\geq 2$,  for the second inequality to be an equality, by \Cref{lm:dmax}  we must have $\ell=1$. By \Cref{cor:decoupledboundsdg}, for the third inequality to be an equality 
we must have $\rG$ of type 
(I$_n$)
and $k=2$, meaning 
$\wt{\rG}(\RR)$
must have exactly one compact factor.

In this case, though, for $\rG$ of type (I$_n$),
we have $2g=\dim V=m n^2$ with $n=\dim U_i\geq 3$
and $m\geq 2$ even (since the (I$_n$) type is not self-dual for $n\geq 3$).
It follows that $g\geq n^2\geq 2n+2$, and so
$\dmax(g)>\dmax(2n)$ by \Cref{lm:dmax}.
Finally, by \Cref{lm:key-estimate}
we have
$d(\rG)\leq \lceil \frac{n}{2}\rceil \cdot \lfloor \frac{n}{2}\rfloor\leq \dmax(2n)< \dmax(g)$.
\end{proof}

Now we prove the above technical claim.

\begin{proof}[Proof of \Cref{claim:ineq}]
Fix a decomposition $V_\CC=\bigoplus_\alpha V_\alpha$
and isomorphisms $V_\alpha\cong V_{\alpha,1}\otimes\dots\otimes V_{\alpha,k}$ for all $\alpha$.

Consider first the $\QQ$-factor $\wt{\rG}_{(1)}$ of $\wt{\rG}$,
so that $\wt{\rG}_{(1),\CC}=\wt{\rG}_{1,\CC}\times\dots\times\wt{\rG}_{k_1,\CC}$.
Up to relabelling the factors of $\wt{\rG}_{(1),\CC}$, we can
assume that $U_1\cong V_{\gamma,1}$ for some $\gamma$.
Let $\vs_1,\dots,\vs_{k_1}\in\Gal(\CC/\QQ)$ such that
$\vs_i(\wt{\rG}_{1,\CC})=\wt{\rG}_{i,\CC}$.
Observe that $\vs_i(V_\gamma)$ is a $\wt{\rG}_{\CC}$-invariant subspace of $V_\CC$ and so it belongs to the isotypic component of $V_{\beta(i)}$
for some $\beta(i)$. It follows that
$U_{(1),i}\coloneqq V_{\beta(i),i}$ has the same dimension as $U_1$.

If the irreducible summands of $V_\CC$ are not symplectic,
then the $\GSp$-dual of $V_\beta$ belongs to a different
isotypic component, say to the isotypic component of $V_{\beta(i)^*}$.
We then define $U_{(1),i}^*\coloneqq V_{\beta(i)^*,i}$.

Put $U_{(1)}\coloneqq U_{(1),1}\oplus\dots\oplus U_{(1),k_1}$
and note that $\dim U_{(1)}=k_1\cdot \dim U_1$.
If the $V_\alpha$ are not symplectic,
put also $U^*_{(1)}\coloneqq U^*_{(1),1}\oplus\dots\oplus U^*_{(1),k_1}$.

For $j=2,\dots,\ell$ define analogously $U_{(j)}$ 
(and $U^*_{(j)}$, if the $V_\alpha$ are not symplectic).\\

Assume first that $V$ is inefficient.
By construction,
the auxiliary representation $V'$ introduced in
\Cref{ssc:aux} contains
$\bigoplus_{j=1}^\ell U_{(j)}$.
It follows
that
\begin{align*}
g & =\frac12\dim V=
\frac{1}{2}\sum_\alpha \dim V_\alpha=
\frac{1}{2}\sum_\alpha\Prod(N_V)
\geq \sum_\alpha \Sum(N_V)=\\
&=\dim V'\geq \sum_{j=1}^\ell \dim U_{(j)}=
\sum_{j=1}^\ell k_j\cdot \dim U_j.
\end{align*}

Assume now that the summands $V_\alpha$ are not symplectic.
Then $V'$ contains $\bigoplus_{j=1}^\ell (U_{(j)}\oplus U^*_{(j)})$,
and so
\begin{align*}
g & =\frac12\dim V=
\frac{1}{2}\sum_\alpha\Prod(N_V)
\geq \frac{1}{2}\sum_\alpha \Sum(N_V)=\\
&=\frac{1}{2}\dim V'\geq \frac{1}{2}\sum_{j=1}^\ell \dim (U_{(j)}\oplus U^*_{(j)})=
\sum_{j=1}^\ell k_j\cdot \dim U_j.
\end{align*}

Thus the claim is proven.
\end{proof}

\subsection{Special subvarieties from efficient representations}
Now we have to deal with efficient representations. To that end we begin with the following numerical lemma.
\begin{lm}\label{lm:N}
A finite collection $N$ of integers larger than $1$
is efficient if and only if it belongs to the following list:
   \begin{enumerate}
       \item[(i)] $\{b\}$;
       \item[(ii)] $\{2,b\}$;
       \item[(iii)] $\{3,b\}$ with $3\leq b\leq 5$;
    \item[(iv)] $\{2,2,b\}$ with $2\leq b\leq 3$.
   \end{enumerate}
\end{lm}

\begin{proof}
First, if $|N|\ge 4$, then $N$ is inefficient.
Indeed, for $N=\{2,\dots,2\}$ we have $\Prod(N)=2^{|N|}\geq 4|N|= 2\cdot\Sum(N)$, and increasing each number in $N$ by one
results in increasing $2\cdot\Sum$ by $2$ and $\Prod$ by at least $2^{|N|-1}$.

Next, if $|N|=3$, then $N=\{2,2,2\}$ is efficient  since $\Prod(N)=8<12=2\cdot\Sum(N)$.
    Now, increasing a number in $N$ by one results in increasing 
$2\cdot\Sum$ by $2$ and $\Prod$ by at least $4$; so, in order to find efficient sets of three elements, we can only do it at most once. This shows that $\{2,2,2\}, \{2,2,3\}$  are the only efficient sets of 3 elements, since $\{2,3,3\}, \{2,2,4\}$ are inefficient by inspection.

Next, suppose that $N=\{a,b\}$, with $a\le b$. Then $N$ is efficient if and only if $ab < 2a+2b$, which is equivalent to $(a-2)(b-2)< 4$.  So either $2=a\le b$, or $3=a\le b<6$.
    
Finally, if $|N|=1$, then it is certainly efficient.
\end{proof}

Now we show that efficient representations are never associated to compact special subvarieties of 
$\cA_{g}$
of maximal dimension.

\begin{lm}\label{lm:non-decouplable}
Let $(\rG,\dX)\hra \SD_g$ is an indecomposable embedding of Shimura data, and assume that the associated 
$\GSp$-irreducible representation $V$ is not decoupled.
If $N$ is efficient, then
the pair $(d,g)$ associated to the representation $V$ of $\rG$ satisfies $d<\dmax(g)$.
\end{lm}

In the proof, we will often use the following observation.

\begin{rem}\label{rem:product}
Let $U_i$ be a complex irreducible linear representation
    of the complex algebraic group $G'_i$ for $i=1,2$.
Then $U_1\otimes_\CC U_2$ is a self-dual $(G'_1\times G'_2)$-representation if and only if $U_i$ is a self-dual $G'_i$-representation for $i=1,2$.
Moreover, $U_1\otimes_\CC U_2$ is symplectic if and only if exactly one of $\{U_1,U_2\}$ is symplectic, and the other is orthogonal.
\end{rem}

\begin{proof}[Proof of \Cref{lm:non-decouplable}]
Since $N=\{b\}$ corresponds to a decoupled representation,
we proceed by separately analyzing each efficient case
with $N\neq \{b\}$.\\

{\it{Cases  $N=\{2,2,2\}$ and $N=\{2,2\}$.}}

This can only occur if all factors $\wt{\rG}_i$ are of type A$_1$.
Hence, the dimension $d=d(\rG)$ of the Shimura datum is exactly the number of non-compact factors in
$\wt{\rG}(\RR)$.
On the other hand, each linear summand $V_\alpha$ is nontrivially acted on
by at most one non-compact factor of 
$\wt{\rG}(\RR)$
by \Cref{lm:10.7}.
It follows that there must be at least $d$ such linear summands, each one of dimension $2^{|N|}$.
Hence, $\dim V\geq 2^{|N|}d$ and so $g\geq 2^{|N|-1}d\geq d+1$.

We will show now by contradiction that the case $d=g-1$ does not occur, and so this case is negligible.

Indeed, if $d=g-1$, we must have $N=\{2,2\}$ and $(d,g)=(1,2)$. So $\wt{\rG}_\RR$ is a product of the unit quaternions $\mathrm{Nm}_1(\HH_\RR)$ and $\SL_{2,\RR}$. However the only nontrivial 4-dimensional real representation of $\mathrm{Nm}_1(\HH_\RR)$ is on $\HH_\RR$ by left-multiplication, and its centralizer is isomorphic to itself. Thus $\SL_{2,\RR}$ must act trivially, which is a contradiction.\\

{\it{Case $N=\{2,2,3\}$ or $N=\{3,b\}$ with $3\leq b\leq 5$ .}}

In this case, we have $\dim V_{\alpha,i}=3$ for some $\alpha,i$.
Inspecting \Cref{table:milneslist}, we see that
$\rG_i$ must be isogenous to $\SU_{1,2}$
and $V_{\alpha,i}$ must be the standard representation or its dual.
In any case,
$V_{\alpha,i}$ is not self-dual, and so $V_{\alpha}$ is not symplectic
by \Cref{rem:product}. 
Thus $d<\dmax(g)$ by \Cref{lm:inefficient}.\\

{\it{Case $N=\{2,b\}$ with $b\ge 3$}.}

In this case $\wt{\rG}=\wt{\rG}_{(1)}\times \wt{\rG}_{(2)}$, where $\wt{\rG}_{(1)}$ is $\QQ$-simple of type A$_1$ and $\wt{\rG}_{(2)}$ is $\QQ$-simple of type different from A$_1$.

If $V_{\alpha}$ is  not symplectic, we are done by \Cref{lm:inefficient}. 

So now we assume that $V_\alpha$ is symplectic. In this case, since the factors of $V_\alpha$ of dimension $2$
are symplectic,
by \Cref{rem:product}
the factors of $V_\alpha$ of dimension $b$
must be orthogonal.

We denote by $\ell_1$ the number of non-compact factors of 
$\wt{\rG}_{(1)}(\RR)$
and by $\ell_2$ the number of factors in $\wt{\rG}_{(2),\RR}$. Then $\wt{\rG}_\RR$ has at least $\ell_1+\ell_2\geq 2$ factors, so we have $g\geq b(\ell_1+\ell_2)$ and $d=\ell_1+\sum_{i=1}^{\ell_2}d(\wt{\rG}_{(2),i})$, where we 
recall that
$d(\wt{\rG}_{(2),i})=0$ if
$\wt{\rG}_{(2),i}(\RR)$ 
is compact.

We claim that $d(\wt{\rG}_{(2),i})\leq \dmax(b)$ for every non-compact 
$\wt{\rG}_{(2),i}(\RR)$.
As a consequence, by 
\Cref{lm:dmax} and the fact that
$b\geq 3$, we obtain $d\leq \ell_1+\ell_2 \cdot\dmax(b)<
\dmax(b\ell_1+b\ell_2)=\dmax(g)$, thus concluding the proof of the lemma.

To prove the claim note that, by 
\Cref{table:milneslist},
we always have $d(\wt{\rG}_{(2),i})<b$, except 
in the cases of $\wt{\rG}_{(2),i}$ isogenous to $\SO^*_{2r}$ with
$r\geq 5$.

Finally, we separately analize the case of $\wt{\rG}_{(2)}$
of type (II)$_r$.
We get that $d\leq \ell_1+\frac{r(r-1)}{2}\cdot \ell_2$ and $g\geq 2r(\ell_1+\ell_2)$. If $(d,g)$ is not negligible, we obtain
\begin{align*}
    \dmax(g) & \geq \frac{1}{4}\Big(r(\ell_1+\ell_2)\Big)^2-\frac14\geq \frac12\cdot r^2\cdot (\ell_1+\ell_2)-\frac14=\\
    & =d+\frac{r\ell_2+(r^2-2)\ell_1-\frac12}{2}>d+1
\end{align*}
as desired.
\end{proof}

\subsection{Compact special subvarieties of maximal dimension}

Now we can prove our second main result.

\begin{proof}[Proof of \Cref{thm:dmcAg}]
By \Cref{lm:inefficient} and \Cref{lm:non-decouplable} 
a compact subvariety of 
$\cA_{g}$
of dimension $\dmax(g)$ must be
either a Hodge-generic subvariety, or a compact special subvariety associated to
a decoupled representation, or a product of the two types.

The conclusion is then a consequence of \Cref{thm:dmcAg2}
and of \Cref{rem:unique}.
\end{proof}

The proofs of 
\Cref{lm:key-estimate} and \Cref{prop:estimate}
and \Cref{rem:unique} 
in fact yield more information than \Cref{thm:dmcAg}, allowing us to describe all maximal-dimensional compact subvarieties of
$\cA_{g}$,
not just determining their dimensions.
We collect such information in the following statement.

\begin{thm}\label{thm:maxvar}
All maximal-dimensional compact subvarieties of 
$\cA_{g}$,
in each genus $g$, are described as follows:
\begin{itemize}
\item[(o)]
for $g=1$, a maximal-dimensional compact subvariety of 
$\cA_{1}$
is a point;
    \item[(i)]
    for $g=2$, the maximal-dimensional compact subvarieties of 
    $\cA_{2}$
    are either Hodge-generic curves (for example, components of very general complete intersections), or Shimura curves (see \Cref{ssc:SL2});
    \item[(ii)]
    for $3\le g\leq 15$, all maximal-dimensional compact subvarieties of 
    $\cA_{g}$
    must be Hodge-generic: for example, (components of) very general complete intersections;
    \item[(iii)]
    for $g\geq 16$ even, all maximal-dimensional compact subvarieties of 
    $\cA_{g}$
    are 
    compact special subvarieties corresponding to the (I)$_{g/2}$ case of \Cref{table:Satakelist}
    of the type constructed in 
    \Cref{sec:best} (see also \Cref{prop:estimate});
    \item[(iv)]
    for $g\geq 19$ odd, all maximal-dimensional compact subvarieties of 
    $\cA_{g}$
    are products of a point in
    $\cA_{1}$
    and a special subvariety of 
    $\cA_{g-1}$
    of maximal dimension of the type discussed in (iii), or
    (a component of) a Hecke translate of such product;
    \item[(v)]
    for $g=17$, the maximal-dimensional compact subvarieties of 
    $\cA_{17}$
    are either Hodge-generic (for example, components of very general complete intersections), or the product of a point in 
    $\cA_{1}$
    with a compact 16-dimensional special subvariety of 
    $\cA_{16}$
    of the type discussed in (iii),
    or
    (a component of) a Hecke translate of such product.
\end{itemize}
\end{thm}

\section{The indecomposable locus and the locus of Jacobians}\label{sec:jacobians}

In this section we discuss some consequences of our results, and some open problems concerning the locus of indecomposable ppav 
$\Agind$
and the locus of Jacobians, in particular proving \Cref{cor:dmcMgct}.

We will shortly refer to the image of
the natural morphism 
$\cA_{g'}\times\cA_{g-g'}\ra\cA_{g}$
as to the {\it{locus $\cA_{g'}\times\cA_{g-g'}$ in $\cA_{g}$}}, and similarly for the image
of $\cA^*_{g'}\times\cA^*_{g-g'}\ra\cA^*_{g}$.
We will adopt the same convention for 
similar decomposable loci in 
$\Mgct$,
such as
the image
of $\Mgct[g',1]\times\Mgct[g-g',1]\ra \Mgct[g]$.

\subsection{The locus 
$\Agind$
of indecomposable ppav}\label{rem:Agind}
As mentioned in \Cref{intro-Mg}, it is also very natural to investigate
the moduli space $\Agind$ of indecomposable ppav of dimension $g$, see \cite[\S 1]{kesa} for a further discussion of the motivation.

Thinking of the Satake compactification as a compactification of 
$\Agind\subsetneq\cA_{g}^*$,
we see that the boundary
$\cA_{g}^*\setminus\Agind$ has one irreducible component
of maximal dimension, namely $\cA_{1}^*\times\cA_{g-1}^*$, and it has codimension $g-1$.
Thus $g-2\le \dmcg(\Agind)$, 
while \Cref{thm:dmcgAg} of course implies the following.

\begin{cor}
For $g\geq 2$ the following holds:
\[
g-2\leq \dmcg(
\Agind
)\le g-1.
\]
\end{cor}
It would be very interesting to know which value it in fact is. As we will now see, 
$\dmc(\Agind)=\dmcg(\Agind)=g-2$ 
for $g=2,3,4$, and it is natural to wonder if this is the case for all~$g$.

Indeed, the cases of $g=2,3$ are classical: 
$\Agind[2]\cong\cM_{2}$
is affine,
while $\Agind[3]\cong\cM_{3}$
does not contain a compact surface, for example by \cite{diaz}. 

In general, recall that $\lambda_i$ denotes the $i$'th Chern class of the Hodge rank~$g$ vector bundle on
$\cA_{g}$,
and that the tautological ring
$R^*(\cA_{g})\subseteq CH^*_\QQ(\cA_{g})$
is the subring generated by the classes~$\lambda_i$. 

For $g=4$, from the results of \cite{huto2,hutoTaibi} it follows that the classes of all algebraic subvarieties of 
$\cA_{4}$ lie in $R^*(\cA_{4})$,
and in fact the class of 
$\cA_{1}\times\cA_{3}$
is a nonzero multiple of the class $\lambda_3$. Thus the main result of \cite{kesa} shows that any compact subvariety~$Z$ of 
$\cA_{4}\setminus(\cA_{1}\times\cA_{3})$
must satisfy $\dim Z\le 3\cdot 2/2-1=2$,  since $\lambda_3|_Z=0$. This implies 
$\dmc(\Agind[4])=2$.

However, the situation in higher genus remains mysterious. By \cite{vdgcycles}, for any~$g$ the homology class of $[E]\times
\cA_{g-1}
$ is a multiple of the top Hodge class $\lambda_g$, which in fact vanishes on
$\cA_{g}$.
As detailed in \cite{cmop}, it is possible to define a projection map 
$CH^*_\QQ(\cA_{g})\to R^*(\cA_{g})$.
Then in \cite{coptalk} Canning, Oprea, Pandharipande show that, in general,
the projection of the class of 
$\cA_{1}\times\cA_{g-1}$ to the tautological ring of $\cA_{g}$
is a nonzero multiple of~$\lambda_{g-1}$, but also that for $g=6$ the class of 
$\cA_{1}\times\cA_{5}$
does not lie in $R^*(\cA_{6})$.
In \cite{iribarlopez} Iribar Lopez shows that such a locus of products is not tautological for $g=12$, or even $g\ge 16$, either. This makes the above argument for $g=4$ fail in higher genus.

\subsection{The locus of Jacobians, and the moduli space of curves}
In order to discuss 
$\dmc(\Mgct)$,
we begin by recalling the construction of a compact subvariety contained in its boundary 
$\partial\Mgct=\Mgct\setminus\cM_{g}$,
already pointed out by the first author for~\cite{krcomplete}.

\begin{lm}\label{lm:compact-Mgct}
    For $g\geq 2$ there exists a compact subvariety of 
    $\partial\Mgct[g]$ 
    of dimension $\lfloor 3g/2\rfloor -2$, whose image in 
    $\cA_{g}$
    under $\cJ$ has dimension $\lfloor g/2\rfloor$.
Moreover there exists a compact subvariety of 
$\cJ(\Mgct[g])$
of dimension $\lfloor 2g/3\rfloor$.
\end{lm}
\begin{proof}
We first note that 
$\cJ:\Mgct[2]\xrightarrow{\sim}\cA_{2}$
contains a compact curve, while 
the Torelli map 
$\cJ:\Mgct[3]\ra\cA_{3}$
is proper and surjective, 
and thus 
$\Mgct[3]$
contains a compact surface that is a preimage under~$\cJ$ of a compact surface in 
$\cA_{3}$.

Since the forgetful map 
$\Mgct[g,n]\to\Mgct$
is proper (the fiber over $[C]$ being a blowup of the $n$'th Cartesian power $C^{\times n}$), taking preimages of a compact curve in 
$\Mgct[2]$
and of a compact surface in 
$\Mgct[3]$,
respectively, gives a compact surface in 
$\Mgct[2,1]$,
a compact threefold in 
$\Mgct[2,2]$,
and a compact threefold in
$\Mgct[3,1]$.

Now, for even~$g=2n\geq 4$, inside the boundary stratum
$$
\Mgct[2,1]\times\Mgct[2,2]\times\dots \times\Mgct[2,2]\times\Mgct[2,1] 
$$
of $\Mgct$,
we consider the product of two compact surfaces in 
$\Mgct[2,1]$,
and $n-2$ compact threefolds in 
$\Mgct[2,2]$,
giving altogether a product variety of dimension $2\cdot 2+3\cdot (n-2)=\tfrac{3g}{2}-2$. The image of this subvariety under $\cJ$ is the product of $n$ copies of the $\cJ$ image in 
$\cA_{2}$ of the compact curve in $\Mgct[2]$,
and so has dimension~$\frac{g}{2}$.

For $g=2n+1\geq 5$ odd, we do the same except taking the last factor to be a compact threefold in
$\Mgct[3,1]$.

As for the last claim, we argue in a very similar way, the cases $g=2,3$ having already been treated above.
If $g=3n+1\geq 4$, then the locus
$\Mgct[3,1]\times (\Mgct[3,2])^{n-1}\times\Mgct[1,1]$
inside $\Mgct$ is mapped by
$\cJ$ onto $(\cA_{3})^n\times\cA_{1}$,
which
contains a compact subvariety of dimension $2n$.
If $g=3n+2\geq 5$, then the locus
$\Mgct[2,1]\times\Mgct[3,1]\times(\Mgct[3,2])^{n-1}$
of $\Mgct[g]$ maps onto $\cA_{2}\times(\cA_{3})^n$,
which
contains a compact subvariety of dimension $1+2n$.
If $g=3n\geq 6$, then via $\cJ$
the locus 
$(\Mgct[3,1])^2\times (\Mgct[3,2])^{n-2}$
of $\Mgct[g]$ maps onto $(\cA_{3})^n$,
which contains
a compact subvariety of dimension $2n$.
\end{proof}

We note that all subvarieties constructed 
to prove the last claim of \Cref{lm:compact-Mgct} are compact subvarieties of 
$\Mgct$
of dimension $4g/3$ plus a constant, mapping to loci in 
$J(\Mgct)$
of dimension $2g/3$ plus a constant.\\

We now prove the results on compact subvarieties of
$\Mgct$ and $\cJ(\Mgct)$.

\begin{proof}[Proof of \Cref{cor:dmcMgct}]
As already mentioned in the introduction, the upper bound $\dmc(\cJ(\Mgct))\le g-1$ for $g\le 15$ simply follows from the inclusion 
$\cJ(\Mgct)\subseteq\cA_{g}$
and the $g\le 15$ case of \Cref{thm:dmcAg}, giving 
$\dmc(\cA_{g})=\dmcg(\cA_{g})=g-1$
for this range of~$g$. 

To determine 
$\dmc(\Mgct)$
for $2\le g\le 23$, 
we first observe that 
$\dmc(\Mgct)
\geq \lfloor \frac{3g}{2}\rfloor-2$ by \Cref{lm:compact-Mgct}.
Thus we only need to prove
that $\dmc(\Mgct)
\leq\lfloor\tfrac{3g}2\rfloor-2$ for $2\le g\le 23$.

Since 
$\Mgct[2]$ and $\Mgct[3]$
contain a compact curve and a compact surface, respectively, the statement of the corollary is true for $g=2$ and $g=3$. 

For $4\le g\le 23$, we proceed by induction, assuming the result for all $g'<g$. Indeed, by \Cref{thm:dmcAg}, 
exactly for $g$ in this range we have $\dmc(\cA_{g})<\lfloor\tfrac{3g}2\rfloor-2$ 
and so an irreducible compact subvariety $Z\subsetneq
\Mgct$
that satisfies $Z\cap
\cM_{g}
\neq\emptyset$ (and thus maps generically 1-to-1 to its image in
$\cA_{g}$
under $\cJ$) has dimension strictly
smaller than $\lfloor\tfrac{3g}2\rfloor-2$.

Thus it is enough to deal with irreducible compact $Z\subsetneq
\partial\Mgct$. Such~$Z$ must then be contained in some irreducible component 
$\Mgct[g',1]\times \Mgct[g-g',1]$ of the boundary $\partial\Mgct=\Mgct\setminus\cM_{g}$.
Since the forgetful map 
$\Mgct[k,1]\to\Mgct[k]$
has compact curve fibers, we obtain the equality 
$\dmc(\Mgct[k,1])=1+\dmc(\Mgct[k])$.
It then follows from the inductive assumption that
$$
 \dim Z\le \left(\lfloor\tfrac{3g'}2\rfloor-2\right)+1+\left(\lfloor\tfrac{3(g-g')}2\rfloor-2\right)+1\le \lfloor\tfrac{3g}2\rfloor-2\,,
$$
which proves the upper bound $\dmc(\Mgct)\le\lfloor\tfrac{3g}2\rfloor-2$ for $g\le 23$. 

The last claim $\dmc(\cJ(\Mgct))\geq \lfloor 2g/3\rfloor$
for $g\geq 2$ follows from \Cref{lm:compact-Mgct}.
\end{proof}

The compact subvariety constructed above 
is contained in the boundary 
$\cJ(\Mgct\setminus\cM_{g})$,
and we do not know any construction of a higher-dimensional compact $Z\subsetneq \cJ(\Mgct)$ that intersects both
$\cJ(\cM_{g})$ and~$\cJ(\partial\Mgct)$.

\begin{rem} 
We see no reason to expect the bound $\dmc(\cJ(\Mgct))\leq g-1$ for $g\leq 15$ to be sharp.
It would be very interesting to improve it, e.g.~by bounding the dimensions of the intersections of compact Shimura varieties with $\cJ(\Mgct)$, extending the spirit of the Coleman-Oort conjecture
\cite[Section 5]{oort}.
\end{rem}

\smallskip
For easier future reference, and to summarize the current state of the art, in \Cref{table:dimMg} we give the results of \Cref{cor:dmcMgct} for small genera, and summarize all the prior knowledge on compact subvarieties of 
$\cM_{g}$ and $\Mgct$. 

\begin{table}[H]
$
\begin{array}{|r||rrrr||rrrr||rr||r|}
\hline
  g &3&4&5&6&15&16&17&18&23&24&100 \\
  \hline
  \hbox{$\dmcg(\Mgct)\ge$}&2&2&2&2&2&2&2&2&2&2&2\\
  \hbox{$\dmc(\Mgct)=$}&2&4&5&7&20&22&23&25&32&\ge 34&\ge 148\\
\hbox{$\dmc(\cJ(\Mgct))\le$}&2&3&4&5&14&16&16&20&30&36&196\\
\hbox{$\dmc(\cJ(\Mgct))\ge$}
&2&2&3&4&10&10&11&12&15&16&66\\
  \hline
  \hbox{$\dmcg(\cM_{g})\ge$}&1&1&1&1&1&1&1&1&1&1&1\\
  \hbox{\em covers: 
  $\dmc(\cM_{g})\ge$}&1&1&1&1&2&3&3&3&3&3&5\\
  \hbox{\em Diaz: $\dmc(\cM_{g})\le$}&1&2&3&4&13&14&15&16&21&22&98\\
  \hline
\end{array}
$

\medskip
\caption{Known results for maximal dimensions of compact subvarieties of $\Mgct$ and $\cM_{g}$}
\label{table:dimMg}
\end{table}

Here we recall that Keel and Sadun proved $\dmc(\Mgct)\le 2g-4$ for any $g\ge 3$. The lower bounds on 
$\dmcg(\cM_{g})$ and $\dmcg(\Mgct)$
simply follow from considering the Satake compactification 
$\cM_{g}^*$, which is the closure of $\cJ(\cM_{g})$ in~$\cA_{g}^*$, so that $\codim (\cM_{g}^*\setminus \cJ(\cM_{g}))=2$ and $\codim(\cM_{g}^*\setminus \cJ(\Mgct))=3$.

The lower bounds for $\dmc(\cM_{g})$ are obtained by covering constructions starting 
either from a compact curve in 
$\cM_{3}$
or
from  a one-dimensional compact family of pairs of distinct points on a fixed curve of genus $2$.
The best known results are due to Zaal, following ideas
by Gonz\'alez-D\'\i ez and Harvey in \cite[Section 4]{gdha}, and are 
described in detail in Zaal's thesis, where in particular it is shown in \cite[Thm.~2.3]{zaalthesis} that a compact $d$-fold exists in 
$\cM_{g}$
for any $g\ge 2^{d+1}$. Finally, the best known upper bound for $\dmc(\cM_{g})$ is the famous 40 year old theorem of Diaz \cite{diaz} (which by now has multiple proofs): 
$\dmc(\cM_{g})\le g-2$ for any $g\ge 2$.

\providecommand{\bysame}{\leavevmode\hbox to3em{\hrulefill}\thinspace}
\providecommand{\MR}{\relax\ifhmode\unskip\space\fi MR }
\providecommand{\MRhref}[2]{\href{http://www.ams.org/mathscinet-getitem?mr=#1}{#2}
}
\providecommand{\href}[2]{#2}


\begin{thebibliography}{CMOP24}

\bibitem[BT65]{borel-tits}
Armand Borel and Jacques Tits, \emph{Groupes reductifs}, Publ. {M}ath., {I}nst.
  {H}autes {\'E}tud. {S}ci. \textbf{27} (1965), 659--755.

\bibitem[CMOP24]{cmop}
Samir Canning, Samouil Molcho, Dragos Oprea, and Rahul Pandharipande,
  \emph{Tautological projection for cycles on the moduli space of abelian
  varieties}, 2024, {to appear on \it{Algebraic geometry}}, preprint available at
  \url{https://arxiv.org/abs/2401.15768}.

\bibitem[COP24]{coptalk}
Samir Canning, Dragos Oprea, and Rahul Pandharipande, \emph{Tautological and
  non-tautological cycles on the moduli space of abelian varieties}, 
  Invent. {M}ath. \textbf{242} (2025), 659--723.



\bibitem[Del71]{deligne-travaux}
Pierre Deligne, \emph{Travaux de {S}himura}, S{\'e}minaire {B}ourbaki
  \textbf{13} (1970-1971), 123--165.

\bibitem[Del79]{deligne}
\bysame, \emph{Vari{\'e}t{\'e}s de {S}himura: interpr{\'e}tation modulaire, et
  techniques de construction de mod{`e}les canoniques}, Automorphic forms,
  representations and ` L-functions (Proc. Sympos. Pure Math., Oregon State
  Univ., Corvallis, Ore., 1977), Part 2, Proc. Sympos. Pure Math., XXXIII,
  Amer. Math. Soc., Providence, R.I., 1979, pp.~247--289.

\bibitem[Dia84]{diaz}
Steven Diaz, \emph{A bound on the dimensions of complete subvarieties of
  {${\mathcal M}_{g}$}}, Duke Math. J. \textbf{51} (1984), no.~2, pp.~405--408.

\bibitem[GDH91]{gdha}
Gabino Gonz\'{a}lez-D\'{\i}ez and William Harvey, \emph{On complete curves in
  moduli space. {I}}, Math. Proc. Cambridge Philos. Soc. \textbf{110} (1991),
  no.~3, 461--466.

\bibitem[HM95]{hashimoto1995}
Ki-Iciro Hashimoto and Naoki Murabayashi,
\emph{Shimura curves as intersections of Humbert surfaces and defining equations of QM-curves of genus two}, Tohoku Math.~Journal, Second Series {\bf{47}} (1995), no.~2, pp.~271--296.


\bibitem[Hel78]{helgason}
Sigurdur Helgason, \emph{Differential geometry, {L}ie groups, and symmetric
  spaces}, Pure and Applied Mathematics, vol.~80, Academic Press, Inc.
  [Harcourt Brace Jovanovich, Publishers], New York-London, 1978.

\bibitem[HT12]{huto2}
Klaus Hulek and Orsola Tommasi, \emph{Cohomology of the second {V}oronoi
  compactification of {${\mathcal A}_4$}}, Doc. Math. \textbf{17} (2012),
  195--244.

\bibitem[HT18]{hutoTaibi}
\bysame, \emph{The topology of {${\mathcal A}_g$} and its compactifications},
  Geometry of moduli, Abel Symp., vol.~14, Springer, Cham, 2018, With an
  appendix by Olivier Ta\"{\i}bi, pp.~135--193.

\bibitem[IL25]{iribarlopez}
Aitor Iribar~L\'opez, \emph{Noether-lefschetz cycles on the moduli space of
  abelian varieties}, 2024, preprint available at
  \url{https://arxiv.org/abs/2411.09910}.

\bibitem[Kri12]{krcomplete}
Igor Krichever, \emph{Real normalized differentials and compact cycles in the
  moduli space of curves}, 2012, preprint available at
  \url{https://arxiv.org/abs/1204.2192}.

\bibitem[KS03]{kesa}
Sean Keel and Lorenzo Sadun, \emph{Oort's conjecture for {${\mathcal
  A}_g\otimes{\mathbb C}$}}, J. Amer. Math. Soc. \textbf{16} (2003), no.~4,
  887--900 (electronic).

\bibitem[KU25]{khur}
Nazim Khelifa and David Urbanik, \emph{Existence and density of typical {H}odge
  loci}, Math. Ann. \textbf{391} (2025), no.~1, 819--842.

\bibitem[Lan17]{lanexample}
Kai-Wen Lan, \emph{An example-based introduction to {S}himura varieties},
  Preprint (2017), available at
  \url{https://www-users.cse.umn.edu/~kwlan/articles/intro-sh-ex.pdf}.

\bibitem[Mal43]{malcev}
Anatoly~I. Maltsev, \emph{Orthogonal and symplectic representations of
  semi-simple {L}ie groups}, C. R. (Doklady) Acad. Sci. URSS (N.S.) \textbf{41}
  (1943), 332--335.

\bibitem[Mil90]{milne1990}
James~S. Milne, \emph{Canonical models of (mixed) {S}himura varieties and
  automorphic vector bundles}, Automorphic forms, {S}himura varieties, and
  {$L$}-functions, {V}ol.\ {I} ({A}nn {A}rbor, {MI}, 1988), Perspect. Math.,
  vol.~10, Academic Press, Boston, MA, 1990, pp.~283--414.

\bibitem[Mil13]{milne2011shimura}
\bysame, \emph{Shimura varieties and moduli}, Handbook of moduli. {V}ol. {II},
  Adv. Lect. Math. (ALM), vol.~25, Int. Press, Somerville, MA, 2013,
  pp.~467--548.

\bibitem[Mil17a]{milne:ag}
\bysame, \emph{Algebraic groups}, Cambridge Studies in Advanced Mathematics,
  vol. 170, Cambridge University Press, Cambridge, 2017, The theory of group
  schemes of finite type over a field.

\bibitem[Mil17b]{milne2005}
\bysame, \emph{Introduction to {S}himura varieties}, available at
  {\url{https://www.jmilne.org/math/xnotes/svi.pdf}}, revised notes based on an
  original varsion published in 2005 in {\it{{H}armonic analysis, the trace
  formula, and {S}himura varieties: proceedings of the {C}lay {M}athematics
  {I}nstitute 2003 {S}ummer {S}chool}}, {C}lay {M}athematics {P}roceedings,
  vol.~4, pp.~265--378, {A}mer.~{M}ath.~{S}oc., {P}rovidence, {RI}, September
  2017.

\bibitem[MO13]{moonen-oort}
Ben Moonen and Frans Oort, \emph{The {T}orelli locus and special subvarieties},
  Handbook of moduli. {V}ol. {II}, Adv. Lect. Math. (ALM), vol.~25, Int. Press,
  Somerville, MA, 2013, pp.~549--594.

\bibitem[Moo98]{moonen1998linearityI}
Ben Moonen, \emph{Linearity properties of {S}himura varieties {I}}, Journal of
  {A}lgebraic {G}eometry \textbf{7} (1998), no.~3, 539--568.

\bibitem[MPT19]{MPT}
Ngaiming Mok, Jonathan Pila, and Jacob Tsimerman, \emph{Ax-{S}chanuel for
  {S}himura varieties}, Annals of Mathematics \textbf{189} (2019), no.~3,
  945--978.

\bibitem[Oor97]{oort}
Frans Oort, \emph{Canonical liftings and dense sets of {CM}-points}, Arithmetic
  geometry (Cortona, 1994) \textbf{37} (1997), 228--234.

\bibitem[PRR94]{platonovrapinchuk}
Vlamidir Platonov, Andrei Rapinchuk, and Igor Rapinchuk, \emph{Algebraic groups
  and number theory}, Pure and Applied Mathematics, vol. 139, Academic Press,
  Inc., Boston, MA, 1994, Translated from the 1991 Russian original by Rachel
  Rowen.

\bibitem[PT14]{pila-tsimerman}
Jonathan Pila and Jacob Tsimerman, \emph{Ax-{L}indemann for $\mathcal{A}_g$},
  Annals of Mathematics \textbf{179} (2014), no.~2, 659--681.

\bibitem[Sat56]{satake-compact}
Ichir\^o Satake, \emph{On the compactification of the {S}iegel space}, J.
  Indian Math. Soc. (N.S.) \textbf{20} (1956), 259--281.

\bibitem[Sat65]{satake:imb}
\bysame, \emph{Holomorphic imbeddings of symmetric domains into a {S}iegel
  space}, Amer. J. Math. \textbf{87} (1965), 425--461.

\bibitem[Sat67]{satake:analytic}
\bysame, \emph{Symplectic representations of algebraic groups satisfying a
  certain analyticity condition}, Acta Math. \textbf{117} (1967), 215--279.

\bibitem[Sat80]{satake:book}
\bysame, \emph{Algebraic structures of symmetric domains}, Kan\^o{} Memorial
  Lectures, vol.~4, Iwanami Shoten, Tokyo; Princeton University Press,
  Princeton, NJ, 1980.

\bibitem[Sch73]{schmid}
Wilfried Schmid, \emph{Variation of {H}odge structure: the singularities of the
  period mapping}, Inventiones mathematicae \textbf{22} (1973), no.~3-4,
  211--319.

\bibitem[Sch85]{scharlau}
Winfried Scharlau, \emph{Quadratic and {H}ermitian forms}, Grundlehren der
  mathematischen Wissenschaften [Fundamental Principles of Mathematical
  Sciences], vol. 270, Springer-Verlag, Berlin, 1985.

\bibitem[Sim92]{simpson1992higgs}
Carlos~T. Simpson, \emph{Higgs bundles and local systems}, Publications
  Math{\'e}matiques de l'IH{\'E}S \textbf{75} (1992), 5--95.

\bibitem[vdG99]{vdgcycles}
Gerard van~der Geer, \emph{Cycles on the moduli space of abelian varieties},
  Aspects Math., E33, pp.~65--89, Vieweg, Braunschweig, 1999.

\bibitem[Zaa05]{zaalthesis}
Christiaan Zaal, \emph{Complete subvarieties of moduli spaces of algebraic
  curves}, PhD Dissertation, University of Amsterdam, 2005. Available at
  \url{https://pure.uva.nl/ws/files/3915897/36235_Thesis.pdf}.

\end{thebibliography}
\end{document}